%% file: main.tex
\documentclass[12pt,a4paper]{amsart}

\usepackage[english]{babel}
\usepackage[utf8]{inputenc} 
\usepackage[T1]{fontenc}

\usepackage[a4paper,top=3cm,bottom=2cm,left=3cm,right=3cm,marginparwidth=1.75cm]{geometry}

\usepackage{amsmath, amsthm, amsfonts,amssymb,mathrsfs,centernot,tikz,stmaryrd,calligra}
\usetikzlibrary{calc}
\usepackage{graphicx}
\usepackage{xcolor}

\definecolor{NUcolor}{RGB}{136,90,64}

\let\mathcup\cup
\usepackage{musixtex}
\let\musixcup\cup

\DeclareRobustCommand*{\cup}{%
  \ifmmode
    \expandafter\mathcup
  \else
    \expandafter\musixcup
  \fi
}

\usepackage[
  colorlinks=true,       
  breaklinks=true,        
  hyperfootnotes=true,    
  linkcolor=NUcolor,
  citecolor=NUcolor,
  urlcolor=NUcolor,
]{hyperref}

\newtheorem{theorem}{Theorem}[section]
\newtheorem{lemma}[theorem]{Lemma}
\newtheorem{conjecture}[theorem]{Conjecture}
\newtheorem{proposition}[theorem]{Proposition}
\newtheorem{corollary}[theorem]{Corollary}

\theoremstyle{definition}
\newtheorem{definition}[theorem]{Definition}
\newtheorem{remark}[theorem]{Remark}

\numberwithin{equation}{section}

\def\Xint#1{\mathchoice 
{\XXint\displaystyle\textstyle{#1}}%
{\XXint\textstyle\scriptstyle{#1}}%
{\XXint\scriptstyle\scriptscriptstyle{#1}}%
{\XXint\scriptscriptstyle\scriptscriptstyle{#1}}%
\!\int} 
\def\XXint#1#2#3{{\setbox0=\hbox{$#1{#2#3}{\int}$} 
\vcenter{\hbox{$#2#3$}}\kern-.5\wd0}}

\def\avgint{\Xint-}

\DeclareMathOperator*{\esssup}{ess\,sup}
\DeclareMathOperator*{\essinf}{ess\,inf}

\newcommand{\vertiii}[1]{{\left\vert\kern-0.25ex\left\vert\kern-0.25ex\left\vert #1 
    \right\vert\kern-0.25ex\right\vert\kern-0.25ex\right\vert}}

\newcommand*{\perdef}{\mathrel{\mathop:}=}
\newcommand*{\defper}{=\mathrel{\mathop:}}

\makeatletter
\renewcommand{\tocsection}[3]{%
  \indentlabel{\@ifnotempty{#2}{\bfseries\ignorespaces#1 #2.\!\quad}}\bfseries#3}
\renewcommand{\tocsubsection}[3]{%
  \indentlabel{\@ifnotempty{#2}{
  #1 #2.\quad}}#3}

\newcommand\@dotsep{4.5}
\def\@tocline#1#2#3#4#5#6#7{\relax
  \ifnum #1>\c@tocdepth 
  \else
    \par \addpenalty\@secpenalty\addvspace{#2}%
    \begingroup \hyphenpenalty\@M
    \@ifempty{#4}{%
      \@tempdima\csname r@tocindent\number#1\endcsname\relax
    }{%
      \@tempdima#4\relax
    }%
    \parindent\z@ \leftskip#3\relax \advance\leftskip\@tempdima\relax
    \rightskip\@pnumwidth plus1em \parfillskip-\@pnumwidth
    #5\leavevmode\hskip-\@tempdima{#6}\nobreak
    \leaders\hbox{$\m@th\mkern \@dotsep mu\hbox{.}\mkern \@dotsep mu$}\hfill
    \nobreak
    \hbox to\@pnumwidth{\@tocpagenum{\ifnum#1=1\bfseries\fi#7}}\par
    \nobreak
    \endgroup
  \fi}
\AtBeginDocument{%
\expandafter\renewcommand\csname r@tocindent0\endcsname{0pt}
}
\def\l@subsection{\@tocline{2}{0pt}{2.5pc}{5pc}{}}
\makeatother

\title[Extrapolation via Sawyer-type inequalities]{Extrapolation via Sawyer-type inequalities}

\author[Eduard Roure-Perdices]{Eduard Roure-Perdices}
\address{Vicerectorat de Recerca, Universitat de Barcelona, 08007 Barcelona, Spain.}
\address{Department of Mathematics, University of the Basque Country UPV/EHU, Sarriena s/n, 48940 Leioa - Bizkaia, Spain.} 

\email{eduardroure@protonmail.ch}

\thanks{The author was partially supported by the MICINNU project PGC2018-094522-B-I00 and a Margarita Salas grant [University of Barcelona / Spanish Ministry of Universities / Spain's Reco\-very, Transformation and Resilience Plan / European Union NextGenerationEU]}  

\subjclass[2020]{42B25, 46E30}

\keywords{Extrapolation, Lorentz spaces, restricted weak-type, weights, Sawyer-type inequalities}

\begin{document}

\begin{abstract}
We present a multi-variable extension of Rubio de Francia's restricted weak-type extrapolation theory that does not involve Rubio de Francia's iteration algorithm; instead, we rely on the following Sawyer-type inequality for the weighted Hardy-Littlewood maximal operator $M_u$:

$$
\left \Vert \frac{M_u (fv)}{v} \right \Vert_{L^{1,\infty}(uv)} \leq C_{u,v} \Vert f \Vert_{L^1(uv)}, \quad u, \, uv \in A_{\infty}.
$$

Our approach can be adapted to recover weak-type $A_{\vec P}$ extrapolation schemes, including an endpoint result that falls outside the classical theory.

Among the applications of our work, we highlight extending outside the Banach range the well-known equivalence between restricted weak-type and weak-type for characteristic functions, and obtaining mixed and restricted weak-type bounds with $A_{p}^{\mathcal R}$ weights for relevant families of multi-variable operators, addressing the lack in the literature of these types of estimates. We also reveal several standalone properties of the class $A_{p}^{\mathcal R}$.

\end{abstract}

\maketitle






\vspace{\baselineskip}
\hspace{9.66cm}
{\sc{Llu\'is Llach}}, 
1968
\vspace{-0.4cm}

\begin{music}
\instrumentnumber{1} 
\nobarnumbers%
\nostartrule 
\setstaffs1{1} 
\setclef1{3} 
\generalmeter{\meterfrac{3}{4}}
\startextract 
\notes\qa{def
}\en\bar
\Notes\ha{g
}\en\notes\isluru0e
\qa{e
}\en\bar
\notes\tslur0e
\qa{ege
}\en\bar
\Notes\ha{h
}\en\notes\isluru1f
\qa{f
}\en\bar
\notes\tslur1f
\qa{fed
}\en\bar
\Notes\ha{O
}\en\notes\qa{^c
}\en\bar
\Notes\ha{e
}\en\notes\qa{^c
}\en\bar
\Notesp\ha{.d
}\en
\setdoubleBAR
\endextract 
\end{music}

\tableofcontents

\section{Introduction}

In the topic of weighted theory, a result that has attracted the attention of many researchers in the field is the so-called \textit{Rubio de Francia's extrapolation theorem} (see \cite{rdfrubio, rufa}), which provides a precious shortcut when trying to prove weighted strong-type bounds. In its simplest form, it says that if a sub-linear operator $T$ satisfies that 
\begin{equation*}
    T:L^{p}(v) \longrightarrow L^p(v),
\end{equation*}
for some $1 \leq p < \infty$, and every Muckenhoupt weight $v$ in $A_p$, then 
\begin{equation*}
    T:L^{q}(w) \longrightarrow L^q(w),
\end{equation*}
for every $1<q < \infty$, and every Muckenhoupt weight $w$ in $A_q$ (see Section~\ref{s2} for definitions). 

Many alternative proofs of this theorem are available in the literature (see \cite{auma,umpbook,duo,gacu}), also tracking the sharp dependence of $\Vert T\Vert_{L^{q}(w) \rightarrow L^q(w)}$ in terms of $[w]_{A_q}$ (see \cite{dgpp}), and off-diagonal results where the domain and target Lebes\-gue spaces differ both in terms of exponents and weights (see \cite{cuma,duoextrapol,harb} for strong-type results, and \cite{extraneu} for weak-type ones). Moreover, it was discovered that the operator $T$ plays no role in the extrapolation process, and one can simply work with families of pairs of measurable functions (see \cite{extraump,CUMP3,duoextrapol}).

Around the beginning of the current millennium, the topic of multi-varia\-ble operators started gathering interest, with the resolution of Calder\'on's conjecture (see \cite{mlacey,mlacey2}) and the development of a systematic treatment of multi-linear Calder\'on-Zygmund operators (see \cite{graftor}), and the first results on multi-variable Rubio de Francia's extrapolation appeared. 

In \cite{gramar}, it was proved that if an $m$-variable operator $T$ satisfies that
\begin{equation*}
    T: L^{p_1}(v_1)\times \dots \times L^{p_m}(v_m) \longrightarrow L^{p}(v_1^{p/p_1}\dots v_m^{p/p_m}),
\end{equation*}
for some exponents $1 \leq p_1,\dots,p_m<\infty$, with $\frac{1}{p}=\frac{1}{p_1}+\dots+\frac{1}{p_m}$, and all weights $v_1\in A_{p_1},\dots,v_m \in A_{p_m}$, then
\begin{equation*}
    T: L^{q_1}(w_1)\times \dots \times L^{q_m}(w_m) \longrightarrow L^{q}(w_1^{q/q_1}\dots w_m^{q/q_m}),
\end{equation*}
for all exponents $1 < q_1,\dots,q_m<\infty$, with $\frac{1}{q}=\frac{1}{q_1}+\dots+\frac{1}{q_m}$, and all weights $w_1\in A_{q_1},\dots,w_m \in A_{q_m}$. 

In \cite{duoextrapol}, the sharp dependence of $\Vert T\Vert_{L^{q_1}(w_1)\times \dots \times L^{q_m}(w_m) \rightarrow L^{q}(w_1^{q/q_1}\dots w_m^{q/q_m})}$ in terms of $([w_i]_{A_{q_i}})_{1\leq i \leq m}$ was established, and analogous multi-variable weak-type extrapolation schemes were studied in \cite{casto}. Once again, the operator $T$ plays no role, and all the results can be stated for $(m+1)$-tuples of measurable functions.

Recently, multi-variable strong-type extrapolation theorems for $A_{\vec P}$ weights have been obtained in \cite{multiap2,multiap1,zoe}, solving in the affirmative a question that had been going around for about a decade, since the publication of \cite{LOPTT}, where such weights were introduced.

Rubio de Francia's extrapolation theory provides a potent set of tools in Harmonic Analysis, but it has a weak spot; namely, it does not allow to produce estimates in the endpoint $q_1=\dots=q_m=1$, which can be easily seen by considering $m$-variable commutators (see \cite{LOPTT}).

In the case of one-variable extrapolation, the works of M. J. Carro, L. Grafakos, and J. Soria (see \cite{cgs}), M. J. Carro and J. Soria (see \cite{cs}), and S. Baena-Miret and M. J. Carro (see \cite{sergi}) give a solution to this problem assuming a slightly stronger extrapolation hypothesis. They showed that if a sub-linear operator $T$ satisfies that
\begin{equation*}
    T:L^{p,1}(v) \longrightarrow L^{p,\infty}(v),
\end{equation*}
for some exponent $1 \leq p < \infty$, and every weight $v$ in $\widehat A_p$, then 
\begin{equation*}
    T:L^{q,\min \left \{1,\frac{q}{p} \right \}}(w) \longrightarrow L^{q,\infty}(w),
\end{equation*}
for every exponent $1 \leq q < \infty$, and every weight $w$ in $\widehat A_q$. Here, for $r\geq 1$, the class $\widehat A_r$ contains all the weights of the form $(Mh)^{1-r}u$, where $h\in L^1_{loc}(\mathbb R^n)$ and $u\in A_1$. If $r=1$, then $\widehat A_1=A_1$, but for $r>1$, $A_r \subsetneq \widehat A_r\subseteq A_r ^{\mathcal R}$. 

In general, the classical strong and weak-type Rubio de Francia's extrapolation theorems rely on three fundamental ingredients: factorization of $A_r$ weights, construction of $A_1$ weights via Rubio de Francia's iteration algorithm (see \cite{bloom,rdfrubio}), and sharp weighted bounds for the Hardy-Littlewood maximal operator $M$. However, in the setting of restricted weak-type Rubio de Francia's extrapolation, many technical difficulties appear. For instance, no factorization result is known for $A_r ^{\mathcal R}$ weights, which justifies the need for the class $\widehat A_r$. Also, within this framework, the Rubio de Francia's iteration algorithm can not be defined and has to be carefully replaced by the Hardy-Littlewood maximal operator $M$ in the construction of weights. Fortunately, we do have sharp weighted restricted weak-type bounds for $M$.

The main purpose of this project is to build upon \cite{sergi,cgs,cs,casto,thesis} and extend to the multi-variable setting the restricted weak-type Rubio de Francia's extrapolation techniques discussed there. 


The first result that we were able to deduce, presented in \cite[Theorem 3.2.1]{thesis}, allows us to extrapolate down to the endpoint $(1,1,\frac{1}{2})$ from a diagonal estimate. Simply put, if a two-variable operator $T$ satisfies that
\begin{equation*}
    T: L^{r,1}(v_1)\times L^{r,1}(v_2) \longrightarrow L^{\frac{r}{2},\infty}(v_1^{1/2}v_2^{1/2}),
\end{equation*}
for some exponent $1 < r <\infty$, and all weights $v_1,v_2\in \widehat A_{r}$, then
\begin{equation*}
    T: L^{1,\frac{1}{r}}(w_1)\times L^{1,\frac{1}{r}}(w_2) \longrightarrow L^{\frac{1}{2},\infty}(w_1^{1/2}w_2^{1/2}),
\end{equation*}
for all weights $w_1,w_2 \in A_1$. The crucial point in its proof is the endpoint bound 
\begin{equation}\label{endpointmandm}
    M^{\otimes}:  L^{1}(w_1)\times L^{1}(w_2) \longrightarrow L^{\frac{1}{2},\infty}(w_1^{1/2}w_2^{1/2}),
\end{equation}
proved in \cite{LOPTT} and refined in \cite[Theorem 2.4.1]{thesis} and \cite[Theorem 3]{prp}. Here, the operator $M^{\otimes}$ is defined for locally integrable functions $f_1$ and $f_2$ by
\begin{equation*}
    M^{\otimes}(f_1,f_2)(x)\perdef Mf_1(x) Mf_2(x), \quad x \in \mathbb R^n,
\end{equation*}
where $M$ is the Hardy-Littlewood maximal operator. 

For simplicity, in \cite[Chapter 3]{thesis} we decided to work on two-variable extrapolation; the extension from two variables to multiple variables is just a matter of notation.

The approach to establishing general downwards extrapolation schemes is now evident: find some auxiliary operator $\mathscr Z$ for which we can prove mixed and restricted weak-type inequalities, and use the extrapolation hypotheses to transfer such bounds to the generic operator $T$. The operator $\mathscr Z$ plays the same role as the Hardy-Littlewood maximal operator plays in the one-variable restricted weak-type extrapolation theory of Rubio de Francia.

As it turns out, sometimes we can take $M^{\otimes}$ to be our auxiliary operator (see \cite[Theorem 3.2.4]{thesis}). Moreover, our preliminary mixed-type inequalities for $M^{\otimes}$ in \cite[Theorem 2.2.10]{thesis} encouraged us to develop the multi-variable mixed-type extrapolation theory of Rubio de Francia, partially presented in \cite[Section 3.3]{thesis} and completed in Section~\ref{s6}.

After a detailed analysis of the proof of \eqref{endpointmandm} in \cite{LOPTT}, we concluded that the complete solution to multi-variable mixed and restricted weak-type extrapolation, along with the corresponding bounds for $M^{\otimes}$, relies on weighted inequalities for operators of the form 
\begin{equation*}
    \mathscr Zf = \frac{Mf}{U}
\end{equation*}
 on Lorentz spaces, being $U$ some nice weight. This discovery forced us into developing our theory of Sawyer-type inequalities for Lorentz spaces, displayed in \cite[Section 2]{thesis} and \cite{prp}. 
 
 One of our most remarkable achievements is the following, which, unexpectedly, keeps popping up in our calculations: if $u\in A_\infty$ and $uv\in A_\infty$, then there exists a well-behaved constant $C_{u,v}$ such that for every measurable function $f$,
\begin{equation}\label{mainsawyer}
\left \Vert \frac{M_u (fv)}{v}\right \Vert_{L^{1,\infty}(uv)} \leq C_{u,v}  \int_{\mathbb R^n} |f(x)| u(x) v(x) dx.
\end{equation}

By exploiting \eqref{mainsawyer}, the Sawyer-type inequality in \cite[Theorem 2]{prp}, and its dual version in Theorem~\ref{thetasawyer}, we manage to produce Theorem~\ref{multiextrapoldowntotal}, the multi-variable mixed and res\-tricted weak-type extrapolation scheme that we were seeking, fulfilling the original goal of our project. In general terms, we have that if an $m$-variable operator $T$ satisfies that
\begin{equation*}
    T:L^{p_1,1}(v_1)\times \dots \times  L^{p_{\ell},1}(v_{\ell}) \times L^{p_{\ell+1}}(v_{\ell+1}) \times \dots \times L^{p_m}(v_m)\longrightarrow L^{p,\infty}(v_1^{p/p_1} \dots v_m^{p/p_m}),
\end{equation*}
for some exponents $1\leq p_1,\dots,p_m < \infty$, with $\frac{1}{p}=\frac{1}{p_1}+\dots+\frac{1}{p_m}$, and all weights $v_i \in \widehat{A}_{p_i,\infty}$, $i=1,\dots,\ell$, and $v_i \in {A}_{p_i}$, $i=\ell+1,\dots,m$, then
\begin{equation*}
    T:\left( \prod_{i=1}^{\ell} L^{q_i,\min  \left \{1,\frac{q_i}{p_i} \right\}}(w_i)\right) \times \left(\prod_{i=\ell+1}^{m} L^{q_i,\min  \{p_i,q_i \}}(w_i)\right)\longrightarrow L^{q,\infty}(w_1^{q/q_1}\dots w_m^{q/q_m}),
\end{equation*}
for all exponents $1\leq q_1,\dots, q_{\ell} < \infty$,  $1<q_{\ell+1},\dots,q_m<\infty$, with $\frac{1}{q}=\frac{1}{q_1}+\dots+\frac{1}{q_m}$, and all weights $w_i \in \widehat{A}_{q_i,\infty}$, $i=1,\dots,\ell$, and $w_i \in {A}_{q_i}$, $i=\ell+1,\dots,m$. Here, for $r\geq 1$, the class $\widehat A_{r,\infty}$ is an extension of $\widehat A_r$ such that $\widehat A_r \subseteq \widehat A_{r,\infty} \subseteq A_r ^{\mathcal R}$.

 In particular, and ignoring some technicalities (see Theorem~\ref{extrapolchar}), we get that if an $m$-variable operator $T$ satisfies that
\begin{equation*}
    T: L^{p_1,1}(v_1)\times \dots \times L^{p_m,1}(v_m) \longrightarrow L^{p,\infty}(v_1^{p/p_1}\dots v_m^{p/p_m}),
\end{equation*}
for some exponents $1 \leq p_1,\dots,p_m<\infty$, with $\frac{1}{p}=\frac{1}{p_1}+\dots+\frac{1}{p_m}$, and all weights $v_1\in \widehat A_{p_1},\dots,v_m \in \widehat A_{p_m}$, then
\begin{equation*}
    T: L^{q_1,1}(w_1)\times \dots \times L^{q_m,1}(w_m) \longrightarrow L^{q,\infty}(w_1^{q/q_1}\dots w_m^{q/q_m}),
\end{equation*}
for all exponents $1 \leq q_1,\dots,q_m<\infty$, with $\frac{1}{q}=\frac{1}{q_1}+\dots+\frac{1}{q_m}$, and all weights $w_1\in \widehat A_{q_1,\infty},\dots,w_m \in \widehat A_{q_m,\infty}$. For a one-weight version, see Corollary~\ref{corollmultiextrapolup}. 

Following A. Cordoba's philosophy on strong-type extrapolation,
we can say that: 
\begin{center}
{\calligra {There is no restricted weak type, only weighted (1,...,1,1/m).}} 
\end{center}

Let us point out that in the mixed-type setting, and when working with $A_r$ weights, we can either follow the classical approach, using Rubio de Francia's iteration algorithm, or our new strategy, with \eqref{mainsawyer} and related Sawyer-type inequalities, to run the extrapolation procedures, with the former leading to better constants than the latter, but in the restricted weak-type setting, the first option is not available, and we have no choice but to use Sawyer-type inequalities.

Our techniques can be further enhanced and, when combined with recently unearthed structures inside $A_p^{\mathcal R}$ (see Subsection~\ref{ss81}), can be used to tackle extrapolation assuming multi-variable conditions on the tuples of measures involved (see Subsection~\ref{ss82} and Theorem~\ref{nextmultiextrapoltotal}), including endpoint results that are out of reach for the classical procedures in \cite{multiap2,multiap1,zoe} (see Theorem~\ref{A1new}).

Inspired by \cite{duoextrapol}, we derive our multi-variable extrapolation theorems from one-variable off-diagonal extrapolation results (see Section~\ref{s5} and Subsection~\ref{ss82}).


For technical reasons, in all our extrapolation arguments, we require the constants in each estimate to depend increasingly on the constants of the weights involved. This hypothesis may seem restrictive at first, but it is not, since sharp constants are this way (see \cite[Footnote 3]{dgpp}).

Note that for mixed and restricted weak-type bounds for multi-variable operators, the Lorentz spaces that we consider have first exponents $1\leq r_1,\dots,r_m<\infty$, and $r$ such that $\frac{1}{r}=\frac{1}{r_1}+\dots+\frac{1}{r_m}$. Hence, we can identify each choice of exponents $r_1,\dots,r_m$ with the point $(\frac{1}{r_1},\dots,\frac{1}{r_m})$ in the \textit{space of parameters} $(0,1]^m$. 

A relevant region inside this $m$-cube is the so-called \textit{Banach range} (see Figure~\ref{figbanach}),
\begin{equation*}
   \mathfrak B_m \perdef \{(x_1,\dots,x_m) \in (0,1]^m : x_1+\dots+x_m < 1\},
\end{equation*}
where the corresponding values of the exponent $r$ are strictly bigger than one, and hence, $L^{r,\infty}(v)$ is a Banach space, being $v$ a weight. In particular, duality is available (see \cite[Theorem 1.4.16.(v)]{grafclas}). 

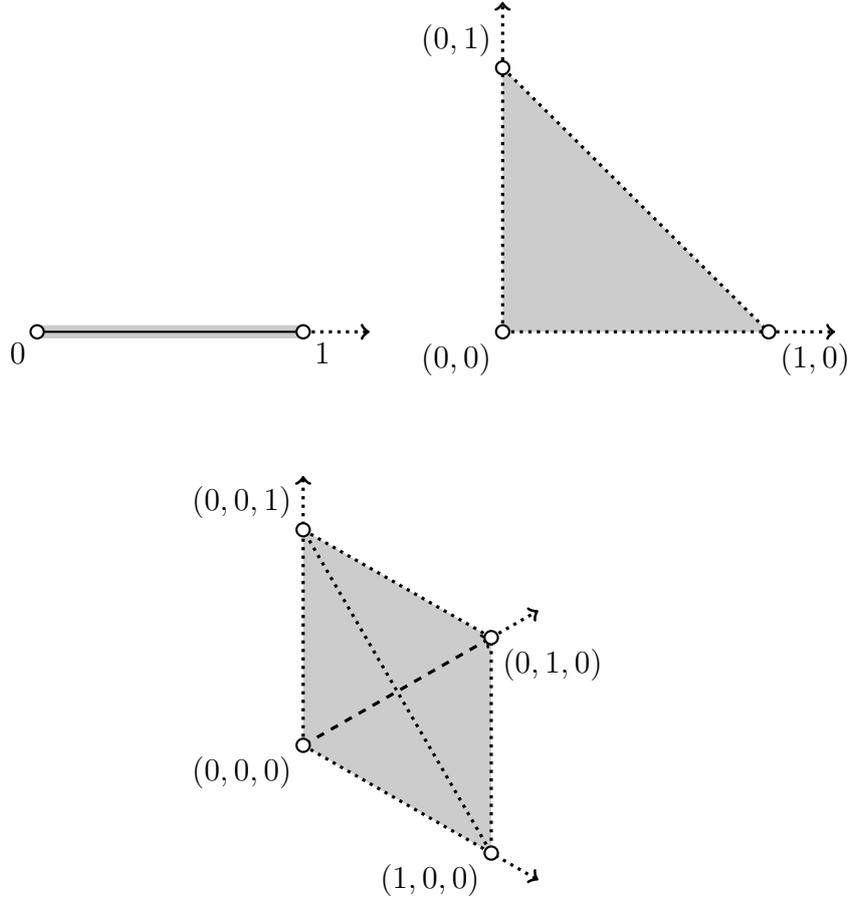
\begin{figure}[th]
\centering
\begin{tikzpicture}[scale=3.5]
\draw[line width=5,draw=gray!40] (0,0) -- (1,0);
\draw[ thick,draw=black] (0,0) -- (1,0);
\draw[very thick, dotted,draw=black,->] (1,0) -- (1.25,0);
\draw[white,fill] (0,0) circle [radius=0.025];
\draw[black,thick] (0,0) circle [radius=0.025] node [anchor=north east] at (0, 0) {$0$};
\draw[white,fill] (1,0) circle [radius=0.025];
\draw[black,thick] (1,0) circle [radius=0.025] node [anchor=north west] at (1, 0) {$1$};

\fill[gray!40] (1.75,0) -- (2.75,0) -- (1.75,1) -- (1.75,0);
\draw[very thick, dotted,draw=black] (1.75,1) -- (2.75,0);
\draw[very thick, dotted,draw=black,->] (1.75,0) -- (1.75,1.25);
\draw[very thick, dotted,draw=black,->] (1.75,0) -- (3,0);
\draw[white,fill] (1.75,0) circle [radius=0.025];
\draw[black,thick] (1.75,0) circle [radius=0.025] node [anchor=north east] at (1.75, 0) {$(0,0)$};
\draw[white,fill] (2.75,0) circle [radius=0.025];
\draw[black,thick] (2.75,0) circle [radius=0.025] node [anchor=north west] at (2.75, 0) {$(1,0)$};
\draw[white,fill] (1.75,1) circle [radius=0.025];
\draw[black,thick] (1.75,1) circle [radius=0.025] node [anchor=south east] at (1.75, 1) {$(0,1)$};


\fill[gray!40] (1,-0.75) -- (1,-1.566496580927726) -- ($(1,-1.566496580927726)+0.5*sqrt(2)*(1,0)-(0,0.4082482904638631)$) -- ($(1,-1.566496580927726)+0.5*sqrt(2)*(1,0)+(0,0.4082482904638631)$) -- (1,-0.75);

\draw[very thick, dotted,draw=black,->] (1,-1.566496580927726) -- (1,-0.5458758547680684);
\draw[very thick, dashed,draw=black,->] (1,-1.566496580927726) -- ($(1,-1.566496580927726)+0.5*sqrt(2)*(1,0)+(0,0.4082482904638631)$);
\draw[very thick, dotted,draw=black,->] ($(1,-1.566496580927726)+0.5*sqrt(2)*(1,0)+(0,0.4082482904638631)$) -- ($(1,-1.566496580927726)+(0.8838834764831844,0.5103103630798288)$);
\draw[very thick, dotted,draw=black,->] (1,-1.566496580927726) -- ($(1,-1.566496580927726)+(0.8838834764831844,-0.5103103630798288)$);

\draw[very thick, dotted,draw=black] (1,-0.75) -- ($(1,-1.566496580927726)+0.5*sqrt(2)*(1,0)+(0,0.4082482904638631)$) -- ($(1,-1.566496580927726)+0.5*sqrt(2)*(1,0)-(0,0.4082482904638631)$) -- (1,-0.75);

\draw[white,fill] (1,-0.75) circle [radius=0.025];
\draw[black,thick] (1,-0.75) circle [radius=0.025] node [anchor=south east] at (1,-0.75) {$(0,0,1)$};

\draw[white,fill] (1,-1.566496580927726) circle [radius=0.025];
\draw[black,thick] (1,-1.566496580927726) circle [radius=0.025] node [anchor=north east] at (1,-1.566496580927726) {$(0,0,0)$};

\draw[white,fill] ($(1,-1.566496580927726)+0.5*sqrt(2)*(1,0)+(0,0.4082482904638631)$) circle [radius=0.025];
\draw[black,thick] ($(1,-1.566496580927726)+0.5*sqrt(2)*(1,0)+(0,0.4082482904638631)$) circle [radius=0.025] node [anchor=north west] at ($(1,-1.566496580927726)+0.5*sqrt(2)*(1,0)+(0,0.4082482904638631)$) {$(0,1,0)$};

\draw[white,fill] ($(1,-1.566496580927726)+0.5*sqrt(2)*(1,0)-(0,0.4082482904638631)$) circle [radius=0.025];
\draw[black,thick] ($(1,-1.566496580927726)+0.5*sqrt(2)*(1,0)-(0,0.4082482904638631)$) circle [radius=0.025] node [anchor=north east] at ($(1,-1.566496580927726)+0.5*sqrt(2)*(1,0)-(0,0.4082482904638631)$) {$(1,0,0)$};

\end{tikzpicture}
\rule{.8\textwidth}{.4pt}
\caption[Pictorial representation of the Banach range for one, two, and three variables]{Pictorial representation of the Banach range for one, two, and three variables.}
\label{figbanach}
\end{figure}

Duality has proved to be a powerful tool in the study of weighted inequalities for classical operators, especially when combined with sparse domination techniques. Thus, working with Lorentz spaces where duality is not available is a problem in practice, and it gets worse as we increase the number of variables $m$, since the Banach range shrinks fast. In fact, one can check that 
\begin{equation*}
    \left| \mathfrak B_m \right| = \frac{1}{m!}.
\end{equation*}

This lack of duality can sometimes be circumvented by wisely using Kolmogorov's inequalities (see \cite{prp} and \cite[Chapter 5]{thesis}), but not always, and that's when our extrapolation kicks in. We can prove bounds in the Banach range by hand, and then effectively extend them outside it via multi-variable mixed or restricted weak-type extrapolation techniques.

In particular, our extrapolation schemes are handy for overcoming two fundamental problems of weak Lebesgue spaces $L^{r,\infty}(v)$ with $0<r\leq 1$, strongly related to the absence of duality: the lack of H\"older-type inequalities with the change of measures, and Minkowski's integral inequality.

The first problem becomes an obstacle when working with product-type operators. Nevertheless, using our H\"older-type inequalities from \cite[Subsection 2.2.1]{thesis}, we can obtain bounds for such operators in the Banach range, and then apply an extrapolation result to extend them past such a range of exponents. For the exact details, see Proposition~\ref{producttype} and Theorem~\ref{producttypeextrapol}. These arguments also apply to some multi-variable commutators, as shown in Corollary~\ref{commutextrapol}.

The second problem is an impediment when trying to produce bounds for averaging operators. Here, the strategy is to prove estimates in the Banach range using Minkowski's integral inequality and then extrapolate outside it, as we see in Theorem~\ref{averagingops}.

As a particular case, we started working with multi-linear multipliers of the form 
\begin{equation*}
    T_{\mathfrak m}(f_1,\dots,f_m)(x)\perdef \int_{\mathbb R} \dots \int_{\mathbb R} \mathfrak m(\vec \xi) \widehat f_1 (\xi_1)\dots \widehat f_m (\xi_m) e^{2 \pi i x (\xi_1+\dots+\xi_m)}d\xi_1 \dots d\xi_m,
\end{equation*}
initially defined for Schwartz functions $f_1,\dots,f_m$, and $x \in \mathbb R$. The study of such operators was initiated by R. R. Coifman and Y. Meyer (see \cite{come1,come2}). In recent years, the interest in them has increased, following the works by M. Lacey and C. Thiele on the bi-linear Hilbert transform and Calder\'on's conjecture (see \cite{mlacey3,mlacey4,mlacey2}). For more information and results on $m$-linear multipliers and related topics, see \cite{futom,grafmod,graka,graftor,keste,maldonaibo,tathi}.

Following \cite{duo}, we found that for nice symbols $\mathfrak m$, it is possible to write $T_{\mathfrak m}$ as an averaging operator of products of modulated and translated Hilbert transforms, and hence, in Theorem~\ref{bimultipliers}, we are able to establish mixed-type inequalities for these operators using our multi-variable extrapolation tools, combined with bounds on weight\-ed Lorentz spaces for the point-wise product of Hilbert transforms.

It's time to face the elephant in the room:  mixed and restricted weak-type results with $A_p^{\mathcal R}$ weights are scarce in the literature and generally difficult to prove, even in the Banach range. This means that the required estimate to start extrapolating will not always be available. Fortunately, we can find plenty of weak-type $(1,\dots,1,\frac{1}{m})$ operators out there, and extrapolate upwards from the endpoint $p_1=\dots=p_m=1$ with Corollary~\ref{cextrapolone} or Theorem~\ref{extrapolchar} to cover the full range of mixed and restricted weak-type bounds for them. A remarkable example is the family of so-called multi-linear bounded oscillation operators (see \cite{cao}), including Calder\'on-Zygmund operators, Littlewood-Paley square operators, Fourier integral operators, higher order Calder\'on commutators, and maximally modulated singular integrals. See Theorem~\ref{multiosc} for details.

Developing \eqref{mainsawyer} and its corollaries, Sawyer-type inequalities for Lorentz spaces and mixed and restricted weak-type bounds for $M^{\otimes}$, and using them to deduce multi-variable extrapolation schemes is an original idea of the author---E. R. P.---and was first described by him in \cite{thesis}. 
What follows is a satisfactory closure of that piece of work.

\section{Preliminaries}\label{s2}

In this section, we introduce some basic concepts that we will use throughout this document.

\subsection{Notation and conventions}\hfill\vspace{2.5mm}

In general, we will work in $\mathbb R^n$, with $1\leq n \in \mathbb N$. Unless otherwise specified, by a \textit{function} $f$ we mean a real or complex-valued function on $\mathbb R^n$. If we say that a function $f$ is \textit{measurable}, but we don't specify any measure, then it is with respect to the Lebesgue measure on $\mathbb R^n$. The same applies to measurable sets, integrals, and also to the expression a.e.; that is, \textit{almost everywhere}.

Given a measure $\nu$, and a $\nu$-measurable set $E$, we use the notation 
\begin{equation*}
    \nu(E) \perdef \int_E d \nu.
\end{equation*}
If $\nu$ is the Lebesgue measure, then we simply write $|E|$. Given a measurable function $f$, and a measurable set $E$, with $|E|\neq 0$, we use the notation 
\begin{equation*}
    \avgint_E f \perdef \frac{1}{|E|}\int_E f(x)dx.
\end{equation*}

A \textit{cube} Q is a subset of $\mathbb R^n$ that admits an expression as a Cartesian product of $n$ intervals of the same length, the \textit{side length of} $Q$, denoted by $\ell_Q$. If these intervals are all open, then the cube is called \textit{open}, and if they are all closed, then the cube is called \textit{closed}. By default, our cubes will be open, and we can find a unique $x_Q \in \mathbb R^n$, called the \textit{center of} $Q$, such that $Q=\left \{x_{Q}+\ell_Qy : y \in (-\frac{1}{2},\frac{1}{2})^n\right\}$. For $\gamma > 0$, $\gamma Q \perdef \left \{x_{Q}+\gamma \ell_Q y : y \in (-\frac{1}{2},\frac{1}{2})^n \right\}$. With the obvious modifications, these notions extend to arbitrary cubes.

Given non-negative quantities $\mathscr A$ and $\mathscr B$, we write $\mathscr A \lesssim \mathscr B$ if there exists a finite constant $C>0$, independent of $\mathscr A$ and $\mathscr B$, such that $\mathscr A\leq C \mathscr B$. If $\mathscr A\lesssim \mathscr B \lesssim \mathscr A$, then we write $\mathscr A\eqsim \mathscr B$. The constant $C$ is called the \textit{implicit constant}. Usually, we will denote implicit constants by $\kappa, c,\mathfrak c, C,$ or $\mathfrak C$. In many cases, they will depend on some parameters $\alpha_1,\dots,\alpha_{\ell}$, and if we want to point out that dependence, we will use subscripts, e.g., $\mathscr A \lesssim_{\alpha_1,\dots,\alpha_{\ell}} \mathscr B$, or $\mathscr A \eqsim_{\alpha_1,\dots,\alpha_{\ell}} \mathscr B$, or $\mathscr A \leq C_{\alpha_1,\dots,\alpha_{\ell}} \mathscr B$. We shall use numerical subscripts and superscripts to label different implicit constants appearing in the same argument. We write $\mathscr A \leq C(\alpha_1,\dots,\alpha_{\ell}) \mathscr B$ when we want to interpret $C$ as a function of the parameters $\alpha_1,\dots,\alpha_{\ell}$. In these cases, we may replace $C$ by other symbols, like $\phi,\varphi,\Phi,\psi,$ or $\Psi$, especially when the dependence on the parameters is monotonically increasing.

Given real or complex vector spaces $X_1,\dots,X_m$, and $Y$, endowed with quasi-norms $\Vert \cdot \Vert_{X_1},\dots,\Vert \cdot \Vert_{X_m}$, and $\Vert \cdot \Vert_{Y}$, respectively, and an operator $T$ defined on $X_1 \times \dots \times X_m$ and taking values in $Y$, we use the notation 
\begin{equation*}
    T : X_1  \times \dots \times X_m \longrightarrow Y
\end{equation*}
to indicate that $T$ is a \textit{bounded operator from $ X_1  \times \dots \times X_m$ to $Y$}; that is, there exists a finite constant $C>0$ such that for all $f_1 \in X_1,\dots,f_m \in X_m$,
\begin{equation*}
    \Vert T(f_1,\dots,f_m)\Vert_Y \leq C \prod_{i=1}^m \Vert f_i\Vert_{X_i}.
\end{equation*}
Among all such constants $C$, we shall denote by $\Vert T \Vert_{\prod_{i=1}^m X_i \rightarrow Y}$ the smallest one.

We adhere to the usual convention that the \textit{empty sum} (the sum containing no terms) is equal to zero, and the \textit{empty product} is equal to one.

\subsection{Lorentz spaces and weights}\hfill\vspace{2.5mm}

We include a brief exposition about Lebesgue and Lorentz spaces, containing definitions and well-known properties. For a detailed discussion, see \cite{BS,grafclas}.  

Given $0<p<\infty$, and a $\sigma$-finite measure space $(X,\nu)$,
$L^{p}(X,\nu)$ is the set of $\nu$-measurable functions $f$ on $X$ such that
\begin{equation*}
\Vert f \Vert_{L^{p}(X,\nu)}
\perdef \left(\int_X |f|^p d\nu \right)^{1/p}  <\infty,
\end{equation*}
and $L^{\infty}(X,\nu)$ is the set of $\nu$-measurable functions $f$ on $X$ such that
\begin{align*}
    \Vert f \Vert_{L^{\infty}(X,\nu)} & \perdef \nu \text{-} \esssup_{x \in X} |f(x)| 
    < \infty.
\end{align*}

The \textit{Lebesgue space} $L^p(X,\nu)$ is a Banach space for $1\leq p \leq \infty$, and a quasi-Banach space for $0<p<1$. 

Given $0<p,q<\infty$, and a $\nu$-measurable function $f$ on $X$, define
\begin{equation*}
    \Vert f \Vert_{L^{p,q}(X,\nu)}\perdef 
        \left(p \int_0^\infty y^q \lambda_f^\nu (y) ^{q/p} \frac{dy}{y} \right)^{1/q} = \left( \int_0^{\infty} t^{q/p} f^*_{\nu}(t)^q \frac{dt}{t}\right)^{1/q},
\end{equation*}
and for $q=\infty$, define
\begin{equation*}
    \Vert f \Vert_{L^{p,\infty}(X,\nu)}\perdef  \sup_{y>0}y  \lambda_f^\nu(y)^{1/p}=\sup_{t>0}t^{1/p}f_\nu^*(t),
\end{equation*}
where $\lambda_f^{\nu}$ is the \textit{distribution function of $f$ with respect to $\nu$}, defined on $[0,\infty)$ by
\begin{equation*}
    \lambda_f^\nu(y)\perdef \nu( \{x\in X : |f(x)|>y\}),
\end{equation*}
and $f_\nu^*$ is the \textit{decreasing rearrangement of $f$ with respect to $\nu$}, defined on $[0,\infty)$ by 
$$
f_\nu^*(t) \perdef \inf\{y>0:\lambda_f^\nu(y)\leq t\}.
$$

The set of all $\nu$-measurable functions $f$ on $X$ with $\Vert f \Vert_{L^{p,q}(X,\nu)}<\infty$ is denoted by $L^{p,q}(X,\nu)$, and it is called the \textit{Lorentz space with indices $p$ and $q$}. The space $L^{\infty,\infty}(X,\nu)$ is $L^{\infty}(X,\nu)$ by definition. 

For $0<p\leq \infty$, $L^{p,p}(X,\nu)=L^p(X,\nu)$, and hence, Lebesgue spaces are particular examples of Lorentz spaces. The space $L^{p,\infty}(X,\nu)$ is usually called \textit{weak} $L^{p}(X,\nu)$. 

Some Lorentz spaces that will be of great interest for us are $L^{p,1}(\mathbb R^n,\nu)$, $L^{p,p}(\mathbb R^n,\nu)$, and $L^{p,\infty}(\mathbb R^n,\nu)$, where $d\nu(x) =w(x)dx$, and $w$ is a weight; i.e., $0<w\in L^1_{loc}(\mathbb R^n)$. For such measures on $\mathbb R^n$, we shall write $L^{p,q}(\nu)$, $L^{p,q}(w)$, or $L^{p,q}(\mathbb R^n)$ if $w=1$.

In general, $L^{p,q}(X,\nu)$ is a quasi-Banach space, but if $1<p<\infty$ and $1\leq q \leq \infty$, or $p=q=1$, or $p=q=\infty$, then it can be normed to become a Banach space.

Lorentz spaces are nested; that is, if $0<p<\infty$, and $0<q<r\leq \infty$, then
\begin{equation*}
    L^{p,q}(X,\nu) \hookrightarrow L^{p,r}(X,\nu),
\end{equation*}
and for every $f \in L^{p,q}(X,\nu)$,
\begin{equation*}
    \Vert f \Vert_{L^{p,r}(X,\nu)} \leq \left( \frac{q}{p}\right)^{\frac{r-q}{rq}} \Vert f \Vert_{L^{p,q}(X,\nu)}.
\end{equation*}

Given parameters $0<r<p<\infty$, consider the quantity
\begin{equation*}
\vertiii{f}_{L^{p,\infty}(X,\nu)}\perdef \sup_{0<\nu(E)<\infty} \nu(E)^{\frac{1}{p}-\frac{1}{r}}\left(\int_E |f|^r d \nu \right)^{1/r},
\end{equation*}
where the supremum is taken over all $\nu$-measurable sets $E \subseteq X$ such that $0<\nu(E)<\infty$. We have that
\begin{equation*}
    \|f\|_{L^{p,\infty}(X,\nu)} \leq \vertiii{f}_{L^{p,\infty}(X,\nu)} \leq \left(\frac{p}{p-r} \right)^{1/r} \|f\|_{L^{p,\infty}(X,\nu)}.
\end{equation*}
This 
is classical (see 
\cite[Exercise 1.1.12]{grafclas}), and we will refer to these inequalities as \textit{Kolmogorov's inequalities}.

For $f\in L^1_{loc}(\mathbb R^n)$, the \textit{Hardy-Littlewood maximal operator} $M$, introduced in \cite{hali}, is defined by
\begin{equation*}
Mf(x)\perdef\sup_{Q\ni x}\frac{1}{|Q|}\int_Q |f(y)|dy, \quad x\in \mathbb R ^n,
\end{equation*}
where the supremum is taken over all cubes $Q\subseteq\mathbb R^n$ containing $x$. For an exponent $\mu>0$, we write $M^{\mu}f \perdef M(|f|^{1/\mu})^\mu$. 

In \cite{mucken}, Muckenhoupt studied the boundedness of $M$ on Lebesgue spaces $L^p(w)$, obtaining that for $1<p<\infty$, 
\begin{equation*}
M:L^p(w)\longrightarrow L^p(w)
\end{equation*}
if, and only if $w\in A_p$; that is, if 
\begin{equation*}
[ w]_{A_p}\perdef \sup_{Q} \left(\avgint_Q w \right) \left( \avgint_Q w  ^{1-p'} \right)^{p-1} <\infty.
\end{equation*}
Moreover, if $1\le p<\infty$, 
\begin{equation*}
M:L^p(w)\longrightarrow L^{p, \infty}(w)
\end{equation*}
if, and only if $w\in A_p$, where a weight $w\in A_1$ if 
\begin{equation*}
[w]_{A_1}\perdef
\sup_Q \left( \avgint_Q w\right) \big(\essinf_{x\in Q} w(x)\big)^{-1}<\infty.
\end{equation*}

Buckley proved in \cite{buckley} (see also \cite{carlos}) that for $1\leq p <\infty$,
\begin{equation*}
    \Vert M\Vert_{L^{p}(w)\to L^{p,\infty}(w)} \lesssim_{n}
    [w]_{A_p}^{1/p},
\end{equation*}
and if $p>1$, then
\begin{equation*}
   \Vert M\Vert_{L^{p}(w)\to L^{p}(w)} \lesssim_{n} p' 
   [w]_{A_p}^{\frac{1}{p-1}}.
\end{equation*}

In \cite{CHK,kt}, Chung, Hunt, and Kurtz, and Kerman, and Torchinsky proved that for $1\leq p < \infty$,
\begin{equation*}
M:L^{p,1}(w) \longrightarrow L^{p, \infty}(w) 
\end{equation*}
if, and only if $w\in A_p^{\mathcal R}$, where a weight $w$ is in $A_p^{\mathcal R}$ if 
\begin{equation*}
[w]_{A_p^{\mathcal R}} \perdef\sup_{Q} w(Q)^{1/p}\frac{\Vert \chi_Q w^{-1}\Vert_{L^{p',\infty}(w)}}{|Q|} <\infty,
\end{equation*}
or equivalently, if
\begin{equation*}
\Vert w \Vert _{A_p^{\mathcal R}} \perdef\sup_Q \sup_{E\subseteq Q} \frac{|E|}{|Q|}\left(\frac{w(Q)}{w(E)}\right)^{1/p} <\infty.
\end{equation*}
We have that $[w]_{A_p^{\mathcal R}} \leq \Vert w \Vert_{A_p^{\mathcal R}} \leq p [w]_{A_p^{\mathcal R}}$, and  
\begin{equation*}
\Vert M\Vert_{L^{p,1}(w)\to L^{p,\infty}(w)} \eqsim_{n,p} [ w]_{A_p^{\mathcal R}}.
\end{equation*}

A remarkable subclass of $A_p^{\mathcal R}$ is $\widehat A_p$, introduced in \cite{cgs}. Given $1\leq p < \infty$, a weight $w$ belongs to the class $\widehat A_{p}$ if there exist a function $h \in L^1_{loc}(\mathbb R^n)$, and a weight $u\in A_1$ such that $w= (Mh)^{1-p} u$. It is possible to associate a constant to this class of weights, given by
\begin{equation*}
    \Vert w \Vert_{\widehat A_{p}} \perdef  \inf \, [u]_{A_1}^{1/p},
\end{equation*}
where the infimum is taken over all weights $u\in A_1$ such that $w= (Mh)^{1-p} u$. If $w\in \widehat A_p$, then $\Vert w\Vert_{A_p^{\mathcal R}}\lesssim_{n,p} \Vert w \Vert_{\widehat A_{p}}$, and $\widehat A_p \subseteq A_p^{\mathcal R}$, but it is not known if this inclusion is strict for $p>1$. Note that $\widehat A_1 = A_1$, and for $p>1$, $A_p \subsetneq \widehat A_p$.

We now introduce some other classes of weights that will appear later. For more information about them, see \cite{CUMP,cun,DMRO,gcrf,grafclas,hpr,torch}.

Define the class of weights
\begin{equation*}
A_{\infty}\perdef\bigcup_{p\geq 1} A_p = \bigcup_{p\geq 1} A_p ^{\mathcal R}.
\end{equation*}
It is known that a weight $w\in A_{\infty}$ if, and only if
\begin{equation*}
[w]_{A_{\infty}}\perdef \sup_Q \frac{1}{w(Q)} \int_Q M(w\chi_Q) < \infty.
\end{equation*}

Given $s>1$, we say that a weight $w\in RH_s$ if
\begin{equation*}
[w]_{RH_s}\perdef\sup_Q \frac{|Q|}{w(Q)} \left( \avgint_Q w^s\right)^{1/s}<\infty,
\end{equation*}
and $w\in RH_{\infty}$ if
\begin{equation*}
[w]_{RH_{\infty}}\perdef 
\sup_Q \frac{|Q|}{w(Q)} \esssup_{x\in Q} w(x) <\infty.
\end{equation*}
We have that
\begin{equation*}
A_\infty = \bigcup_{1<s\leq \infty } RH_s.
\end{equation*}

For a positive Borel measure $\nu$ on $\mathbb R^n$,
$$
M_{\nu} f(x)\perdef\sup_{Q\ni x}\frac{1}{\nu(Q)}\int_Q |f(y)|  d\nu(y), \quad x \in \mathbb R^n,
$$
where the supremum is taken over all cubes $Q \subseteq \mathbb R^n$ containing $x$ and such that $0<\nu(Q)<\infty$. If $d\nu(y)=w(y)dy$ for some weight $w$, we simply write $M_w$, and call it the \textit{weighted Hardy-Littlewood maximal operator}.

Given a positive Borel measure $\nu$ on $\mathbb R^n$, we say it is \textit{locally finite} if $\nu(Q)< \infty$ for every cube $Q \subseteq \mathbb R^n$, 
and \textit{doubling} if there exists a constant $C>0$ such that for every cube $Q \subseteq \mathbb R^n$, $\nu(2Q)\leq C \nu(Q)$; the smallest of such constants is the \textit{doubling constant of} $\nu$. For a measure $\nu$ with these properties, and $f\in L^1_{loc}(\mathbb R^n, \nu)$, $|f(x)|\leq M_{\nu}f(x)$ $\nu$-a.e. $x \in \mathbb R^n$ (see \cite[Theorem 7.8]{folland} and \cite[Theorem 8.4.6]{czaja}). A particular case of interest is $d\nu(y)=w(y)dy$, with $w \in A_{\infty}$.

Given a positive, locally finite, doubling Borel measure $\nu$ on $\mathbb R^n$, and $0<w\in L^1_{loc}(\mathbb R^n, \nu)$, we say that $w\in A_p(\nu)$, with $p>1$, if
\begin{equation*}
[w]_{A_p(\nu)} \perdef  \sup_{Q} \left(\frac{1}{\nu(Q)}\int_Q wd \nu \right) \left( \frac{1}{\nu(Q)}\int_Q w  ^{1-p'}d\nu \right)^{p-1} <\infty, 
\end{equation*}
and $w\in A_1(\nu)$ if 
\begin{align*}
[w]_{A_1(\nu)} \perdef 
\sup_Q \left(\frac{1}{\nu(Q)}\int_Q w d \nu \right) \big(\nu\text{-}\essinf_{x\in Q} w(x)\big)^{-1}<\infty.
\end{align*}

Also, for $p\geq 1$, $w \in A_p^{\mathcal R}(\nu)$ if 
\begin{equation*}
[w]_{A_p^{\mathcal R}(\nu)} \perdef\sup_{Q} \left(\int_Q w d\nu \right)^{1/p}\frac{\Vert \chi_Q w^{-1}\Vert_{L^{p',\infty}(w d \nu)}}{\nu(Q)} <\infty,
\end{equation*}
and $w \in \widehat A_{p}(\nu)$ if there exist a function $h \in L^1_{loc}(\mathbb R^n,\nu)$, and $u\in A_1(\nu)$ such that $w= (M_{\nu} h)^{1-p} u$. To this class of functions we can associate the constant
\begin{equation*}
    \Vert w \Vert_{\widehat A_{p}(\nu)} \perdef  \inf \, [u]_{A_1(\nu)}^{1/p},
\end{equation*}
where the infimum is taken over all $u\in A_1(\nu)$ such that $w= (M_{\nu}h)^{1-p} u$.

As before, we define
\begin{equation*}
A_{\infty}(\nu)\perdef  \bigcup_{p\geq 1} A_p(\nu).
\end{equation*}
For $w\in A_{\infty}(\nu)$,
\begin{equation*}
[w]_{A_{\infty}(\nu)}\perdef  \sup_Q \frac{1}{\int_Q w d \nu} \int_Q M_{\nu}(w\chi_Q) d \nu < \infty.
\end{equation*} 

If $p>1$, then $M_{\nu}: L^p(w d\nu) \longrightarrow L^p(w d\nu)$ if, and only if $w\in A_p(\nu)$, and if $p\geq 1$, then $M_{\nu}:L^{p,1}(w d\nu) \longrightarrow L^{p, \infty}(w d \nu) $ if, and only if 
$w \in A_p^{\mathcal R}(\nu)$.

In \cite{LOPTT}, the following multi-variable extension of the Hardy-Littlewood maximal operator was introduced in connection with the theory of multi-linear Calder\'on-Zygmund operators:
$$
\mathcal M (\vec f)(x)\perdef \sup_{Q\ni x} \prod_{i=1}^m \left(\frac{1}{|Q|}\int_Q |f_i(y_i)|dy_i \right) , \quad x \in \mathbb R^n,
$$
for $\vec{f}=(f_1,\dots,f_m)$, with $f_i\in L^1_{loc}(\mathbb R ^n)$, $i=1,\dots,m$. 

 For $1 \leq p_1,\dots, p_m <\infty $, $\vec P= (p_1,\dots,p_m)$, $\frac{1}{p}=\frac{1}{p_1}+ \dots + \frac{1}{p_m}$, and weights $w_1, \dots,w_m$, with $\vec w= (w_1,\dots,w_m)$, and $\nu_{\vec w}\perdef w_1^{p/p_1} \dots w_m ^{p/p_m}$, $$
\mathcal M : L^{p_1}(w_1) \times \dots \times L^{p_m}(w_m) \longrightarrow L^{p,\infty}(\nu_{\vec w})
$$
if, and only if $\vec w \in A_{\vec P}$; that is, if
$$
[\vec w]_{A_{\vec P}}\perdef \sup_Q \left( \avgint_Q \nu_{\vec w}\right)^{1/p} \prod_{i=1}^m \left( \avgint_Q w_i ^{1-p_i'}\right)^{1/p_i'} < \infty,
$$
where $\left( \avgint_Q w_i ^{1-p_i'}\right)^{1/p_i'}$ is replaced by $\big(\essinf_{x\in Q} w_i (x)\big)^{-1}$ if $p_i = 1$. Moreover, if $1 < p_1,\dots, p_m <\infty $, then 
$$
\mathcal M : L^{p_1}(w_1) \times \dots \times L^{p_m}(w_m) \longrightarrow L^{p}(\nu_{\vec w})
$$
if, and only if $\vec w \in A_{\vec P}$. 

More generally, for $0 \leq \ell \leq m$, and $\vec R= (p_1,\dots,p_{\ell})$,
\begin{equation*}
    \mathcal M :L^{p_1,1}(w_1)\times \dots \times L^{p_{\ell},1}(w_{\ell})\times L^{p_{\ell+1}}(w_{\ell+1}) \times \dots \times L^{p_{m}}(w_{m}) \longrightarrow L^{p,\infty}(\nu_{\vec w})
\end{equation*}
if, and only if $\vec w \in A_{\vec P,\vec R}^{\mathfrak M}$; that is, if
\begin{align*}
[\vec w]_{A_{\vec P,\vec R}^{\mathfrak M}}  \perdef  \sup_{Q} \nu_{\vec w}(Q)^{1/p} \left(\prod_{i=1} ^{\ell} \frac{\Vert \chi_Q w_i ^{-1} \Vert_{L^{p_i',\infty}(w_i)}}{|Q|}\right) \left(\prod_{i=\ell+1} ^{m} \frac{\Vert \chi_Q w_i ^{-1} \Vert_{L^{p_i'}(w_i)}}{|Q|}\right) < \infty.
\end{align*}
If $\ell=0$, then $A_{\vec P,\vec R}^{\mathfrak M} = A_{\vec P}$, and if $\ell = m$, then $A_{\vec P,\vec R}^{\mathfrak M} = A_{\vec P}^{\mathcal R}$ (see \cite[Theorem 9]{prp} and \cite[Remark 5.2.9]{thesis}).

\subsection{Types of bounds}\hfill\vspace{2.5mm}

Let $m\geq 1$, and let $T$ be an $m$-variable operator defined for suitable measurable functions on $\mathbb R^n$. Given exponents $0<p_1,q_1,\dots,p_m,q_m,p,q \leq \infty$, and weights $w_1,\dots,w_m,w$, suppose that
\begin{equation*}
    T:L^{p_1,q_1}(w_1)\times \dots \times L^{p_m,q_m}(w_m) \longrightarrow L^{p,q}(w);
\end{equation*}
that is, $T$ is a bounded operator from $L^{p_1,q_1}(w_1)\times \dots \times L^{p_m,q_m}(w_m)$ to  $L^{p,q}(w)$.

\begin{itemize}
    \item [($a$)] We say that $T$ is of \textit{strong type} $(p_1,\dots,p_m,p)$ if $q_1=p_1,\dots,q_m=p_m$, and $q=p$. 

\item [($b$)] We say that $T$ is of \textit{weak type} $(p_1,\dots,p_m,p)$ if $q_1=p_1,\dots,q_m=p_m$, and $q=\infty$. 

\item [($c$)] We say that $T$ is of \textit{restricted weak type} $(p_1,\dots,p_m,p)$ if $q_1=\dots=q_m=1$, and $q=\infty$. We may also use this terminology if $0<q_i\leq 1$, $i=1,\dots,m$. 

\item [($d$)] We say that $T$ is of \textit{mixed type} $(p_1,\dots,p_{\ell},p_{\ell+1},\dots,p_m,p)$, with $1\leq\ell < m$, if $q_{1}\leq 1,\dots,q_{\ell}\leq 1$ and $q_{\ell+1}=p_{\ell+1},\dots,q_m=p_m$, and $q=\infty$. We may also use this terminology if $1\leq p_1,\dots,p_m<\infty$, and $w_i\in A_{p_i}^{\mathcal R}$, for $i=1,\dots,\ell$, and $w_i\in A_{p_i}$, for $i=\ell+1,\dots,m$, independently of the choice of the other exponents. We first introduced this definition in \cite{thesis}.
\end{itemize}

Analogously, we will talk about strong, weak, mixed, and restricted weak-type inequalities.

The definitions of types strong and weak are standard (see \cite[Section 1.3]{grafclas}), but the ones of mixed and restricted weak may vary depending on the source (see \cite{AM}, \cite[Chapter 4]{BS}, \cite{cgs,cs,dune}, \cite[Section 1.4]{grafclas}, \cite{stws}). In the long run, referring to our mixed type as \textit{mild type} may be convenient.

\subsection{\texorpdfstring{$(\varepsilon,\delta)$-atomic operators}{Epsilon-delta atomic operators}}\hfill\vspace{2.5mm}

We introduce multi-variable extensions of some topics presented in \cite{linear,multi,cgs}.

\begin{definition}
Given $\delta>0$, we say that a tuple of functions $(\mathfrak a_1,\dots, \mathfrak a_m)$ in $\prod_{i=1}^m L^1(\mathbb R^n)$ is a $\delta$-\textit{atom} if
$$
\int _{\mathbb R ^n} \dots  \int _{\mathbb R ^n} \mathfrak a_1(x_1) \dots \mathfrak a_m(x_m) dx_1 \dots dx_m = 0,
$$
and there exist cubes $Q_1, \dots, Q_m \subseteq \mathbb R^n$ such that for $i=1,\dots,m$, $|Q_i|\leq \delta$ and $\text{supp } \mathfrak a_i \subseteq Q_i$.
\end{definition}

Recall that for a measurable function $f$, 
$$\Vert f\Vert_{L^1(\mathbb R^n)+L^\infty(\mathbb R^n)} \perdef \int _0 ^1 f^*(s)ds.$$

\begin{definition}
Let $T$ be a multi-sub-linear operator defined for suitable measurable functions (see \cite[Page 494]{grafmod}).

\begin{enumerate}
\item [($a$)] We say that $T$ is $(\varepsilon,\delta)$-\textit{atomic} if for every $\varepsilon>0$, there exists $\delta>0$ such that for every $\delta$-atom $(\mathfrak a_1,\dots, \mathfrak a_m)$, 
$$\Vert T(\mathfrak a_1,\dots, \mathfrak a_m) \Vert_{L^1(\mathbb R^n)+L^\infty(\mathbb R^n)} \leq \varepsilon \prod_{i=1}^m \Vert \mathfrak a_i\Vert_{L^1(\mathbb R^n)}.$$

\item [($b$)] We say that $T$ is $(\varepsilon,\delta)$-\textit{atomic approximable} if there exists a sequence $\{T_k\}_{k\in \mathbb N}$ of $(\varepsilon,\delta)$-atomic operators such that for all measurable sets $E_1,\dots, E_m \subseteq \mathbb R^n$, and all $f_1,\dots,f_m \in L^1(\mathbb R^n)$ such that $\Vert f_i \Vert_{L^\infty(\mathbb R^n)} \leq 1$, $i=1,\dots,m$,
$$|T_k(\chi_{E_1},\dots,\chi_{E_m})|\leq |T(\chi_{E_1},\dots,\chi_{E_m})|, \quad \text{ and } \quad |T(\vec f)|\leq \liminf\limits_{k\rightarrow \infty} |T_k(\vec f)|.$$

\item [($c$)] We say that $T$ is \textit{iterative} $(\varepsilon,\delta)$-\textit{atomic} (resp. \textit{approximable}) if for all functions $g_1,\dots,g_m\in L^1(\mathbb R^n)$ with $\Vert g_i \Vert_{L^\infty(\mathbb R^n)} \leq 1$, $i=1,\dots,m$, the $m$ one-variable operators of the form $T(g_1,\dots,g_{i-1},\cdot,g_{i+1},\dots,g_m)$ are $(\varepsilon,\delta)$-\textit{atomic} (resp. \textit{approximable}).
\end{enumerate}
\end{definition}

The next result will be needed later. It is a two-weight version of \cite[Theorem 3.5]{cgs}, and the proof is the same.

\begin{theorem}\label{lin}
Let $T$ be a sub-linear operator that is $(\varepsilon,\delta)$-atomic approximable, and fix $0<q<\infty$. Given weights $u \in A_1$ and $v$, if there exists a constant $C>0$ such that for every measurable set $E\subseteq \mathbb R^n$,
$$
\Vert T(\chi_{E})\Vert_{L^{q,\infty}(v)}\leq C u(E),
$$
then 
$$
T:L^1(u) \longrightarrow L^{q,\infty}(v),
$$
with constant bounded by $2^n C  [u]_{A_1}$. 
\end{theorem}

A remarkable multi-variable extension of Theorem~\ref{lin} is the following, the proof of which is similar to that of \cite[Theorem 3.9]{multi}, with obvious modifications based on the proof of \cite[Theorem 3.5]{cgs}.

\begin{theorem}\label{epsdelta}
Let $T$ be a multi-sub-linear operator that is $(\varepsilon,\delta)$-atomic approximable or iterative $(\varepsilon,\delta)$-atomic approximable. Given weights $u_1,\dots,u_m \in A_1$, and $u=u_1^{1/m}\dots u_m^{1/m}$, if there exists a constant $C>0$ such that for all measurable sets $E_1,\dots,E_m \subseteq \mathbb R^n$,
$$
\Vert T(\chi_{E_1},\dots,\chi_{E_m})\Vert_{L^{\frac{1}{m},\infty}(u)}\leq C u_1(E_1)\dots u_m(E_m),
$$
then 
$$
T:L^1(u_1) \times \dots \times L^1(u_m) \longrightarrow L^{\frac{1}{m},\infty}(u),
$$
with constant bounded by $2^{mn}  C  [u_1]_{A_1}\dots [u_m]_{A_1}$. 
\end{theorem}

\section{Technical results}

In this section, we gather some technical results that we will use throughout this article.

\subsection{Interpolation of weights}\hfill\vspace{2.5mm}

The next theorem gives us a restricted weak-type interpolation result for weights.

\begin{theorem}\label{interpol}
Fix an integer $m\geq 2$. Let $0<p<\infty$, and $0\leq \theta_1, \dots, \theta_m \leq 1$ such that $\theta_1+\dots+\theta_m=1$. Let $u_1,\dots,u_m,v_1,\dots,v_m$ be weights, and write $u=u_1 ^{\theta_1} \dots u_m^{\theta_m}$, and $v=v_1 ^{\theta_1} \dots v_m^{\theta_m}$. Let $T$ be a sub-linear operator defined for characteristic functions. Suppose that for $i=1,\dots,m$, there exists a constant $C_i>0$ such that for every measurable set $F\subseteq \mathbb R^n$,
\begin{equation}\label{eqinterpolh1}
\left \Vert T(\chi_F) \right \Vert_{L^{p,\infty}(u_i)} \leq C_i \left \Vert \chi_F  \right \Vert_{L^{p,1}(v_i)}.
\end{equation}
Then, for $C=
\min \{C_1+ \dots + C_m, m C_1^{\theta_1} \dots C_m^{\theta_m}\}$, and every measurable set $E\subseteq \mathbb R^n$,
\begin{equation}\label{eqinterpolh2}
\left \Vert T(\chi_E) \right \Vert_{L^{p,\infty}(u)} \leq C \left \Vert \chi_E  \right \Vert_{L^{p,1}(v)}.
\end{equation}
\end{theorem}

\begin{proof}
Without loss of generality, we can assume that for $i=1,\dots,m$, $\theta_i \neq 0$. First, we prove that \eqref{eqinterpolh2} holds for $C=C_1+\dots + C_m$. The case $m=2$ is part of popular folklore; in \cite[Lemma 4.1.3]{thesis} we gave a proof based on \cite[Lemma 3]{vargas}. 

For the case $m>2$, we proceed by applying the case $m=2$ iteratively $m-1$ times. Let us assume that we are performing the $k$th iteration, with $1\leq k \leq m-1$, and that all the previous iterations are already done. We choose the weights 
\begin{equation*} 
u_1^{(k)}\perdef \prod_{i=1}^{k} u_i^{\frac{\theta_i}{\theta_1+\dots+\theta_{k}}}, \quad v_1^{(k)}\perdef  \prod_{i=1}^{k} v_i^{\frac{\theta_i}{\theta_1+\dots+\theta_{k}}}, \quad u_2^{(k)} \perdef   u_{k+1}, \quad v_2^{(k)} \perdef   v_{k+1},
\end{equation*} 
and the exponents
\begin{equation*} 
\theta_1 ^{(k)} \perdef   \frac{\theta_1+\dots+\theta_{k}}{\theta_1+\dots+\theta_{k+1}}, \quad \theta_2^{(k)} \perdef  \frac{\theta_{k+1}}{\theta_1+\dots+\theta_{k+1}},
\end{equation*}
and write
\begin{equation*} 
C^{(k)} \perdef  C_1 + \dots + C_k.
\end{equation*}

In virtue of \eqref{eqinterpolh1}, if we apply the case $m=2$ for the $k$th time, with exponents $\theta_1^{(k)}$ and $\theta_2^{(k)}$, and weights $u_1^{(k)},u_2^{(k)}, v_1^{(k)}$, and $v_2^{(k)}$, then we get that for every measurable set $E\subseteq \mathbb R^n$,
\begin{equation*}
\Vert T(\chi_E) \Vert_{L^{p,\infty}(u_1^{(k+1)})} \leq C^{(k+1)} \Vert \chi_E \Vert_{L^{p,1}(v_1^{(k+1)})},
\end{equation*}
since $(u_1^{(k)})^{\theta_1^{(k)}}(u_2^{(k)})^{\theta_2^{(k)}}= u_1 ^{(k+1)}$, $(v_1^{(k)})^{\theta_1^{(k)}}(v_2^{(k)})^{\theta_2^{(k)}}= v_1 ^{(k+1)}$, and 
$C^{(k+1)} = C^{(k)}+C_{k+1}$. In particular, for $k=m-1$, $u_1^{(m)}=u$ and $v_1^{(m)}=v$, and we conclude that
\begin{equation*}
\Vert T(\chi_E) \Vert_{L^{p,\infty}(u)} \leq 
(C_1+ \dots + C_m) \Vert \chi_E \Vert_{L^{p,1}(v)}.
\end{equation*}

Now, by hypothesis, for $i=1,\dots,m$, and for every measurable set $F\subseteq \mathbb R^n$,
\begin{equation*}
\left \Vert T(\chi_F) \right \Vert_{L^{p,\infty}(u_i)} \leq C_i \left \Vert \chi_F  \right \Vert_{L^{p,1}(v_i)}= \left \Vert \chi_F  \right \Vert_{L^{p,1}(C_i ^p v_i)},
\end{equation*}
and applying the result that we have just proved, we deduce that for every measurable set $E\subseteq \mathbb R^n$,
\begin{equation*}
\left \Vert T(\chi_E) \right \Vert_{L^{p,\infty}(u)} \leq m \left \Vert \chi_E  \right \Vert_{L^{p,1}(C_1^{p\theta_1}\dots C_m^{p\theta_m}v)} = m C_1^{\theta_1}\dots C_m^{\theta_m}\left \Vert \chi_E  \right \Vert_{L^{p,1}(v)}.
\end{equation*}
\end{proof}



\subsection{\texorpdfstring{Extensions of $\widehat A_p$}{Extensions of Apgorro}}\hfill\vspace{2.5mm}

Let us start by defining the following class of weights, which was introduced in an unpublished version of \cite{cgs}.

\begin{definition}
Given $1\leq p < \infty$, and $1\leq N\in \mathbb N$, we say that a weight $w$ belongs to the class $\widehat A_{p,N}$ if there exist measurable functions $h_1, \dots, h_N \in L^1_{loc}(\mathbb R^n)$, parameters $\theta_1,\dots,\theta_N \in (0,1]$, with $\theta_1 + \dots + \theta_N=1$, and a weight $u\in A_1$ such that
\begin{equation}\label{descomp}
    w=\left (\prod_{i=1}^N (Mh_i) ^{\theta_i} \right)^{1-p} u.
\end{equation}

We can associate a constant to this class of weights, given by
\begin{equation*}
    \Vert w \Vert_{\widehat A_{p,N}} \perdef  \inf \, [u]_{A_1}^{1/p},
\end{equation*}
where the infimum is taken over all weights $u\in A_1$ such that $w$ can be written as \eqref{descomp}. 

We also define
\begin{equation*}
\widehat{A}_{p,\infty} \perdef  \bigcup_{N=1} ^\infty \widehat{A}_{p,N},
\end{equation*}
with the corresponding associated constant, given by
\begin{equation*}
[ w ] _{\widehat{A}_{p,\infty}}\perdef \inf_{N\geq 1} \Vert w \Vert_{\widehat{A}_{p,N}}.
\end{equation*}
For convenience, we take $\widehat A_{p,0}\perdef A_p$.
\end{definition}

It is clear that $\widehat{A}_{1,\infty}=A_1$, and $\widehat{A}_{p,1}=\widehat A_p$. 
Also, observe that for every $N\geq 1$, $\widehat A_{p,N} \subseteq \widehat A_{p,N+1}$, and $\Vert w \Vert_{\widehat{A}_{p,N+1}} \leq \Vert w \Vert_{\widehat{A}_{p,N}}$, but we don't know if these inclusion relations are strict.

We will use Theorem~\ref{interpol} to show that for $p\geq 1$, $\widehat A_{p,\infty} \subseteq A_p ^{\mathcal R}$, but due to the dependence on $m$ of the constant $C$ obtained there, we can't work with $[\ \cdot \ ]_{\widehat A_{p,\infty}}$, and we need to introduce a new constant for weights in $\widehat A_{p,\infty}$.


\begin{definition}
Given $1 \leq p < \infty$, and $w \in \widehat A_{p,\infty}$, we define the constant
\begin{equation*}
\Vert w \Vert_{\widehat A_{p,\infty}} \perdef   \inf_{N\geq 1} N \Vert w \Vert_{\widehat{A}_{p,N}}.
\end{equation*}
\end{definition}

We can see that $\Vert w \Vert_{\widehat A_{1,\infty}}=[ w ]_{\widehat{A}_{1,\infty}}=[w]_{A_1}$, and in general, $[ w ] _{\widehat{A}_{p,\infty}} \leq \Vert w \Vert_{\widehat A_{p,\infty}}$. Moreover, $\Vert w \Vert_{\widehat A_{p,\infty}} < \infty$ if, and only if  $[ w ] _{\widehat{A}_{p,\infty}}<\infty$, but we don't know if there exists an increasing function $\digamma:[1,\infty) \longrightarrow [0,\infty)$ such that $\Vert w \Vert_{\widehat A_{p,\infty}} \leq \digamma([ w ] _{\widehat{A}_{p,\infty}})$.

We can now prove that for $p\geq 1$, $\widehat A_{p,\infty} \subseteq A_p ^{\mathcal R}$.

\begin{theorem}\label{aprgorro}
Given $1\leq p <\infty$, there exists a constant $C>0$, depending only on $p$ and the dimension $n$, such that for every $N\geq1$, and every weight $w\in \widehat A_{p,N}$, 
\begin{equation}\label{eqaprgorroc}
    [w]_{A_p ^{\mathcal R}} \leq  C N \Vert w \Vert _{\widehat A_{p,N}}.
\end{equation}
In particular, if $w \in \widehat A_{p,\infty}$, then $w \in A_p ^{\mathcal R}$, and 
\begin{equation*}
    [w]_{A_p ^{\mathcal R}} \leq C \Vert w \Vert _{\widehat A_{p,\infty}}.
\end{equation*}
\end{theorem}

\begin{proof}
Observe that if $p=1$, then the result is true for any $C\geq 1$, so we will assume that $p>1$.

For $N=1$, if $w \in \widehat A_{p,1}$, then we can find a locally integrable function $h$, and a weight $u\in A_1$ such that $w=(Mh)^{1-p} u$. It was proved in \cite[Corollary 2.8]{cgs} that $[(Mh)^{1-p}u]_{A_p ^{\mathcal R}}\leq c_{n,p} [u]_{A_1}^{1/p}$, and taking the infimum over all such weights $u\in A_1$, we get that $[w]_{A_p^{\mathcal R }} \leq c_{n,p} \Vert w \Vert _{\widehat A_{p,1}}$. By Lemma~\ref{wweights1}, we can take $c_{n,p}= \mathfrak c_n^{1/p'}$.


For $N \geq 2$, if $w \in \widehat A_{p,N}$, then we can find locally integrable functions $h_1,\dots,h_{N}$, a weight $u\in A_1$, and real values $0<\theta_1,\dots,\theta_{N}\leq 1$, with $\sum_{i=1}^{N}\theta_i = 1$, such that 
$$w=\left(\prod_{i=1}^{N}(Mh_i) ^{\theta_i}\right)^{1-p}u = \prod_{i=1}^N \left((Mh_i)^{1-p}u\right)^{\theta_i}\defper \prod_{i=1}^N w_i.$$

Note that for $i=1,\dots,m$, $w_i \in \widehat A_{p,1}$, and we already know that $\widehat A_{p,1} \subseteq A_p^{\mathcal R}$, so in virtue of \cite[Remark 10]{prp}, for every measurable set $E\subseteq \mathbb R^n$,
\begin{align*}
\begin{split}
\left \Vert M(\chi_E) \right \Vert_{L^{p,\infty}(w_i)} & \leq 2^n 72^{n/p} [w_i]_{A_p^{\mathcal R}} \left \Vert \chi_E  \right \Vert_{L^{p,1}(w_i)} 
\leq 2^n 72^{n/p} c_{n,p} [u]_{A_1}^{1/p} \left \Vert \chi_E  \right \Vert_{L^{p,1}(w_i)}.
\end{split}
\end{align*}  
We can now apply Theorem~\ref{interpol} to deduce that for every measurable set $E\subseteq \mathbb R^n$,
\begin{equation*}
\left \Vert M(\chi_E) \right \Vert_{L^{p,\infty}(w)} \leq 2^n 72^{n/p} c_{n,p} N[u]_{A_1}^{1/p} \left \Vert \chi_E  \right \Vert_{L^{p,1}(w)}.
\end{equation*}  

Thus, \cite[Theorem 10]{prp} implies that
\begin{equation*} 
[w]_{A_{p}^{\mathcal R}} \leq \Vert w \Vert_{A_{p}^{\mathcal R}} \leq 2^n 72^{n/p} p c_{n,p} N [u]_{A_1}^{1/p},
\end{equation*} 
and taking the infimum over all suitable representations of $w$, we conclude that
\begin{equation*} 
[ w]_{A_{p}^{\mathcal R}} \leq 2^n 72^{n/p} p c_{n,p} N \Vert w \Vert_{\widehat A_{p,N}},
\end{equation*} 
and hence, \eqref{eqaprgorroc} holds taking $C = 2^n 72^{n/p} p c_{n,p}$. In particular, $C\lesssim_n p$.

Finally, given $w\in \widehat A_{p,\infty}$, we have that 
\begin{equation*} 
[w]_{A_{p}^{\mathcal R}} \leq C \inf_{N\geq1 \, : \, w \in \widehat A_{p,N}}  N \Vert w \Vert_{\widehat A_{p,N}}= C \Vert w \Vert_{\widehat A_{p,\infty}}, 
\end{equation*} 
because if $N\geq 1$ is such that $w\not \in \widehat A_{p,N}$, then $\Vert w \Vert_{\widehat A_{p,N}} = \inf \varnothing = \infty$.
\end{proof}

\subsection{Construction of weights}\hfill\vspace{2.5mm}

The following result produces weights in $A_{\infty}$.

\begin{lemma}\label{pesos4}
Let $1\leq p,q <\infty$. Let $u\in A_{q}$, $v\in A_p$, and take $W=\left( \frac{u}{v}\right)^{1/p}$. Then, $W\in A_{1+\frac{q}{p}}$, and
\begin{equation*}
    [W]_{A_{1+\frac{q}{p}}} \leq [u]_{A_q}^{1/p} [v]_{A_p}^{1/p}.
\end{equation*}
\end{lemma}

\begin{proof}
By definition,
\begin{equation}\label{pesos41}
    [W]_{A_{1+\frac{q}{p}}} =\sup_Q \left(\avgint_Q \left( \frac{u}{v}\right)^{1/p}\right) \left(\avgint_Q \left( \frac{u}{v}\right)^{-\frac{1}{q}}\right)^{q/p}.
\end{equation}

Fix a cube $Q\subseteq \mathbb R^n$. To estimate the first factor in \eqref{pesos41}, in virtue of H\"older's inequality with exponent $p\geq 1$, we get that
\begin{equation}\label{pesos42}
   \avgint_Q \left( \frac{u}{v}\right)^{1/p}\leq \left(\avgint_Q u \right)^{1/p} \left(\avgint_Q v^{1-p'}\right)^{\frac{p-1}{p}},
\end{equation}
where the last term is interpreted as $\esssup_{x\in Q} v(x)^{-1}$ if $p=1$.

Similarly, to estimate the second factor in \eqref{pesos41}, in virtue of H\"older's inequality with exponent $q\geq 1$, we have that
\begin{equation}\label{pesos43}
   \avgint_Q \left( \frac{u}{v}\right)^{-\frac{1}{q}}\leq \left(\avgint_Q v \right)^{1/q}\left(\avgint_Q u^{1-q'}\right)^{\frac{q-1}{q}},
\end{equation}
where the last term is interpreted as $\esssup_{x\in Q} u(x)^{-1}$ if $q=1$.

Combining \eqref{pesos41}, \eqref{pesos42}, and \eqref{pesos43}, we obtain that
\begin{align*}
   [W]_{A_{1+\frac{q}{p}}} 
   \leq \sup_Q \left(\avgint_Q u \right)^{1/p} \left(\avgint_Q u^{1-q'}\right)^{\frac{q-1}{p}} \left(\avgint_Q v \right)^{1/p} \left(\avgint_Q v^{1-p'}\right)^{\frac{p-1}{p}} \leq [u]_{A_q}^{1/p} [v]_{A_p}^{1/p}.
\end{align*}
\end{proof}

The next lemma allows us to construct nice weights in $\widehat A_{p,\infty}$.

\begin{lemma}\label{weightsgorro}
Let $1\leq q \leq p$ and $1\leq N\in \mathbb N$, and let $w$ be a weight. For a measurable function $h\in L^1_{loc}(\mathbb R^n)$, let $v=(Mh)^{q-p}w$. If $w\in \widehat A_{q,N}$, then $v\in \widehat A_{p,N+1}$, and
\begin{equation}\label{eqweightsgorroc1}
    \Vert v\Vert_{ \widehat A_{p,N+1}} \leq \Vert w \Vert_{\widehat A_{q,N}}^{q/p}.
\end{equation}
In particular, if $w\in \widehat A_{q,\infty}$, then $v\in \widehat A_{p,\infty}$, and
\begin{equation}\label{eqweightsgorroc2}
    \Vert v\Vert_{ \widehat A_{p,\infty}} \leq 2 \Vert w \Vert_{\widehat A_{q,\infty}}.
\end{equation}
\end{lemma}

\begin{proof}
Fix $\gamma >1$. For a weight $w \in \widehat A_{q,N}$, we can find measurable functions $h_1,\dots,h_N \in L^1_{loc}(\mathbb R^n)$, parameters $\theta_1, \dots, \theta_N \in (0,1]$, with $\theta_1+ \dots + \theta_N= 1$, and a weight $u\in A_1$ such that $w = \left(\prod_{i=1}^N (Mh_i)^{\theta_i}\right) ^{1-q} u$, with $[u]_{A_1}^{1/q} \leq \gamma \Vert w \Vert_{\widehat A_{q,N}}$. 

Note that if $p=1$, then $v=w=u$, so \eqref{eqweightsgorroc1} holds. If $p>1$, then
\begin{equation*}
    v= \left((Mh)^{\frac{q-p}{1-p}}(Mh_1)^{\theta_1 \frac{1-q}{1-p}}\dots (Mh_N)^{\theta_N \frac{1-q}{1-p}}\right)^{1-p}u,
\end{equation*}
and since $\frac{q-p}{1-p}+(\theta_1+\dots+\theta_N)\frac{1-q}{1-p}=1$, we have that $v \in \widehat A_{p,N+1}$, with
\begin{equation*}
     \left \| v\right \|_{\widehat {A}_{p,N+1}}\leq  [u]_{A_1}^{1/p} \leq \gamma^{q/p} \left \| w\right \|_{\widehat{A}_{q,N}}^{q/p},
\end{equation*}
and \eqref{eqweightsgorroc1} follows letting $\gamma$ tend to $1$. If $q=1$, then $v\in \widehat A_p$, and $\left \| v\right \|_{\widehat {A}_{p}} \leq [ w]_{A_1}^{1/p}$.

Finally, if $w \in \widehat A_{q,\infty}$, then we can find a natural number $N\geq 1$ such that $w \in \widehat A_{q,N}$, and in virtue of \eqref{eqweightsgorroc1}, we get that
\begin{equation*}
     \Vert v \Vert_{\widehat {A}_{p,\infty}} \leq (N+1) \left \| v\right \|_{\widehat {A}_{p,N+1}}     \leq 2 N \left \| w\right \|_{\widehat{A}_{q,N}}^{q/p} \leq 2 N \left \| w\right \|_{\widehat{A}_{q,N}},
\end{equation*}
and taking the infimum over all such $N\geq 1$, we obtain \eqref{eqweightsgorroc2}.
\end{proof}

The following result also lets us construct nice weights in $\widehat A_{p,\infty}$.

\begin{lemma}\label{pesos5}
Let $1<p<q$ and $1\leq N\in \mathbb N$, and let $w$ be a weight. For a measurable function $h\in L^1_{loc}(\mathbb R^n)$, let $v=w^{\frac{p-1}{q-1}}(Mh)^{\frac{q-p}{q-1}}$. If $w\in \widehat A_{q,N}$, then $v\in \widehat A_{p,N}$, and
\begin{equation}\label{eqpesos5c1}
    \Vert v \Vert_{\widehat A_{p,N}} \leq c \Vert w \Vert_{\widehat A_{q,N}}^{q/p},
\end{equation}
with $c$ independent of $h$. In particular, if $w\in \widehat A_{q,\infty}$, then $v\in \widehat A_{p,\infty}$, and
\begin{equation}\label{eqpesos5c2}
    \Vert v \Vert_{\widehat A_{p,\infty}} \leq c \Vert w \Vert_{\widehat A_{q,\infty}} ^{q/p}.
\end{equation}
\end{lemma}

\begin{proof}
Fix $\gamma>1$. For a weight $w \in \widehat{A}_{q,N}$, we can find measurable functions $h_1,\dots,h_N \in L^1_{loc}(\mathbb R^n)$, parameters $\theta_1, \dots, \theta_N \in (0,1]$, with $\theta_1+ \dots + \theta_N= 1$, and a weight $u\in A_1$ such that $w = \left(\prod_{i=1}^N (Mh_i)^{\theta_i} \right)^{1-q}u$, with $[u]_{A_1}^{1/q} \leq \gamma \Vert w \Vert_{\widehat A_{q,N}}$. Thus, 
\begin{align*}
\begin{split}
    v = \left(\prod_{i=1}^N (Mh_i)^{\theta_i} \right)^{1-p} u^{\frac{p-1}{q-1}}(Mh)^{\frac{q-p}{q-1}}\defper  \left(\prod_{i=1}^N (Mh_i)^{\theta_i} \right)^{1-p} \widetilde u.
    \end{split}
\end{align*}

Applying \cite[Lemma 2.12]{cs}, we see that $\widetilde u \in {A}_1$, with $[\widetilde u]_{A_1} \leq \kappa_n \frac{q-1}{p-1}  [u]_{A_1}$, and hence, $v \in \widehat A_{p,N}$, with 
 \begin{equation*}
     \Vert v \Vert_{\widehat A_{p,N}} \leq [\widetilde u]_{A_1}^{1/p} \leq \left(\kappa_n\frac{q-1}{p-1} \right)^{1/p} [u]_{A_1}^{1/p} \leq \left(\gamma^q \kappa_n\frac{q-1}{p-1} \right)^{1/p} \Vert w \Vert_{\widehat{A}_{q,N}}^{q/p},
 \end{equation*}
 and letting $\gamma$ tend to $1$, \eqref{eqpesos5c1} holds with $c=\left(\kappa_n\frac{q-1}{p-1} \right)^{1/p}$.
 
Finally, if $w \in \widehat A_{q,\infty}$, then we can find a natural number $N\geq 1$ such that $w \in \widehat A_{q,N}$, and in virtue of \eqref{eqpesos5c1}, we get that
\begin{equation*}
     \Vert v \Vert_{\widehat {A}_{p,\infty}} \leq N \left \| v\right \|_{\widehat {A}_{p,N}} \leq c N \left \| w\right \|_{\widehat{A}_{q,N}}^{q/p} \leq c (N \left \| w\right \|_{\widehat{A}_{q,N}})^{q/p},
\end{equation*}
and taking the infimum over all such $N\geq 1$, we obtain \eqref{eqpesos5c2}.
\end{proof}

\subsection{\texorpdfstring{Some properties of $A_p^{\mathcal R}$ weights}{Some properties of ApR weights}}\hfill\vspace{2.5mm}

The following result allows us to construct $A_p^{\mathcal R}$ weights and settles a question raised in \cite[Remark 4.1.8]{thesis}.

\begin{proposition}\label{weightcombi1}
Fix an integer $m\geq 2$. Let $1\leq p_1,\dots,p_m < \infty$, and $0<\theta_1, \dots, \theta_m \leq 1$ such that $\theta_1+\dots+\theta_m=1$. Given weights $w_1 \in A_{p_1} ^{\mathcal R},\dots, w_m\in A_{p_m} ^{\mathcal R}$, the weight $w=w_1 ^{\theta_1}\dots w_m^{\theta_m}$ is in $A_{\max\{p_1,\dots,p_m\}} ^{\mathcal R}$, and 
\begin{equation*}
[w]_{A_{\max\{p_1,\dots,p_m\}} ^{\mathcal R}} \lesssim \prod_{i=1}^m [w_i]_{A_{p_i} ^{\mathcal R}}^{\theta_i}.
\end{equation*}
Moreover, for every cube $Q \subseteq \mathbb R^n$,
\begin{equation*}
    \prod_{i=1}^m \left(\int_{Q} w_i \right)^{\theta_i} \lesssim \left(\prod_{i=1}^m [ w_i]_{A_{p_i}^{\mathcal R}}^{p_i \theta_i} \right) \int_Q w.
\end{equation*}
\end{proposition}

\begin{proof}
For $i=1,\dots,m$, and  $p\perdef  \max\{p_1,\dots,p_m\}$, we know from \cite[Remark 10]{prp} that for every measurable set $F \subseteq \mathbb R^n$,
$$
\Vert M(\chi_F)\Vert_{L^{p,\infty}(w_i)} \leq 2^n 72^{n/p} [w_i]_{A_p ^{\mathcal R}} \Vert \chi_F \Vert_{L^{p,1}(w_i)}. 
$$

In virtue of Theorem~\ref{interpol}, we get that for every measurable set $E \subseteq \mathbb R^n$,
\begin{align*}
\Vert M(\chi_E)\Vert_{L^{p,\infty}(w)} & \leq 2^n 72^{n/p} m \left(\prod_{i=1}^m [w_i]_{A_{p} ^{\mathcal R}}^{\theta_i}\right) \Vert \chi_E \Vert_{L^{p,1}(w)} \\ & \leq 2^n 72^{n/p} m \left(\prod_{i \, : \, p_i \neq p} p_i ^{\theta_i} \right) \left (\prod_{i=1}^m [w_i]_{A_{p_i} ^{\mathcal R}}^{\theta_i}\right) \Vert \chi_E \Vert_{L^{p,1}(w)},
\end{align*}
and hence (see \cite[Theorem 10]{prp}),
$$
[w]_{A_{p} ^{\mathcal R}} \leq \Vert w \Vert_{A_{p} ^{\mathcal R}} \leq 2^n 72^{n/p} m p \left(\prod_{i \, : \, p_i \neq p} p_i ^{\theta_i} \right) \prod_{i=1}^m [w_i]_{A_{p_i} ^{\mathcal R}}^{\theta_i}.
$$

Now, fix a cube $Q\subseteq \mathbb R^n$. It follows from \cite[Proposition 12]{quico} that there exist measurable sets $G_1, \dots,G_m \subseteq Q$ such that for $i=1,\dots,m$,
$$
|Q| \leq 2^m |G_i| \quad \text{ and } \quad \prod_{i=1}^m w_i (G_i)^{\theta_i} \leq 2^{1-\frac{1}{m}} w(Q).
$$

Since $w_i \in A_{p_i}^{\mathcal R}$, we have that
$$
w_i (Q) \leq 2^{mp_i} \Vert w_i \Vert_{A_{p_i}^{\mathcal R}}^{p_i} w_i(G_i) \leq 2^{mp_i} p_i^{p_i} [w_i]_{A_{p_i}^{\mathcal R}}^{p_i} w_i(G_i), 
$$
so
$$
\prod_{i=1}^m w_i(Q) ^{\theta_i} \leq 2^{1-\frac{1}{m}} \left(\prod_{i=1}^m 2^{mp_i\theta_i} p_i^{p_i \theta_i} \right) \left(\prod_{i=1}^m [ w_i]_{A_{p_i}^{\mathcal R}}^{p_i \theta_i} \right) w(Q).
$$
\end{proof}

\begin{remark}\label{rmkweightcombi1}
    Our approach is different than the one presented in \cite{xuya}. In the case $p_1=\dots=p_m=1$, a better result was obtained there: for every cube $Q \subseteq \mathbb R^n$,
    $$
\prod_{i=1}^m w_i(Q) ^{\theta_i}   \leq \left(\prod_{i=1}^m [ w_i]_{A_{1}}^{\theta_i} \right) w(Q).
    $$
It can be deduced directly from the fact that for $v\in A_1$, $\avgint_Q v \leq [v]_{A_1} \essinf_{x\in Q}v(x)$.
\end{remark}

\begin{remark}
Under the hypotheses of Proposition~\ref{weightcombi1}, if for some  $1\leq \ell \leq m$, $w_1 \in A_{p_1}, \dots, w_{\ell} \in A_{p_{\ell}}$, then we get that for every cube $Q \subseteq \mathbb R^n$,
$$
\prod_{i=1}^m w_i(Q) ^{\theta_i} \leq 2^{1-\frac{1}{m}} \left(\prod_{i=1}^m 2^{mp_i\theta_i} p_i^{p_i \theta_i} \right) \left(\prod_{i=1}^{\ell} [ w_i]_{A_{p_i}}^{ \theta_i} \right) \left(\prod_{i=\ell+1}^m [ w_i]_{A_{p_i}^{\mathcal R}}^{p_i \theta_i} \right) w(Q),
$$
which provides a quantitative extension of \cite[Lemma 3.1]{crp} and \cite[Corollary 1.5]{cumoen}.
\end{remark}

\section{A dual Sawyer-type inequality}

We devote this section to the study of a novel restricted weak-type inequality for the Hardy-Littlewood maximal operator $M$. It can be interpreted as a dual version of \cite[Theorem 2]{prp}, and generalizes \cite[Lemma 2.6]{cs} and \cite[Lemma 2.5]{sergi}. It is based on our study of the case $\theta=\mu=1$ in \cite[Theorem 7]{prp}, and the proof borrows some ideas from the aforementioned sources. 

\begin{theorem}\label{thetasawyer}
Fix exponents $p > 1$ and $\frac{1}{p'} < \theta < 1$, and an integer $N\geq 1$. Given measurable functions $h_1, \dots, h_N\in L^1_{loc}(\mathbb R^n)$, parameters $\theta_1, \dots, \theta_N \in (0,1]$, with $\theta_1+ \dots + \theta_N= 1$, and a weight $w\in A_1$, consider the $\widehat{A}_{p,N}$ weights $u = \left(\prod_{i=1}^N (Mh_i)^{\theta_i} \right)^{1-p}w$ and $u_{\theta} = \left(\prod_{i=1}^N (Mh_i)^{\theta_i} \right)^{1-p}w^{\theta}$, and let $v$ be a weight such that $uv^p \in A_{\infty}$. Then, for every exponent $ \theta < \mu \leq 1$, there exists a constant $C>0$ such that for every measurable function $f$,
\begin{equation}\label{eqthetasawyer}
    \left \Vert \frac{M^{\mu}(fu_{\theta}v^{p-1})}{u_{\theta}}\right \Vert_{L^{p',\infty}(u)} \leq C \Vert f \Vert_{L^{p',1}(uv^p)}.
\end{equation}
\end{theorem}

\begin{proof}
Since $p'>1$, it is enough to establish the result for characteristic functions. Let $E\subseteq \mathbb R^n$ be a measurable set such that $0<uv^p(E)<\infty$, and take $f=\chi_E$. 

In virtue of Kolmogorov's inequality, we obtain that
\begin{align}\label{thetasawyereq1}
\left \| \frac{M^{\mu}(fu_{\theta}v^{p-1})}{u_{\theta}} \right \|_{L^{p',\infty}(u)} \leq \sup_{0<u(F)<\infty} \left \| \frac{M^{\mu}(fu_{\theta}v^{p-1})}{u_{\theta}}\chi_F \right \|_{L^{1/\theta}(u)}u(F)^{\frac{1}{p'}-\theta},
\end{align}
where the supremum is taken over all measurable sets $F\subseteq \mathbb R^n$ with $0<u(F)<\infty$. 

For one of such sets $F$, and applying Fefferman-Stein's inequality (see \cite[Lemma 1]{fefst}), we have that
\begin{align}\label{thetasawyereq2}
\begin{split}
\left \| \frac{M^{\mu}(fu_{\theta}v^{p-1})}{u_{\theta}}\chi_F \right \|_{L^{1/\theta}(u)}^{1/\theta} & = \int_F M(u_{\theta}^{1/\mu}v^{\frac{p-1}{\mu}}\chi_E )^{\mu/\theta} \left(\prod_{i=1}^N (Mh_i)^{\theta_i} \right)^{(p-1)\left(\frac{1}{\theta}-1\right)} \\ & \leq \kappa_n^{\mu / \theta} \frac{\mu}{\mu-\theta} \int_E \left(\prod_{i=1}^N (Mh_i)^{\theta_i} \right)^{\frac{1-p}{\theta}} w v^{\frac{p-1}{\theta}} M(\nu \chi_F),
\end{split}
\end{align}
with $\nu \perdef  \left(\prod_{i=1}^N (Mh_i)^{\theta_i} \right)^{(p-1)\left(\frac{1}{\theta}-1\right)}$. Note that $\nu \in A_1$, with $[\nu]_{A_1} \leq \frac{c_n}{1-(p-1)\left(\frac{1}{\theta}-1\right)}$. 

We now deal with $M(\nu \chi_F)$. Fix $x\in E$, and let $P\subseteq \mathbb R^n$ be a cube containing $x$. Then,
\begin{align}\label{thetasawyereq3}
\begin{split}
    \frac{1}{|P|}\int _P \chi_F \nu & = \frac{1}{|P|}\int _P \chi_F \nu u^{-1} u \leq \frac{1}{|P|} \|\chi_P \nu u^{-1} \|_{L^{\theta p',\infty}(u)} \|\chi_F \chi_P\|_{L^{(\theta p')',1}(u)} \\ & = \frac{\theta p'}{\theta p'-1}\left( \frac{u(P)^{\frac{1}{(\theta p')'}}}{\nu(P)} \|\chi_P \nu u^{-1} \|_{L^{\theta p',\infty}(u)}\right) \frac{\nu(P)}{|P|} \left(\frac{u(F\cap P)}{u(P)}\right)^{\frac{1}{(\theta p')'}} \\ & \leq \frac{\theta p'}{\theta p'-1} \left(\sup_Q \frac{u(Q)^{\frac{1}{(\theta p')'}}}{\nu(Q)} \|\chi_Q \nu u^{-1} \|_{L^{\theta p',\infty}(u)} \right) [\nu]_{A_1} \nu(x) M_u(\chi_F)(x)^{\frac{1}{(\theta p')'}}.
\end{split}
\end{align}

Write $\alpha \perdef (\theta p')'$ and observe that $\alpha > p$. Kolmogorov's inequality gives us that
\begin{align*}
[u \nu^{-1}]_{A_{\alpha}^{\mathcal R}(\nu)} & \perdef \sup_Q \frac{u(Q)^{1/\alpha}}{\nu(Q)} \|\chi_Q \nu u^{-1} \|_{L^{\alpha ',\infty}(u)} \\& \leq \sup_Q \sup_{G \subseteq Q} \frac{\nu(G)}{\nu(Q)} \left(\frac{u(Q)}{u(G)} \right)^{1/{\alpha}} \defper \| u \nu^{-1} \|_{A_{\alpha}^{\mathcal R}(\nu)}.
\end{align*}

Now, for a cube $Q\subseteq \mathbb R^n$, and a nonempty measurable set $G \subseteq Q$, 
\begin{align*}
    \left(\frac{u(Q)}{u(G)} \right)^{1/{\alpha}} = \left(\frac{|Q|}{|G|} \right)^{p/\alpha} \left(\frac{|G|}{|Q|} \left(\frac{u(Q)}{u(G)}\right)^{1/p} \right)^{p/{\alpha}} \leq \left(\frac{|Q|}{|G|} \right)^{p/\alpha} \| u \|_{A_p^{\mathcal R}}^{p/\alpha},
\end{align*}
and applying H\"older's inequality and Remark~\ref{rmkweightcombi1}, we deduce that
\begin{align*}
    \frac{\nu(G)}{\nu(Q)} \leq 
    \left(\prod_{i=1}^N [(Mh_i)^{(p-1)\left(\frac{1}{\theta}-1\right)}]_{A_1}^{\theta_i} \right)\prod_{i=1}^N \left( \frac{(Mh_i)^{(p-1)\left(\frac{1}{\theta}-1\right)}(G)}{(Mh_i)^{(p-1)\left(\frac{1}{\theta}-1\right)}(Q)} \right)^{\theta_i}.
\end{align*}
Note that for $N=1$, this last estimate is not necessary.

By \cite[Lemma 2.5]{cs},
$$
\frac{(Mh_i)^{(p-1)\left(\frac{1}{\theta}-1\right)}(G)}{(Mh_i)^{(p-1)\left(\frac{1}{\theta}-1\right)}(Q)} \leq \frac{c_n}{1-(p-1)\left(\frac{1}{\theta}-1\right)} \left(\frac{|G|}{|Q|} \right)^{p/\alpha},
$$
so the previous computations yield that
$$
[u \nu^{-1}]_{A_{\alpha}^{\mathcal R}(\nu)} \leq \frac{
c_n^2}{\left(1-(p-1)\left(\frac{1}{\theta}-1\right)\right)^2} \| u \|_{A_p^{\mathcal R}}^{p/\alpha},
$$
and form \eqref{thetasawyereq3} we obtain that
$M(\nu \chi_F)(x) \leq C_1 \nu(x) M_u(\chi_F)(x)^{\frac{1}{(\theta p')'}}$, with
$$
C_1 \perdef \frac{
c_n^3 \theta p' \| u \|_{A_p^{\mathcal R}}^{\frac{p}{(\theta p')'}}}{(\theta p'-1)\left(1-(p-1)\left(\frac{1}{\theta}-1\right)\right)^3} = 
c_n^3 p \left(\frac{\theta}{1+p(\theta-1)} \right)^4 \| u \|_{A_p^{\mathcal R}}^{\frac{1}{\theta}+p\left(1-\frac{1}{\theta}\right)}.
$$

Thus,
\begin{align}\label{thetasawyereq4}
\begin{split}
     &\int_E \left(\prod_{i=1}^N (Mh_i)^{\theta_i} \right)^{\frac{1-p}{\theta}} w v^{\frac{p-1}{\theta}} M(\nu \chi_F) \leq C_1 \int_E \left(\prod_{i=1}^N (Mh_i)^{\theta_i} \right)^{\frac{1-p}{\theta}} w v^{\frac{p-1}{\theta}} \nu M_u(\chi_F)^{\frac{1}{(\theta p')'}} \\ & = C_1 \int_E u v^{\frac{p-1}{\theta}} M_u(\chi_F)^{\frac{1}{(\theta p')'}} = C_1 \int_E  \left(\frac{M_u(\chi_F)}{v^p}\right)^{\frac{1}{(\theta p')'}} u v^p, 
     \end{split}
\end{align}
and applying H\"older's inequality with exponent $(\theta p')'>1$, we get that
\begin{align}\label{thetasawyereq5}
    \int_E  \left(\frac{M_u(\chi_F)}{v^p}\right)^{\frac{1}{(\theta p')'}} u v^p \leq \theta p' \left \| \frac{M_u(\chi_F)}{v^p} \right \|_{L^{1,\infty}(uv^p)}^{\frac{1}{(\theta p')'}} uv^p(E)^{\frac{1}{\theta p'}}.
\end{align}

Finally, in virtue of \cite[Theorem 1]{prp}, if $r\geq 1$ is such that $uv^p \in A_{r}^{\mathcal R}$, then
\begin{align}\label{thetasawyereq6}
\left \| \frac{M_u(\chi_F)}{v^p} \right \|_{L^{1,\infty}(uv^p)} \leq \widetilde{\mathscr E}_{r,p}^n([u]_{A_{p}^{\mathcal R}},[uv^p]_{A_{r}^{\mathcal R}}) u(F) \defper C_2 u(F),
\end{align}
where $\widetilde{\mathscr E}_{r,p}^n : [1,\infty)^2\longrightarrow [0,\infty)$ is a function that increases in each variable and depends only on $r, p$, and the dimension $n$. Note that if $v=1$, then we can take $C_2 \eqsim_{n} 2^{np} \Vert u \Vert_{A_{p}^{\mathcal R}}^p$.

Combining \eqref{thetasawyereq2}, \eqref{thetasawyereq4}, \eqref{thetasawyereq5}, and \eqref{thetasawyereq6}, we deduce that
\begin{align*}
    \left \| \frac{M^{\mu}(\chi_E u_{\theta}v^{p-1})}{u_{\theta}}\chi_F \right \|_{L^{1/\theta}(u)}^{1/\theta} \leq \kappa_n^{\mu/\theta} \frac{\mu}{\mu-\theta} C_1 \theta p' C_2^{\frac{1}{(\theta p')'}}u(F)^{\frac{1}{(\theta p')'}} uv^p(E)^{\frac{1}{\theta p'}},
\end{align*}
and from \eqref{thetasawyereq1} we obtain that
\begin{align*}
    \left \| \frac{M^{\mu}(\chi_E u_{\theta}v^{p-1})}{u_{\theta}} \right \|_{L^{p',\infty}(u)} \leq C_3 uv^p (E)^{1/p'},
\end{align*}
with $C_3\perdef \kappa_n^{\mu} \left(\frac{\mu}{\mu-\theta} C_1 \theta p' C_2^{\frac{1}{(\theta p')'}}\right)^{\theta}$, from which the desired result follows by a standard extension argument (see \cite[Page 231]{BS},
\cite[Exercise 1.4.7]{grafclas} or \cite[Appendix]{stws}).

Indeed, for an arbitrary measurable function $f$, if $\Vert f \Vert_{L^{p',1}(uv^p)}=\infty$, we are done. Take $f\in L^{p',1}(uv^p)$ and observe that, by H\"older's inequality with exponent $\mu p'>1$, 
\begin{align*}
|f|^{1/\mu} u_{\theta}^{1/\mu} v^{\frac{p-1}{\mu}} = \left(|f|^{p'}uv^p\right)^{\frac{1}{\mu p'}} \left(u w^{p(\theta -1)}\right)^{\frac{1}{\mu p}} \in L^1_{loc}(\mathbb R^n).
\end{align*}

Now, for every integer $k$, consider the set $E_k \perdef \{x\in \mathbb R^n : 2^k < |f(x)| \leq 2^{k+1}\}$. Then, $|f|^{1/\mu} \leq 2^{1/\mu} \sum_{k\in \mathbb Z} 2^{k/\mu} \chi_{E_k}$, so $M^{\mu}(fu_{\theta}v^{p-1})\leq 2 \sum_{k\in \mathbb Z} 2^k M^{\mu}(\chi_{E_k}u_{\theta}v^{p-1})$, and hence,
\begin{align*}
\left \| \frac{M^{\mu}(f u_{\theta}v^{p-1})}{u_{\theta}} \right \|_{L^{p',\infty}(u)} & \leq 2 p \sum_{k\in \mathbb Z} 2^k \left \| \frac{M^{\mu}(\chi_{E_k} u_{\theta}v^{p-1})}{u_{\theta}} \right \|_{L^{p',\infty}(u)} \\ & \leq 2 p C_3 \sum_{k\in \mathbb Z}  2^k uv^p(E_k)^{1/p'} \leq 2 p C_3 \sum_{k\in \mathbb Z}  2^k \lambda_{f}^{uv^p}(2^k)^{1/p'} 
\\ & \leq 4 p C_3 \sum_{k\in \mathbb Z} \int_{2^k}^{2^{k+1}} \lambda_{f}^{uv^p}(t)^{1/p'} 
dt \leq  4 (p-1) C_3 \Vert f \Vert_{L^{p',1}(uv^p)},
\end{align*}
concluding that \eqref{eqthetasawyer} holds with $C=4 (p-1)C_3$. 
\end{proof}

\begin{remark}\label{rmksawyer}
    The constant $C$ that we have obtained satisfies that
    \begin{align*}
    C & \leq \mathfrak{c}_n p^3 \left(\frac{\mu}{\mu-\theta} \right)^{\theta} 
    \left(\frac{\theta}{1+p(\theta-1)} \right)^{4\theta
    } 
    \widetilde{ \mathscr E}_{r,p}^n([u]_{A_{p}^{\mathcal R}},[uv^p]_{A_{r}^{\mathcal R}})^{\theta -1 + \frac{1}{p}} [u]_{A_p^{\mathcal R}}^{1+p(\theta - 1)} \\ & \defper \phi([u]_{A_{p}^{\mathcal R}},[uv^p]_{A_{r}^{\mathcal R}}).
    \end{align*}
    Since the function $\phi$ behaves well at the endpoint $p=1$, it would be interesting to study an analog of \eqref{eqthetasawyer} involving the limit spaces $L^{\infty}(uv)$ and $W(u)$ (see \cite{weakinf} and \cite[Page 385]{BS}), and its applications to upper endpoint extrapolation (see \cite{liomis,zoe,zoegr}).
\end{remark}

\section{One-variable off-diagonal extrapolation}\label{s5}

In \cite{duoextrapol}, multi-variable strong-type extrapolation theorems were obtained as corollaries of one-variable off-diagonal strong-type extrapolation schemes; that is, results in which the target space is different from the domain, both in terms of exponents and weights. In the case of multi-variable restricted weak-type extrapolation, we observe a similar phenomenon, and we can also deduce our results from one-variable off-diagonal restricted weak-type extrapolation theorems. 

\subsection{Restricted weak-type results}\hfill\vspace{2.5mm}

Let us start with the downwards extrapolation. The following result generalizes and extends \cite[Theorem 2.11]{cgs}, \cite[Theorem 2.13]{cgs}, and \cite[Theorem 4.2.14]{thesis}.

\begin{theorem}\label{offdown}
Let $0 \leq \alpha<\infty$, and let $\nu \in A_{\infty}$. Fix an integer $N \geq 1$. Given measurable functions $f$ and $g$, suppose that for some exponent $1\leq p<\infty$, and every weight $v\in \widehat A_{p,N+1}$,
\begin{equation}\label{eqoffdownh}
\Vert g \Vert_{L^{p_\alpha,\infty}(V)} \leq \psi(\Vert v\Vert_{\widehat{A}_{p,N+1}}) \Vert f \Vert_{L^{p,1}(v)},
\end{equation}
where $\frac{1}{p_\alpha} = \frac{1}{p} + \alpha$, $V = v^{p_{\alpha} / p} \nu ^{\alpha p_\alpha}$, and $\psi:[1,\infty) \longrightarrow [0,\infty)$ is an increasing function. Then, for every exponent $1\leq q \leq p$, and every weight $w\in \widehat{A}_{q,N}$, 
\begin{equation}\label{eqoffdownc}
\Vert g \Vert_{L^{q_\alpha,\infty}(W)} \leq \Psi (\Vert w\Vert_{\widehat{A}_{q,N}}) \Vert f \Vert_{L^{q,\frac{q}{p}}(w)},
\end{equation}
where $\frac{1}{q_\alpha} = \frac{1}{q} + \alpha$, $W = w^{q_{\alpha} / q} \nu ^{\alpha q_\alpha}$, and $\Psi :[1,\infty) \longrightarrow [0,\infty)$ is an increasing function. The same result is valid in the case $N=\infty$. If $q=1$, then we can take $N=0$.
\end{theorem}

\begin{proof}
Observe that if $q=p$, then there is nothing to prove, so we may assume that $q<p$. Pick a weight $w\in \widehat A_{q,N}$. We may also assume that $\Vert f \Vert_{L^{q,\frac{q}{p}}(w)} < \infty$. In particular, $f$ is locally integrable. Fix $y>0$ and $\gamma >0$. We have that
\begin{align}\label{eqoffdown1}
\begin{split}
\lambda_g ^{W} (y) & =\int_{\{|g|>y\}}W \leq \lambda ^{W} _{\mathscr Z} (\gamma y) + \int_{\{|g|>y\}} \left( \frac{\gamma y}{\mathscr Z} \right)^{p_\alpha - q_\alpha}W \defper I + II,
\end{split}
\end{align}
where $\mathscr Z \perdef (Mf)^{q/q_{\alpha}}\left(\frac{w}{\nu}\right)^{\alpha}$.

To estimate the term $I$ in \eqref{eqoffdown1}, we have that
\begin{equation}\label{eqoffdown2}
    I = \frac{(\gamma y)^{q_{\alpha}}}{(\gamma y)^{q_{\alpha}}} \lambda ^{W}_ {\mathscr Z}(\gamma y) \leq \frac{1}{(\gamma y)^{q_{\alpha}}} \Vert \mathscr Z \Vert_{L^{q_{\alpha},\infty}(W)}^{q_{\alpha}} = \frac{1}{(\gamma y)^{q_{\alpha}}} \left \Vert \frac{Mf}{U} \right \Vert_{L^{q,\infty}(w U^{q})}^{q},
\end{equation}
with $U \perdef \left(\frac{\nu}{w}\right)^{\frac{\alpha q_{\alpha}}{q}}$. Note that $U\in A_{\infty}$. Indeed, if $\alpha=0$, then $U=1$, and if $\alpha>0$, then $\widehat A_{q,N} \subseteq A_{q+\frac{1}{\alpha}}$, so in virtue of Lemma~\ref{pesos4}, $U=\left(\frac{\nu}{w}\right)^{\frac{1}{q+\frac{1}{\alpha}}} \in A_{\infty}$. Moreover, since $\nu\in A_{\infty}$, and $\frac{q_{\alpha}}{q}+\alpha q_{\alpha}=1$, $W \in A_{\infty}$, and there exists $r\geq 1$ such that $W\in A_{r}^{\mathcal R}$. If $s\geq 1$ is such that $\nu\in A_{s}^{\mathcal R}$, then we can choose $r\perdef \max \{q,s\}$, and applying Proposition~\ref{weightcombi1}, we get that $[W]_{A_{r}^{\mathcal R}}\leq \mathfrak c_n r (q[w]_{A_{q}^{\mathcal R}})^{q_{\alpha}/q}(s[\nu]_{A_{s}^{\mathcal R}})^{\alpha q_{\alpha}}$. 

In virtue of \cite[Theorem 2]{prp} and Theorem~\ref{aprgorro}, we deduce that
\begin{align}\label{eqoffdown6}
\begin{split}
    \left \Vert \frac{Mf}{U} \right \Vert_{L^{q,\infty}(w U^{q})} & \leq \mathscr E_{r,q}^n([w]_{A_{q}^{\mathcal R}},[W]_{A_{r}^{\mathcal
    R}})\Vert f \Vert_{L^{q,1}(w)} \\ & \leq \mathscr E_{r,q}^n(C N \Vert w \Vert_{\widehat A_{q,N}},  \mathfrak c_n r (qC N\Vert w \Vert_{\widehat A_{q,N}})^{q_{\alpha}/q}(s[\nu]_{A_{s}^{\mathcal R}})^{\alpha q_{\alpha}}) \Vert f \Vert_{L^{q,1}(w)} \\ & \leq p^{1-\frac{p}{q}}\mathscr E_{r,q}^n(C N \Vert w \Vert_{\widehat A_{q,N}},  \mathfrak c_n r (qC N\Vert w \Vert_{\widehat A_{q,N}})^{q_{\alpha}/q}(s[\nu]_{A_{s}^{\mathcal R}})^{\alpha q_{\alpha}})\Vert f \Vert_{L^{q,\frac{q}{p}}(w)} \\ & \defper p^{1-\frac{p}{q}} \phi_{\nu,w} \Vert f \Vert_{L^{q,\frac{q}{p}}(w)},
\end{split}
\end{align}
and combining \eqref{eqoffdown2} and \eqref{eqoffdown6}, we obtain that
\begin{equation}\label{eqoffdown7}
    I \leq \frac{1}{(\gamma y)^{q_{\alpha}}}p^{q-p} \phi_{\nu,w}^{q} \Vert f \Vert_{L^{q,\frac{q}{p}}(w)}^{q}.
\end{equation}

We proceed to estimate the term $II$ in \eqref{eqoffdown1}. Take $v \perdef (Mf)^{q-p}w$. Since $w \in \widehat A_{q,N}$, it follows from Lemma~\ref{weightsgorro} that $v \in \widehat A_{p,N+1}$, with $\Vert v \Vert_{\widehat A_{p,N+1}} \leq \Vert w \Vert_{\widehat A_{q,N}}^{q/p}$. Note that if $q=1$, then $v \in \widehat A_p$, with $\Vert v \Vert_{\widehat A_{p}} \leq [w]_{ A_{1}}^{1/p}$. Observe that
\begin{align*}
    \mathscr Z ^{q_{\alpha}-p_{\alpha}} W & = (Mf)^{q\left(1-\frac{p_{\alpha}}{q_{\alpha}}\right)}w^{\alpha(q_{\alpha}-p_{\alpha})+\frac{q_{\alpha}}{q}}\nu^{\alpha(p_{\alpha}-q_{\alpha})+\alpha q_{\alpha}} \\ & = (Mf)^{\frac{p_{\alpha}}{p}(q-p)}w^{p_{\alpha}/p} \nu^{\alpha p_{\alpha}}=v^{p_{\alpha}/p}\nu^{\alpha p_{\alpha}},
\end{align*}
so by \eqref{eqoffdownh} and the monotonicity of $\psi$, we get that
\begin{align}\label{eqoffdown4}
\begin{split}
II = \frac{(\gamma y)^{p_{\alpha}}}{(\gamma y)^{q_{\alpha}}} \int_{\{|g|>y\}} v^{p_{\alpha}/p}\nu^{\alpha p_{\alpha}} 
\leq \frac{\gamma ^{p_{\alpha}}}{(\gamma y)^{q_{\alpha}}} \psi ( \Vert w \Vert_{\widehat{A}_{q,N}}^{q/p}) ^{p_{\alpha}} \left \| f \right \|_{L^{p,1}(v)}^{p_{\alpha}},
\end{split}
\end{align}
with
\begin{align}\label{eqoffdown5}
\begin{split}
\left \| f \right \|_{L^{p,1}(v)} &= p \int_0^{\infty} \left( \int_{\{|f|>z\}}v\right)^{1/p} dz \\ & \leq p \int_0^{\infty} z^{q/p}\left( \int_{\{|f|>z\}}w\right)^{1/p} \frac{dz}{z} = \frac{p}{q} \left \| f \right \|_{L^{q,\frac{q}{p}}(w)}^{q/p}.
\end{split}
\end{align}

Combining the estimates \eqref{eqoffdown1}, \eqref{eqoffdown7}, \eqref{eqoffdown4}, and \eqref{eqoffdown5}, we conclude that
\begin{align*}
    \begin{split}
        \lambda_g^{W}(y) & \leq \frac{1}{(\gamma y)^{q_{\alpha}}} p^{q-p} \phi_{\nu,w}^{q} \Vert f \Vert_{L^{q,\frac{q}{p}}(w)}^{q} \\ & + \frac{\gamma ^{p_{\alpha}}}{(\gamma y)^{q_{\alpha}}} \left(\frac{p}{q}\right)^{p_{\alpha}} \psi ( \Vert w \Vert_{\widehat{A}_{q,N}}^{q/p}) ^{p_{\alpha}} \left \| f \right \|_{L^{q,\frac{q}{p}}(w)}^{\frac{p_{\alpha}q}{p}},
    \end{split}
\end{align*}
and taking the infimum over all $\gamma >0$ (see \cite[Lemma 3.1.1]{thesis}), it follows that 
\begin{align*}
    \begin{split}
        y^{q_{\alpha}} \lambda_g ^ {W} (y) & \leq \frac{p_{\alpha}}{p_{\alpha}-q_{\alpha}}\left( \frac{p_{\alpha}-q_{\alpha}}{q_{\alpha}}\right)^{q_{\alpha}/p_{\alpha}} \left(\frac{p}{q}\right)^{q_{\alpha}} p^{q_{\alpha}\left(2-\frac{q}{p}-\frac{p}{q}\right)} \\ & \times \phi_{\nu,w}^{q_{\alpha}\left(1-\frac{q}{p}\right)}  \psi (\Vert w \Vert_{\widehat{A}_{q,N}}^{q/p})^{q_{\alpha}} \left \| f \right \|_{L^{q,\frac{q}{p}}(w)}^{q_{\alpha}}. 
    \end{split}
\end{align*}

Finally, raising everything to the power $\frac{1}{q_{\alpha}}$ in this last expression, and taking the supremum over all $y>0$, we see that \eqref{eqoffdownc} holds, with
\begin{equation}\label{eqoffdown9}
\Psi (\xi) = \mathfrak C_{p,q}^{\alpha} \mathscr E_{r,q}^n(C N \xi,  \mathfrak c_n r (q C N\xi)^{q_{\alpha}/q}(s[\nu]_{A_{s}^{\mathcal R}})^{\alpha q_{\alpha}})^{1-\frac{q}{p}} \psi (\xi^{q/p}), \quad \xi\geq 1,
\end{equation}
where 
\begin{equation*}
   \mathfrak C_{p,q}^{\alpha}\perdef p^{3-\frac{q}{p}-\frac{p}{q}} q^{-1} \left(\frac{p_{\alpha}}{p_{\alpha}-q_{\alpha}}\right)^{1/q_{\alpha}}\left( \frac{p_{\alpha}-q_{\alpha}}{q_{\alpha}}\right)^{1/p_{\alpha}}.
\end{equation*}

It remains to discuss the case $N=\infty$. By hypothesis, we have that for every weight $v\in \widehat A_{p,\infty}$,
\begin{equation}\label{eqoffdown10}
\Vert g \Vert_{L^{p_\alpha,\infty}(V)} \leq \psi(\Vert v\Vert_{\widehat{A}_{p,\infty}}) \Vert f \Vert_{L^{p,1}(v)}.
\end{equation}
Pick a weight $w\in \widehat A_{q,\infty}$. We can find an integer $N_0 \geq 1$ such that $w\in \widehat A_{q,N_0}$, with $\Vert w \Vert_{\widehat A_{q,\infty}} \leq N_0\Vert w \Vert_{\widehat A_{q,N_0}} \leq 2 \Vert w \Vert_{\widehat A_{q,\infty}}$. From \eqref{eqoffdown10}, we deduce that for every weight $v\in \widehat A_{p,N_0+1}$,
\begin{equation*}
\Vert g \Vert_{L^{p_\alpha,\infty}(V)} \leq \psi((N_0+1)\Vert v\Vert_{\widehat{A}_{p,N_0+1}}) \Vert f \Vert_{L^{p,1}(v)} \defper \psi_{N_0}(\Vert v\Vert_{\widehat{A}_{p,N_0+1}}) \Vert f \Vert_{L^{p,1}(v)},
\end{equation*}
and applying Theorem~\ref{offdown} for $N_0$, we conclude that 
\begin{equation*}
\Vert g \Vert_{L^{q_\alpha,\infty}(W)} \leq \Psi_{N_0}(\Vert w\Vert_{\widehat{A}_{q,N_0}}) \Vert f \Vert_{L^{q,\frac{q}{p}}(w)} \leq \Psi(\Vert w\Vert_{\widehat{A}_{q,\infty}}) \Vert f \Vert_{L^{q,\frac{q}{p}}(w)},
\end{equation*}
with $\Psi_{N_0}$ as in \eqref{eqoffdown9}, and
\begin{equation}\label{eqoffdown11}
\Psi (\xi) = \mathfrak C_{p,q}^{\alpha} \mathscr E_{r,q}^n(2 C \xi,  \mathfrak c_n r (2q C \xi)^{q_{\alpha}/q}(s[\nu]_{A_{s}^{\mathcal R}})^{\alpha q_{\alpha}})^{1-\frac{q}{p}} \psi (4\xi), \quad \xi \geq 1.
\end{equation}
\end{proof}

\begin{remark}
For $N=1$ and $\alpha=0$, a version of Theorem~\ref{offdown} for sub-linear operators and $\widehat A_{p}$ weights was obtained in \cite[Theorem 2.13]{cgs}, avoiding the class $\widehat A_{p,2}$ via an interpolation argument based on Theorem~\ref{interpol} for $m=2$.
\end{remark}

We now discuss the upwards extrapolation. The following result generalizes and extends \cite[Theorem 1.6]{sergi}, \cite[Theorem 3.1]{cs}, and \cite[Theorem 4.2.18]{thesis}, using some ideas from their proofs.

\begin{theorem}\label{offup}
Let $0 \leq \alpha<\infty$, and let $\nu \in A_{\infty}$. Fix an integer $N\geq 1$. Given measurable functions $f$ and $g$, suppose that for some exponent $1\leq p<\infty$, and every weight $v\in \widehat A_{p,N}$,
\begin{equation}\label{eqoffuph}
\Vert g \Vert_{L^{p_\alpha,\infty}(V)} \leq \psi(\Vert v\Vert_{\widehat{A}_{p,N}}) \Vert f \Vert_{L^{p,1}(v)},
\end{equation}
where $\frac{1}{p_\alpha} = \frac{1}{p} + \alpha$, $V = v^{p_{\alpha} / p} \nu ^{\alpha p_\alpha}$, and $\psi:[1,\infty) \longrightarrow [0,\infty)$ is an increasing function. Then, for every finite exponent $q \geq p$, and every weight $w \in \widehat A_{q,N}$, 
\begin{equation}\label{eqoffupc}
\Vert g \Vert_{L^{q_\alpha,\infty}(W)} \leq \Psi (\Vert w\Vert_{\widehat{A}_{q,N}}) \Vert f \Vert_{L^{q,1}(w)},
\end{equation}
where $\frac{1}{q_\alpha} = \frac{1}{q} + \alpha$, $W = w^{q_{\alpha} / q} \nu ^{\alpha q_\alpha}$, and $\Psi :[1,\infty) \longrightarrow [0,\infty)$ is an increasing function. The same result is valid in the case $N=\infty$.
\end{theorem}

\begin{proof}
If $q=p$, then there is nothing to prove, so we may assume that $q>p$. Pick a weight $w \in \widehat{A}_{q,N}$. We can find measurable functions $h_1,\dots,h_N \in L^1_{loc}(\mathbb R^n)$, parameters $\theta_1, \dots, \theta_N \in (0,1]$, with $\theta_1+ \dots + \theta_N= 1$, and a weight $u\in A_1$ such that $w = \left(\prod_{i=1}^N (Mh_i)^{\theta_i} \right)^{1-q}u$, with $[u]_{A_1}^{1/q} \leq (1+\frac{1}{q}) \Vert w \Vert_{\widehat A_{q,N}}$. As usual, we may assume that $\left \| f \right \|_{L^{q,1}(w)}<\infty$. 

Fix a natural number $\varrho \geq 1$, and let $g_{\varrho} \perdef |g|\chi_{B(0,\varrho)}$, where $B(0,\varrho)$ is the ball of center $0$ and radius $\varrho$ in $\mathbb R^n$. We will prove \eqref{eqoffupc} for the pair $(f,g_\varrho)$. Since $g_{\varrho} \leq |g|$, we already know that \eqref{eqoffuph} holds for $(f,g_{\varrho})$. Fix $y>0$ such that $\lambda_{g_{\varrho}}^{W}(y)\neq 0$. If no such $y$ exists, then $\left \| g_{\varrho} \right \|_{L^{q_{\alpha},\infty}(W)}=0$ and we are done. 

In order to apply \eqref{eqoffuph}, we want to find a weight $v\in \widehat A_{p,N}$ such that $\lambda_{g_{\varrho}} ^{W}(y) \leq \lambda_{g_{\varrho}} ^{V}(y)$. We take 
\begin{align}\label{eqoffup0}
\begin{split}
    v & \perdef w_{\beta}^{\frac{p-1}{q-1}}\left(M^{\mu}(w_{\theta}w^{-\frac{1}{q'}}W^{1/q'}
    \chi_{\{|g_{\varrho}|>y\}})u^{\tau (1-\mu)}\right)^{\frac{q-p}{q-1}},
    \end{split}
\end{align}
where $w_{\beta} \perdef w u^{\beta -1}$ and $w_{\theta} \perdef w u^{\theta -1}$, with
$$\tau \perdef 1+\frac{1}{2^{n+1}[u]_{A_1}}, \quad 
\frac{1}{q'} 
<\theta < 1,
$$

$$
\beta \perdef \left \{ \begin{array}{lc}
    1, & p = 1, \\
    \frac{q-1}{p-1} - \theta \cdot \frac{q-p}{p-1}, & p > 1,
\end{array} \right. 
\quad \text{and } \quad \mu \perdef \left \{\begin{array}{lc}
    1-\frac{1-\theta}{\tau},  & p=1, \\
   1,  & p>1.
\end{array}\right.$$
For $p>1$, we also impose that $\beta \leq \tau$, or equivalently,
$$
\frac{q-1}{q-p}-\tau \cdot \frac{p-1}{q-p} \leq \theta.
$$

If $p=1$, then $v = M^{\mu}(w_{\theta}w^{-\frac{1}{q'}}W^{1/q'} \chi_{\{|g_{\varrho}|>y\}})u^{\tau (1-\mu)}$. In virtue of \cite[Lemma 3.26]{carlos}, $u^{\tau}\in A_1$, with $[u^{\tau}]_{A_1}\lesssim [u]_{A_1}$, and we can use \cite[Lemma 2.12]{cs} to deduce that $v \in A_{1}$, with 
\begin{equation}\label{eqoffup5}
    [ v ]_{A_{1}} \lesssim_n \frac{[ u ]_{{A}_{1}}}{1-\mu}  \leq \frac{\tau (1+\frac{1}{q})^{q}}{1-\theta} \Vert w \Vert_{\widehat A_{q,N}}^q \lesssim \frac{\Vert w \Vert_{\widehat A_{q,N}}^q}{1-\theta}.
\end{equation}
 
If $p>1$, then applying Lemma~\ref{pesos5} we see that $v \in \widehat A_{p,N}$, and $\Vert v \Vert_{\widehat A_{p,N}} \leq c \Vert w_{\beta} \Vert_{\widehat{A}_{q,N}}^{q/p}$, with $c\eqsim_n \left(\frac{q-1}{p-1} \right)^{1/p}$. Since $1<\beta \leq \tau$, 
$\Vert w_{\beta} \Vert_{\widehat{A}_{q,N}} \leq [u^{\beta}]_{A_1}^{1/q} 
\lesssim [u]_{A_1}^{1/q} \leq (1+\frac{1}{q}) \Vert w \Vert_{\widehat{A}_{q,N}}$, so 
\begin{equation*}
\Vert v \Vert_{\widehat A_{p,N}} \lesssim_n \left(\frac{q-1}{p-1} \right)^{1/p} \Vert w \Vert_{\widehat{A}_{q,N}} ^{q/p}.
\end{equation*}

Thus, $\Vert v \Vert_{\widehat A_{p,N}} \leq C_{\theta} \Vert w \Vert_{\widehat{A}_{q,N}} ^{q/p}$, with $C_{\theta} \eqsim_{n} \frac{1}{1-\theta}$ for $p=1$, and $C_{\theta}\eqsim  c$ for $p>1$.

Observe that 
 \begin{align*}
     \begin{split}
         V = v^{p_{\alpha} / p} \nu ^{\alpha p_\alpha} & \geq w_{\beta}^{\frac{p_{\alpha}}{p}\cdot\frac{p-1}{q - 1}}w_{\theta}^{\frac{p_{\alpha}}{p}\cdot\frac{q-p}{q-1}} w^{-\frac{p_{\alpha}}{pq'}\cdot\frac{q-p}{q-1}} W^{\frac{p_{\alpha}}{pq'}\cdot\frac{q-p}{q-1}}  u^{\tau (1-\mu)\frac{p_{\alpha}}{p}\cdot\frac{q-p}{q-1}} \nu^{\alpha p_{\alpha}} \chi_{\{|g_{\varrho}|>y\}} \\ & = w^{\gamma_1} u^{\gamma_2} \nu^{\gamma_3} = w^{q_{\alpha}/q} \nu^{\alpha q_{\alpha}} \chi_{\{|g_{\varrho}|>y\}}=W\chi_{\{|g_{\varrho}|>y\}},
     \end{split}
 \end{align*}
 since
 \begin{align*}
     \gamma_1 &\perdef \frac{p_{\alpha}}{p}\cdot\frac{p-1}{q - 1} + \frac{p_{\alpha}}{p}\cdot\frac{q-p}{q-1} - \frac{p_{\alpha}}{pq'}\cdot\frac{q-p}{q-1} + \frac{q_{\alpha}}{q}\cdot\frac{p_{\alpha}}{pq'}\cdot\frac{q-p}{q-1} \\ & = \frac{p_{\alpha}}{p}\left(1-\left(1-\frac{p}{q} \right)\left(1-\frac{q_{\alpha}}{q} \right) \right) = \frac{p_{\alpha}}{p} \cdot \frac{1+\alpha p}{1+\alpha q} = \frac{q_{\alpha}}{q} , \\ \gamma_2 & \perdef (\beta-1)\frac{p_{\alpha}}{p}\cdot\frac{p-1}{q - 1} + (\theta-1)\frac{p_{\alpha}}{p}\cdot\frac{q-p}{q-1} +  \tau (1-\mu)\frac{p_{\alpha}}{p}\cdot\frac{q-p}{q-1} \\ & = \frac{p_{\alpha}}{p} \left(\frac{p-1}{q-1}\cdot\beta+\frac{q-p}{q-1}\left(\theta +\tau (1-\mu) \right) - 1  \right) = 0, \\ \gamma_3 & \perdef \alpha p_{\alpha} + \alpha q_{\alpha} \cdot \frac{p_{\alpha}}{pq'}\cdot\frac{q-p}{q-1} = \frac{\alpha p}{1+\alpha p} + \frac{1}{1+\alpha p} - \frac{1}{1+\alpha q} = \alpha q_{\alpha},
 \end{align*}
so \eqref{eqoffuph} implies that
\begin{align}\label{eqoffup2}
\begin{split}
\lambda_{g_{\varrho}} ^{W}(y) & \leq \int_{\{|g_{\varrho}|>y\}} v^{p_{\alpha} / p} \nu ^{\alpha p_\alpha} = \lambda_{g_{\varrho}} ^{V}(y) \leq \frac{1}{y^{p_{\alpha}}} \psi(C_{\theta} \Vert w \Vert_{\widehat{A}_{q,N}}^{q/p})^{p_{\alpha}} \left \| f \right \|_{L^{p,1}(v)}^{p_{\alpha}}.  
\end{split}
\end{align}

We want to replace $\Vert f \Vert_{L^{p,1}(v)}$ by $\Vert f \Vert_{L^{q,1}(w)}$ in \eqref{eqoffup2}. Applying H\"older's inequality with exponent $\frac{q}{p}>1$, we obtain that for every $t>0$,
\begin{align}\label{eqoffup02}
\begin{split}
 \lambda_{f}^{v}(t) & = \int_{\{|f|>t\}} \left(\frac{M^{\mu}(w_{\theta}w^{-\frac{1}{q'}}W^{1/q'} \chi_{\{|g_{\varrho}|>y\}})}{w_{\theta}} \right)^{\frac{q-p}{q-1}}w \\ & \leq \frac{q}{p} w(\{|f|>t\})^{p/q} \left \| \frac{M^{\mu}(w_{\theta}w^{-\frac{1}{q'}}W^{1/q'} \chi_{\{|g_{\varrho}|>y\}})}{w_{\theta}} \right \|_{L^{q',\infty}(w)}^{\frac{q-p}{q-1}} \\ & = \frac{q}{p} w(\{|f|>t\})^{p/q} \left \| \frac{M^{\mu}(w_{\theta}U^{q-1} \chi_{\{|g_{\varrho}|>y\}})}{w_{\theta}} \right \|_{L^{q',\infty}(w)}^{\frac{q-p}{q-1}},
 \end{split}
\end{align}
with $U \perdef \left(\frac{W}{w}\right)^{1/q}=\left(\frac{\nu}{w}\right)^{\frac{\alpha}{1+\alpha q}}$. Note that if $\alpha = 0$, then $U=1$, and for $\alpha >0$, $w \in A_q^{\mathcal R} \subseteq A_{q+\frac{1}{\alpha}}$, so $U\in A_{\infty}$ by Lemma~\ref{pesos4}. 

Moreover, if $s\geq 1$ is such that $\nu\in A_{s}^{\mathcal R}$, then by Proposition~\ref{weightcombi1}, $W\in A_{r}^{\mathcal R}$, with $r\perdef \max \{q,s\}$, and $[W]_{A_{r}^{\mathcal R}}\leq \mathfrak c_n r (q[w]_{A_{q}^{\mathcal R}})^{q_{\alpha}/q}(s[\nu]_{A_{s}^{\mathcal R}})^{\alpha q_{\alpha}}$. Also, since $wU^q=W$, Theorem~\ref{thetasawyer}, Remark~\ref{rmksawyer}, and Theorem~\ref{aprgorro} give us that
\begin{align*}
\begin{split}
     \left \| \frac{M^{\mu}(w_{\theta}U^{q-1} \chi_{\{|g_{\varrho}|>y\}})}{w_{\theta}} \right \|_{L^{q',\infty}(w)} & \leq q' \phi([w]_{A_{q}^{\mathcal R}},[W]_{A_{r}^{\mathcal R}}) W(\{|g_{\varrho}|>y\})^{1/q'} \\ & \leq \phi_{\nu,w} W(\{|g_{\varrho}|>y\})^{1/q '},
\end{split}
\end{align*}
with
\begin{equation*}
    \phi_{\nu,w} \perdef q' \phi(C N \Vert w \Vert_{\widehat A_{q,N}},  \mathfrak c_n r (qC N\Vert w \Vert_{\widehat A_{q,N}})^{q_{\alpha}/q}(s[\nu]_{A_{s}^{\mathcal R}})^{\alpha q_{\alpha}}).
\end{equation*}

Plugging this bound into \eqref{eqoffup02}, we get that
\begin{equation}\label{eqoffup03}
     \lambda_{f}^{v}(t) \leq \frac{q}{p}  \phi_{\nu,w}^{\frac{q-p}{q-1}} W(\{|g_{\varrho}|>y\})^{1-\frac{p}{q}} w(\{|f|>t\})^{p/q},
\end{equation}
and hence,
\begin{align}\label{eqoffup3}
    \begin{split}
        \left \| f \right \|_{L^{p,1}(v)} &= p \int_0 ^{\infty} \lambda_{f}^{v}(t)^{1/p} dt  \\ & \leq p \left( \frac{q}{p}\right)^{1/p} \phi_{\nu,w}^{\frac{1}{p}\cdot \frac{q-p}{q-1}} 
        W(\{|g_{\varrho}|>y\})^{\frac{1}{p}-\frac{1}{q}}\int_0^{\infty} w(\{|f|>t\})^{1/q}dt \\ & = \left(\frac{p}{q}\right)^{1/p '} \phi_{\nu,w}^{\frac{1}{p}\cdot \frac{q-p}{q-1}}W(\{|g_{\varrho}|>y\})^{\frac{1}{p}-\frac{1}{q}} \Vert f \Vert_{L^{q,1}(w)}.
    \end{split}
\end{align}

Combining the estimates \eqref{eqoffup2} and \eqref{eqoffup3}, we have that
\begin{equation}\label{eqoffup4}
    \lambda_{g_{\varrho}} ^{W}(y) \leq \frac{1}{y^{p_{\alpha}}} \Psi_{\theta} (\Vert w \Vert_{\widehat{A}_{q,N}})^{p_{\alpha}} \left \| f \right \|_{L^{q,1}(w)}^{p_{\alpha}} \lambda_{g_{\varrho}} ^{W}(y) ^{1-\frac{p_{\alpha}}{q_{\alpha}}},
\end{equation}
with
\begin{align*}
    \Psi_{\theta} (\Vert w \Vert_{\widehat{A}_{q,N}}) & \perdef  \left(\frac{p}{q}\right)^{1/p '} \phi_{\nu,w}^{\frac{1}{p}\cdot \frac{q-p}{q-1}} \psi (C_{\theta} \Vert w \Vert_{\widehat{A}_{q,N}}^{q/p}).
\end{align*}

By our choice of $y$ and $g_{\varrho}$, $0<\lambda_{g_{\varrho}} ^{W}(y)\leq W(B(0,\varrho))<\infty$, so we can divide by $\lambda_{g_{\varrho}} ^{W}(y)^{1-\frac{p_{\alpha}}{q_{\alpha}}}$ in \eqref{eqoffup4} and raise everything to the power $\frac{1}{p_{\alpha}}$, obtaining that
\begin{align*}
    y\lambda_{g_{\varrho}} ^{W}(y)^{1/q_{\alpha}} & \leq \Psi_{\theta} (\Vert w \Vert_{\widehat{A}_{q,N}}) \left \| f \right \|_{L^{q,1}(w)}, 
\end{align*}
and taking the supremum over all $y>0$, we deduce \eqref{eqoffupc} for the pair $(f,g_{\varrho})$ and the function $\Psi_{\theta}$; the result for the pair $(f,g)$ follows by taking the supremum over all $\varrho\geq 1$. It remains to compute a function $\Psi$ independent of $\theta$. 

We need to control the term $\phi_{\nu,w}$ appropriately, possibly by performing crude estimates and not paying much attention to optimality. Recall that by Remark~\ref{rmksawyer},
\begin{align*}
    \phi_{\nu,w} & \lesssim_n  \frac{q^4}{q-1} \left(\frac{\mu}{\mu-\theta} \right)^{\theta} 
    \left(\frac{\theta}{1+q(\theta-1)} \right)^{4\theta
    }  
    \\ & \times \widetilde{\mathscr E}_{r,q}^n(C N \Vert w \Vert_{\widehat A_{q,N}},  \mathfrak c_n r (qC N\Vert w \Vert_{\widehat A_{q,N}})^{q_{\alpha}/q}(s[\nu]_{A_{s}^{\mathcal R}})^{\alpha q_{\alpha}})^{1/q} C N \Vert w \Vert_{\widehat A_{q,N}},
\end{align*}
and note that
\begin{align*}
\left(\frac{\mu}{\mu-\theta} \right)^{\theta} 
    \left(\frac{\theta}{1+q(\theta-1)} \right)^{4\theta} \leq \frac{1}{\mu-\theta} 
    \left(\frac{1}{1+q(\theta-1)} \right)^{4} \defper \digamma (\theta,\mu).
    \end{align*}

If $p=1$, then for $\theta_* \perdef 1-\frac{1}{5q}$, 
$$
\inf_{\frac{1}{q'}<\theta < 1} \digamma (\theta,\mu) = \tau ' \digamma (\theta_*,1) 
\eqsim_n q [u]_{A_1} \lesssim q \Vert w \Vert_{\widehat A_{q,N}}^q, \quad \text{ and } \quad C_{\theta_*} \eqsim_n \frac{1}{1-\theta_*} \eqsim q.
$$

If $p>1$, then $\mu=1$, and we argue as follows: if $\tau \geq \frac{p'}{(5q)'}-\frac{4}{5(p-1)}$, then we choose $\theta =\theta_*$; otherwise, we choose $\theta =\frac{q-1}{q-p}-\tau \cdot \frac{p-1}{q-p}$.  
Thus, 
$$
\digamma (\theta,1) \lesssim_n \max \left\{ q, \frac{q-p}{p-1} \right\} [u]_{A_1} \lesssim \max \left\{ q, \frac{q-p}{p-1} \right\} \Vert w \Vert_{\widehat A_{q,N}}^q.
$$

In conclusion, \eqref{eqoffupc} holds, with $\Psi (  \xi )$ defined for $\xi \geq 1$ as
\begin{equation}\label{eqoffup7}
    \mathfrak c_n \mathfrak C_{p,q} ^0 (C 
    N \xi^{q+1})^{\frac{1}{p}\cdot \frac{q-p}{q-1}} \widetilde{\mathscr E}_{r,q}^n(C N \xi,  \mathfrak c_n r (q C N \xi)^{q_{\alpha}/q}(s[\nu]_{A_{s}^{\mathcal R}})^{\alpha q_{\alpha}})^{\frac{1}{pq}\cdot \frac{q-p}{q-1}} \psi (\mathfrak c_n \mathfrak C_{p,q}^1 \xi^{q/p}),
\end{equation}
where 
$$
\mathfrak C_{p,q}^0 \perdef \left( \frac{p}{q}\right)^{1/p '} \left(\frac{ q^4}{q-1} \times \left \{\begin{array}{lc}
   q,  & p=1, \\ 
   \max \left\{ q, \frac{q-p}{p-1} \right\} ,  & p>1,
\end{array}\right \}\right) ^{\frac{1}{p}\cdot \frac{q-p}{q-1}}, \quad \text{ and }
$$
$$
\mathfrak C_{p,q}^1 \perdef \left \{\begin{array}{lc}
   q,  & p=1, \\ 
   \left(\frac{q-1}{p-1} \right)^{1/p},  & p>1.
\end{array}\right.
$$

Now, we discuss the case $N=\infty$. By hypothesis, we have that for every weight $v\in \widehat A_{p,\infty}$, \eqref{eqoffdown10} holds. Pick a weight $w\in \widehat A_{q,\infty}$. We can find an integer $N_0 \geq 1$ such that $w\in \widehat A_{q,N_0}$, with $\Vert w \Vert_{\widehat A_{q,\infty}} \leq N_0\Vert w \Vert_{\widehat A_{q,N_0}} \leq 2 \Vert w \Vert_{\widehat A_{q,\infty}}$. From \eqref{eqoffdown10}, we deduce that for every weight $v\in \widehat A_{p,N_0}$,
\begin{equation*}
\Vert g \Vert_{L^{p_\alpha,\infty}(V)} \leq \psi(N_0\Vert v\Vert_{\widehat{A}_{p,N_0}}) \Vert f \Vert_{L^{p,1}(v)} \defper \psi_{N_0}(\Vert v\Vert_{\widehat{A}_{p,N_0}}) \Vert f \Vert_{L^{p,1}(v)},
\end{equation*}
and applying Theorem~\ref{offup} for $N_0$, we conclude that 
\begin{equation*}
\Vert g \Vert_{L^{q_\alpha,\infty}(W)} \leq \Psi_{N_0}(\Vert w\Vert_{\widehat{A}_{q,N_0}}) \Vert f \Vert_{L^{q,1}(w)} \leq \Psi(\Vert w\Vert_{\widehat{A}_{q,\infty}}) \Vert f \Vert_{L^{q,1}(w)},
\end{equation*}
with $\Psi_{N_0}$ as in \eqref{eqoffup7}, and $\Psi (  \xi )$ defined for $\xi \geq 1$ as
\begin{equation}\label{eqoffup8}
\mathfrak c_n \mathfrak C_{p,q}^{0} (C( 
2\xi)^{q+1})^{\frac{1}{p}\cdot \frac{q-p}{q-1}} \widetilde{\mathscr E}_{r,q}^n(2C \xi,  \mathfrak c_n  r (2q C \xi)^{q_{\alpha}/q}(s[\nu]_{A_{s}^{\mathcal R}})^{\alpha q_{\alpha}})^{\frac{1}{pq}\cdot \frac{q-p}{q-1}} \psi (\mathfrak c_n \mathfrak C_{p,q}^{1} (2\xi)^{q/p}).
\end{equation}
\end{proof}

\subsection{Weak-type results}\hfill\vspace{2.5mm}

The arguments and techniques we have used to produce restricted weak-type extrapolation schemes can be easily modified to cover weak-type extrapolation in a context much broader than that of $A_p$ weights. We begin by defining the classes of functions that will play their role in what follows.

\begin{definition}\label{defapsigma}
    Given exponents $1\leq p < \infty$ and $0<\sigma<\infty$, a positive function $w$ belongs to the class $A_{p}^{\sigma}$ if there exist a weight $v \in A_1$, and a positive function $u$ such that $u^{\sigma}\in A_1$, and $w= v^{1-p} u$. Equivalently, $w\in A_p^{\sigma}$ if, and only if $w^{\sigma}\in A_{1+\sigma(p-1)}$, and we associate to this class of functions the constant given by
\begin{equation*}
    [ w ]_{ A_{p}^{\sigma}} \perdef  [w^{\sigma}]_{A_{1+\sigma(p-1)}}.
\end{equation*}

Similarly, given an index $1 \leq N \in \mathbb N \mathcup \{ \infty \}$, a positive function $w$ belongs to the class $\widehat A_{p,N}^{\sigma}$ if $w^{\sigma}\in \widehat A_{1+\sigma(p-1),N}$, and we associate to this class of functions the constant given by
\begin{equation*}
    \Vert w \Vert_{ \widehat A_{p,N}^{\sigma}} \perdef  \Vert w^{\sigma} \Vert_{\widehat A_{1+\sigma(p-1),N}}.
\end{equation*}
\end{definition}

\begin{remark}
Note that in virtue of \cite[Theorem 2.2]{cun}, if $\sigma \geq 1$, then $w\in A_p^{\sigma}$ if, and only if $w \in A_p \cap RH_{\sigma}$, and this last condition was extrapolated in \cite{auma,cuma,duoextrapol,harb}. In \cite{multiap2,multiap1}, a slightly different notation was used. There, they worked with $A_{p,r}$, with $0<r<\infty$, which is the class of positive functions $w$ such that $w^r \in A_{1+\frac{r}{p'}}$. We have that $w \in A_{p,r}$ if, and only if $w^p \in A_p^{\sigma}$, with $\sigma = \frac{r}{p}$. 
\end{remark}

The next result allows us to construct functions in the previous classes. It is an adaptation of Lemma~\ref{weightsgorro}.

\begin{lemma}\label{pesosmulti}
Let $1\leq q < p$, and fix $\sigma,\varsigma>0$ such that $0\leq \frac{\varsigma(q-1)}{\sigma(p-1)}\leq 1$. For a measurable function $h\in L^1_{loc}(\mathbb R^n)$, and $w \in A_q^{\varsigma}$, let $v=(Mh)^{1-p+\frac{\varsigma}{\sigma}(q-1)}w^{\varsigma/\sigma}$.
\begin{enumerate}
\item[($a$)] If $q>1$, then $v\in A_p^{\sigma}$, and
\begin{equation*}
    [ v]_{ A_p^{\sigma}} \leq c [ w ]_{A_{q}^{\varsigma}}^{1+\frac{\sigma(p-1)}{\varsigma(q-1)}},
\end{equation*}
with $c$ independent of $h$.
\item[($b$)] If $q=1$, then $v\in \widehat A_p^{\sigma}$, and
\begin{equation*}
    \Vert v\Vert_{ \widehat A_p^{\sigma}} \leq [w]_{A_1^{\varsigma}}^{\frac{1}{1 +\sigma(p-1)}}.
\end{equation*}
\end{enumerate}
\end{lemma}

\begin{proof}
To prove (a), since $w^{\varsigma}\in A_{1+\varsigma(q-1)}$, we can find positive, measurable functions $\varpi_0$ and $\varpi_1$ such that $w^{\varsigma}=\varpi_0^{\varsigma(1-q)}\varpi_1^{\varsigma}$, with
$$
\varpi_0 \in A_1, \quad [\varpi_0]_{A_1} \leq \mathfrak c_{n} ^{1+\frac{1}{\varsigma(q-1)}}[w]_{A_q^{\varsigma}}^{\frac{1}{\varsigma(q-1)}}, \quad \varpi_1^{\varsigma} \in A_1, \quad \text{ and } \quad [\varpi_1^{\varsigma}]_{A_1} \leq \mathfrak c_{n}^{1+\varsigma(q-1)} [w]_{A_q^{\varsigma}}.
$$
The details on the construction of such functions are available in \cite[Theorem 4.2]{david} and \cite[Lemma 3.18]{carlos}. Write
\begin{equation*}
    v=\left((Mh)^{1-\frac{\varsigma(q-1)}{\sigma(p-1)}}\varpi_{0}^{\frac{\varsigma(q-1)}{\sigma(p-1)}}\right)^{1-p}\varpi_1^{\varsigma/\sigma}\defper \widetilde{\varpi}_0^{1-p}\varpi_1^{\varsigma/\sigma}.
\end{equation*}
In virtue of \cite[Lemma 2.12]{cs}, $\widetilde{\varpi}_{0} \in A_1$, with $[\widetilde{\varpi}_{0}]_{A_1} \lesssim_n  \frac{\sigma(p-1)}{\varsigma(q-1)} [\varpi_{0}]_{A_1}$. Hence, applying \cite[Theorem 4.2]{david}, $v^{\sigma}\in A_{1+\sigma(p-1)}$, with 
\begin{equation*}
    [v]_{A_{p}^{\sigma}}\leq [\widetilde{\varpi}_{0}]_{A_1}^{\sigma(p-1)}[\varpi_{1}^{\varsigma}]_{A_1} \leq \mathfrak c_n^{1 + \varsigma (q-1)} \left( \mathfrak c_{n} ^{1+\frac{1}{\varsigma(q-1)}} \frac{\sigma(p-1)}{\varsigma(q-1)} \right)^{\sigma(p-1)}  [ w ]_{A_{q}^{\varsigma}}^{1+\frac{\sigma(p-1)}{\varsigma(q-1)}}.
\end{equation*}

To prove (b), observe that $v=(Mh)^{1-p} \varpi_1^{\varsigma/\sigma}$, so $v^{\sigma}\in \widehat A_{1+\sigma(p-1)}$, and
\begin{equation*}
     \left \| v\right \|_{\widehat{A}_{p}^{\sigma}}\leq   [\varpi_1^{\varsigma}]_{A_1}^{\frac{1}{1 +\sigma(p-1)}}=  [w]_{A_1^{\varsigma}}^{\frac{1}{1 +\sigma(p-1)}}.
\end{equation*}
\end{proof}

We can now present a weak-type version of the downwards extrapolation in Theorem~\ref{offdown}, working with the class $A_p^{\sigma}$.

\begin{theorem}\label{weakoffdown}
Fix $0 \leq \alpha<\infty$, and let $\nu$ be a positive, measurable function. Given measurable functions $f$ and $g$, suppose that for some exponents $1< p<\infty$ and $\sigma>0$, and every $v\in A_{p}^{\sigma}$,
\begin{equation}\label{eqweakoffdownh}
\Vert g \Vert_{L^{p_\alpha,\infty}(V)} \leq \psi([ v]_{{A}_{p}^{\sigma}}) \Vert f \Vert_{L^{p}(v)},
\end{equation}
where $\frac{1}{p_\alpha} = \frac{1}{p} + \alpha$, $V = v^{p_{\alpha} / p} \nu ^{\alpha p_\alpha}$, and $\psi:[1,\infty) \longrightarrow [0,\infty)$ is an increasing function. Then, for every exponent $1< q \leq p$, and every $w\in {A}_{q}^{\varsigma}$, 
\begin{equation}\label{eqweakoffdownc}
\Vert g \Vert_{L^{q_\alpha,\infty}(W)} \leq \Psi ([w]_{{A}_{q}^{\varsigma}}) \Vert f \Vert_{L^{q}(w)},
\end{equation}
where $\varsigma = \frac{\sigma p}{\sigma p+(1-\sigma)q}$, $\frac{1}{q_\alpha} = \frac{1}{q} + \alpha$, $W = w^{q_{\alpha} / q} \nu ^{\alpha q_\alpha}$, and $\Psi :[1,\infty) \longrightarrow [0,\infty)$ is an increasing function.
\end{theorem}

\begin{proof}
We will prove this statement adapting the proof of Theorem~\ref{offdown}. If $q=p$, then there is nothing to prove, so we may assume that $q<p$. Pick $w\in A_{q}^{\varsigma}$. We may also assume that $\Vert f \Vert_{L^{q}(w)} < \infty$. Fix $y>0$ and $\gamma >0$. We have that
\begin{align}\label{eqweakoffdown1}
\begin{split}
\lambda_g ^{W} (y) & =\int_{\{|g|>y\}}W \leq \lambda ^{W} _{\mathscr Z} (\gamma y) + \int_{\{|g|>y\}} \left( \frac{\gamma y}{\mathscr Z} \right)^{p_\alpha - q_\alpha}W \defper I + II,
\end{split}
\end{align}
where 
$$\mathscr Z \perdef \left(\frac{v^{p_{\alpha}/p}\nu^{\alpha p_{\alpha}}}{W}\right)^{\frac{1}{q_{\alpha}-p_{\alpha}}} \quad \text{ and } \quad v \perdef M(|f|^{\frac{q}{1+\varsigma(q-1)}}w^{\frac{1-\varsigma}{1+\varsigma(q-1)}})^{1-p+\frac{\varsigma}{\sigma}(q-1)}w^{\varsigma/\sigma}.$$

To estimate the term $I$ in \eqref{eqweakoffdown1}, writing $U \perdef \left(\frac{\nu^{\alpha q}}{w^{(1+\alpha q)\varsigma-1}}\right)^{\frac{1}{(1+ \alpha q)(1+\varsigma (q-1))}}$, we obtain, after some involved computations, that
\begin{align}\label{eqweakoffdown2}
\begin{split}
    I  & \leq \frac{\Vert \mathscr Z \Vert_{L^{q_{\alpha},\infty}(W)}^{q_{\alpha}}}{(\gamma y)^{q_{\alpha}}}   = \frac{1}{(\gamma y)^{q_{\alpha}}} \left \Vert \frac{M(|f|^{\frac{q}{1+\varsigma(q-1)}}w^{\frac{1-\varsigma}{1+\varsigma(q-1)}})}{U} \right \Vert_{L^{1+\varsigma(q-1),\infty}(w^{\varsigma} U^{1+\varsigma(q-1)})}^{1+\varsigma(q-1)} \\ & \leq \frac{1}{(\gamma y)^{q_{\alpha}}} \left \Vert M(|f|^{\frac{q}{1+\varsigma(q-1)}}w^{\frac{1-\varsigma}{1+\varsigma(q-1)}}) \right \Vert_{L^{1+\varsigma(q-1)}(w^{\varsigma})}^{1+\varsigma(q-1)} \\ & \leq  \frac{\mathfrak c_n^{1+\varsigma(q-1)} }{(\gamma y)^{q_{\alpha}}} \left(\left(1+\varsigma(q-1)\right)' \right)^{1+\varsigma(q-1)}[w]_{A_q^{\varsigma}}^{1+\frac{1}{\varsigma (q-1)}} \left \Vert f \right \Vert_{L^{q}(w )}^{q},
    \end{split}
\end{align}
where in the last inequality we have used the classical Buckley's bound for the Hardy-Littlewood maximal operator $M$ in \cite[Theorem 2.5]{buckley} (see  \cite[Theorem 3.11]{carlos}). It is worth mentioning that we need the second inequality because an optimal weak-type version of the Sawyer-type inequality in \cite[Theorem 2]{prp} is not known.

To estimate the term $II$ in \eqref{eqweakoffdown1}, since $0< \frac{\varsigma(q-1)}{\sigma(p-1)}<1$, it follows from Lemma~\ref{pesosmulti} that $v \in  A_{p}^{\sigma}$, with 
$$[ v]_{A_{p}^{\sigma}} \leq \mathfrak C_0 [ w ]_{A_{q}^{\varsigma}}^{1+\frac{\sigma(p-1)}{\varsigma(q-1)}} \quad \text{ and } \quad \mathfrak C_0 \perdef \mathfrak c_n^{1 + \varsigma (q-1)} \left( \mathfrak c_{n} ^{1+\frac{1}{\varsigma(q-1)}} \frac{\sigma(p-1)}{\varsigma(q-1)} \right)^{\sigma(p-1)},$$ so by \eqref{eqweakoffdownh}, we get that
\begin{align}\label{eqweakoffdown4}
\begin{split}
II & \leq \frac{\gamma ^{p_{\alpha}}}{(\gamma y)^{q_{\alpha}}} \Vert g \Vert_{L^{p_{\alpha},\infty}(v^{p_{\alpha}/p}\nu^{\alpha p_{\alpha}})}^{p_{\alpha}} \leq \frac{\gamma ^{p_{\alpha}}}{(\gamma y)^{q_{\alpha}}} \psi (\mathfrak C_0 [ w ]_{A_{q}^{\varsigma}}^{1+\frac{\sigma(p-1)}{\varsigma(q-1)}}) ^{p_{\alpha}} \left \| f \right \|_{L^{p}(v)}^{p_{\alpha}},
\end{split}
\end{align}
with
\begin{align}\label{eqweakoffdown5}
\begin{split}
    \Vert f \Vert_{L^{p}(v)}&= \left(\int_{\mathbb R^n} |f|^{p} M(|f|^{\frac{q}{1+\varsigma(q-1)}}w^{\frac{1-\varsigma}{1+\varsigma(q-1)}})^{1-p+\frac{\varsigma}{\sigma}(q-1)}w^{\varsigma/\sigma}\right)^{1/p} \\ & \leq \left(\int_{\mathbb R^n} |f|^{p} |f|^{q\left(1-\frac{p}{q}\right)} w^{(1-\varsigma)\left(1-\frac{p}{q}\right)+\varsigma/\sigma}\right)^{1/p} = \Vert f \Vert_{L^{q}(w)}^{q/p}.
\end{split}
\end{align}

Finally, if we argue as in the last steps of the proof of Theorem~\ref{offdown}, we can combine \eqref{eqweakoffdown2}, \eqref{eqweakoffdown4}, and  \eqref{eqweakoffdown5} to conclude that \eqref{eqweakoffdownc} holds, with
\begin{equation*}
\Psi (\xi) = \mathfrak C_1 \xi^{\left(\frac{1}{q}-\frac{1}{p}\right)\left(1+\frac{1}{\varsigma (q-1)}\right)}\psi (\mathfrak C_0 \xi^{1+\frac{\sigma(p-1)}{\varsigma(q-1)}}), \quad \xi\geq 1,
\end{equation*}
where 
\begin{equation}\label{constantc}
   \mathfrak C_1 \perdef \mathfrak C \left( \mathfrak c_{n}\left(1+\varsigma(q-1)\right)' \right)^{\left(\frac{1}{q}-\frac{1}{p}\right)(1+\varsigma(q-1))}, \quad \mathfrak C \perdef \left(\frac{p_{\alpha}}{p_{\alpha}-q_{\alpha}}\right)^{1/q_{\alpha}}\left( \frac{p_{\alpha}-q_{\alpha}}{q_{\alpha}}\right)^{1/p_{\alpha}}.
\end{equation}
\end{proof}

\begin{remark}\label{rmkrubiodefrancia}
We can use Rubio de Francia's iteration algorithm to improve Theorem~\ref{weakoffdown}, producing a better function $\Psi$. Indeed, for $q>1$, we can take 
\begin{equation*}
    v \perdef \mathscr R(|f|^{\frac{q}{1+\varsigma(q-1)}}w^{\frac{1-\varsigma}{1+\varsigma(q-1)}})^{1-p+\frac{\varsigma}{\sigma}(q-1)}w^{\varsigma/\sigma},
\end{equation*}
where for a measurable function $h\in L^{1+\varsigma(q-1)}(w^{\varsigma})$,
\begin{equation*}
    \mathscr R h \perdef \sum_{k=0}^{\infty} \frac{M^k (|h|)}{2^k \Vert M \Vert_{L^{1+\varsigma(q-1)}(w^{\varsigma})}^k}
\end{equation*}
is the Rubio de Francia's iteration algorithm (see \cite{duoextrapol,rdfrubio}). In virtue of \cite[Lemma 2.2]{duoextrapol}, we have that $\Vert \mathscr R h \Vert_{L^{1+\varsigma(q-1)}(w^{\varsigma})}\leq 2 \Vert h \Vert_{L^{1+\varsigma(q-1)}(w^{\varsigma})}$, $|h|\leq \mathscr R h$, and $\mathscr R h \in A_1$, with $[\mathscr R h]_{A_1}\leq 2 \Vert M \Vert_{L^{1+\varsigma(q-1)}(w^{\varsigma})} \leq 2\mathfrak c_n \left(1+\varsigma(q-1)\right)' [w]_{A_q^{\varsigma}}^{\frac{1}{\varsigma (q-1)}}$. Moreover, applying \cite[Lemma 2.1]{duoextrapol}, we obtain that $v \in A_{p}^{\sigma}$, with $[v]_{A_{p}^{\sigma}}\leq \widetilde {\mathfrak C}_0 [w]_{A_{q}^{\varsigma}}^{\frac{\sigma(p-1)}{\varsigma(q-1)}}$ and $\widetilde {\mathfrak C}_0 \perdef \left(2\mathfrak c_n \left(1+\varsigma(q-1)\right)' \right)^{\sigma(p-1)-\varsigma(q-1)}$. Hence, we can rewrite the proof of Theorem~\ref{weakoffdown} to conclude that \eqref{eqweakoffdownc} holds, with
\begin{equation*}
\Psi (\xi) = 2^{\left(\frac{1}{q}-\frac{1}{p}\right)(1+\varsigma(q-1))} \mathfrak C \psi (\widetilde{\mathfrak C}_0 \xi^{\frac{\sigma(p-1)}{\varsigma(q-1)}}), \quad \xi\geq 1.
\end{equation*}
\end{remark}

\begin{remark}\label{rmkmixedextrapolfix}
Note that given $0<r \leq p$, if in \eqref{eqweakoffdownh} we replace $\Vert f \Vert_{L^{p}(v)}$ by $\Vert f \Vert_{L^{p,r}(v)}$, then we can replace estimate \eqref{eqweakoffdown5} by
\begin{align*}
    \Vert f \Vert_{L^{p,r}(v)} \leq p ^{1/r} \left(\int_0 ^{\infty} t^{\frac{r q}{p}} \lambda_{f}^{w}(t)^{r/p} \frac{dt}{t}\right)^{1/r} = \left(\frac{p}{q} \right)^{1/r} \Vert f \Vert_{L^{q, \frac{r q}{p}}(w)}^{q/p},
\end{align*}
and follow the proof of Theorem~\ref{weakoffdown}, using that
\begin{equation*}
    \Vert f \Vert_{L^{q}(w)}\leq \left( \frac{r}{p}\right)^{\frac{p-r}{r q}}\Vert f \Vert_{L^{q,\frac{r q}{p}}(w)}
\end{equation*}
in \eqref{eqweakoffdown2} (see \cite[Proposition 1.4.10]{grafclas}), to conclude that 
\begin{equation*}
   \Vert g \Vert_{L^{q_\alpha,\infty}(W)} \leq \left(\frac{p}{q} \right)^{1/r} \left( \frac{r}{p}\right)^{\left(\frac{p}{r}-1\right)\left(\frac{1}{q}-\frac{1}{p}\right)}\Psi ([w]_{{A}_{q}^{\varsigma}}) \Vert f \Vert_{L^{q,\frac{r q}{p}}(w)}.
\end{equation*}
\end{remark}

We can extrapolate down to the endpoint $q=1$ by following the same argument in the proof of Theorem~\ref{weakoffdown} and using in \eqref{eqweakoffdown2} the classical Sawyer-type inequality for the Hardy-Littlewood maximal operator (see \cite[Theorem 1.4]{CUMP} and \cite[Theorem 2]{prp}).

\begin{theorem}\label{nextoffdown}
Fix $0 \leq \alpha<\infty$, and let $\nu$ be a positive, measurable function. Given measurable functions $f$ and $g$, suppose that for some exponents $1< p<\infty$ and $\sigma>0$, and every function $v\in \widehat A_{p}^{\sigma}$,
\begin{equation*}
\Vert g \Vert_{L^{p_\alpha,\infty}(V)} \leq \psi(\Vert v \Vert_{\widehat A_{p}^{\sigma}}) \Vert f \Vert_{L^{p,1}(v)},
\end{equation*}
where $\frac{1}{p_\alpha} = \frac{1}{p} + \alpha$, $V = v^{p_{\alpha} / p} \nu ^{\alpha p_\alpha}$, and $\psi:[1,\infty) \longrightarrow [0,\infty)$ is an increasing function. Then, for every function $w\in {A}_{1}^{\varsigma}$ such that $W\in A_{r}^{\mathcal R}$ for some $r\geq 1$, 
\begin{equation*}
\Vert g \Vert_{L^{\frac{1}{1 + \alpha},\infty}(W)} \leq \Psi_r ([ w ]_{{A}_{1}^{\varsigma}},[W]_{A_{r}^{\mathcal R}}) \Vert f \Vert_{L^{1,\frac{1}{p}}(w)},
\end{equation*}
where $\varsigma = \frac{\sigma p}{1+\sigma (p-1)}$, $W = w^{\frac{1}{1 + \alpha}} \nu ^{\frac{\alpha}{1 + \alpha}}$, and $\Psi_r :[1,\infty) ^2\longrightarrow [0,\infty)$ is a function increasing in each variable.
\end{theorem}

It is worth mentioning that the following conjecture would allow us to prove Theorem~\ref{nextoffdown} for an arbitrary exponent $1\leq q < p$.

   \begin{conjecture}\label{offsawyer}
        Fix exponents $q \geq 1$ and $\varsigma>0$ (or $0<\varsigma \leq 1$), and write $\varrho = 1+\varsigma(q-1)$. Let $u$ and $v$ be positive, measurable functions such that $u^{\varsigma}\in A_{\varrho}^{\mathcal R}$ and $u^{\varsigma}v^{\varrho}\in A_{r}^{\mathcal R}$ for some $r\geq 1$. Then, there exists a function $\phi :[1,\infty)^2 \longrightarrow [0,\infty)$, increasing in each variable, such that for every measurable function $f$, 
\begin{equation*}
   \left \Vert \frac{M(|f|^{q/\varrho}u^{\frac{1-\varsigma}{\varrho}})}{v} \right \Vert_{L^{\varrho,\infty}(u^{\varsigma} v^{\varrho})}^{\varrho} \leq \phi([u^{\varsigma}]_{A_{\varrho}^{\mathcal R}},[u^{\varsigma}v^{\varrho}]_{A_{r}^{\mathcal R}}) \Vert f \Vert_{L^{q,1}(u)}^q.
\end{equation*}
    \end{conjecture}


\begin{remark}
If $q=1$ or $\varsigma=1$, then Conjecture~\ref{offsawyer} follows from \cite[Theorem 2]{prp} and \cite[Lemma 3]{prp}. In general, for $h\in L^1_{loc}(\mathbb R^n)$ and $w \in A_{\varrho}^{\mathcal R}$, 
    $$
Mh (x) \leq [w]_{A_{\varrho}^{\mathcal R}} \sup_{Q \ni x} \frac{\Vert h \chi_Q \Vert_{L^{\varrho,1 }(w)}}{w (Q)^{1/\varrho}} \defper [w]_{A_{\varrho}^{\mathcal R}} M_{L^{1,\frac{1}{\varrho}}(w)}(|h|^{\varrho})(x)^{1/\varrho}, \quad x \in \mathbb R^n,
    $$
    so the desired result may require to better understand the operator $M_{L^{1,\frac{1}{\varrho}}(w)}$ in the context of Sawyer-type inequalities. For $s \geq 1$, $M_{L^{1,s}(\mathbb R^n)}$ was studied in \cite{bast,lene,stein}.
\end{remark}

Next, we study a weak-type version of the upwards extrapolation in Theorem~\ref{offup} for the class $A_p^{\sigma}$.

\begin{theorem}\label{weakoffup}
Fix $0 \leq \alpha<\infty$, and let $\nu$ be a positive, measurable function. Given measurable functions $f$ and $g$, suppose that for some exponents $1 \leq p \leq q <\infty$ and $\varsigma >0$, and every $v\in A_{p}^{\sigma}$,
\begin{equation}\label{eqweakoffuph}
\Vert g \Vert_{L^{p_\alpha,\infty}(V)} \leq \psi([ v]_{{A}_{p}^{\sigma}}) \Vert f \Vert_{L^{p}(v)},
\end{equation}
where $\sigma = \frac{\varsigma q}{(1-\varsigma)p+\varsigma q}$, $\frac{1}{p_\alpha} = \frac{1}{p} + \alpha$, $V = v^{p_{\alpha} / p} \nu ^{\alpha p_\alpha}$, and $\psi:[1,\infty) \longrightarrow [0,\infty)$ is an increasing function. Then, for every $w\in {A}_{q}^{\varsigma}$, 
\begin{equation}\label{eqweakoffupc}
\Vert g \Vert_{L^{q_\alpha,\infty}(W)} \leq \Psi ([w]_{{A}_{q}^{\varsigma}}) \Vert f \Vert_{L^{q,p}(w)},
\end{equation}
where $\frac{1}{q_\alpha} = \frac{1}{q} + \alpha$, $W = w^{q_{\alpha} / q} \nu ^{\alpha q_\alpha}$, and $\Psi :[1,\infty) \longrightarrow [0,\infty)$ is an increasing function.
\end{theorem}

\begin{proof}
We will adapt the proof of Theorem~\ref{offup}. To do so, we have to choose an appropriate function $v\in A_{p}^{\sigma}$, and find suitable replacements for estimates \eqref{eqoffup02}, \eqref{eqoffup03}, and \eqref{eqoffup3} to control $\Vert f \Vert_{L^{p}(v)}$ by $\Vert f \Vert_{L^{q,p}(w)}$, and keep track of the changes in the constants involved. 

If $q=p$, then there is nothing to prove, so we may assume that $q>p$. As usual, pick $w\in A_{q}^{\varsigma}$ and choose positive, measurable functions $\varpi_0$ and $\varpi_1$ such that $w^{\varsigma}=\varpi_0^{\varsigma(1-q)}\varpi_1^{\varsigma}$, with
$$
\varpi_0 \in A_1, \quad [\varpi_0]_{A_1} \leq \mathfrak c_{n}  ^{1+\frac{1}{\varsigma(q-1)}}[w]_{A_q^{\varsigma}}^{\frac{1}{\varsigma(q-1)}}, \quad \varpi_1^{\varsigma} \in A_1, \quad \text{ and } \quad [\varpi_1^{\varsigma}]_{A_1} \leq \mathfrak c_{n}^{1+\varsigma(q-1)}  [w]_{A_q^{\varsigma}}.
$$

Fix a natural number $\varrho \geq 1$, and let $g_{\varrho} \perdef |g|\chi_{B(0,\varrho)}\chi_{\{W\leq \varrho\}}$. It is enough to prove \eqref{eqweakoffupc} for the pair $(f,g_\varrho)$. We know that \eqref{eqweakoffuph} holds for $(f,g_{\varrho})$ because $g_{\varrho} \leq |g|$. Fix $y>0$ such that $\lambda_{g_{\varrho}}^{W}(y)\neq 0$. If no such $y$ exists, then $\left \| g_{\varrho} \right \|_{L^{q_{\alpha},\infty}(W)}=0$ and we are done. Since $W\chi_{\{|g_{\varrho}|>y\}}\in L^1(\mathbb R^n)$, instead of \eqref{eqoffup0}, we can take
\begin{equation*}
    v \perdef \left \{ \begin{array}{lr}
   \varpi_1^{1/q}M(\varpi_0^{-1}W\chi_{\{|g_{\varrho}|>y\}})^{1/q'}, & \varsigma \geq 1, \, p = 1,  \\
    & \\

   \varpi_1^{\frac{1-\varsigma}{q}}M(w^{1/q}W^{1/q'}\varpi_1^{\frac{\varsigma -1}{q}}\chi_{\{|g_{\varrho}|>y\}}), & \varsigma < 1, \, p = 1,  \\
   &
    \\ 
    w^{\frac{p-1}{q-1}}M(w^{\frac{1}{q} \cdot \frac{q'}{(1+\varsigma(q-1))'}}W^{\frac{1}{(1+\varsigma(q-1))'}}\chi_{\{|g_{\varrho}|>y\}})^{\left(1-\frac{p}{q} \right)(1+\varsigma(q-1))'}, & p>1.
\end{array} \right.
\end{equation*}

Note that $(1+\varsigma(q-1))'=1+\frac{1}{\varsigma(q-1)}$, so for $p>1$, $\frac{\sigma}{\varsigma}\left(\frac{p-1}{q-1} +\frac{(1+\varsigma(q-1))(q-p)}{ q (q-1)}\right)=1$, and for $p=1$, $\frac{\sigma}{\varsigma q}+\frac{\sigma}{q'}=1$ and $\frac{\sigma(1-\varsigma)}{\varsigma q}+\sigma=1$, so in virtue of \cite[Lemma 2.12]{cs}, $v\in A_{p}^{\sigma}$, with 
\begin{equation*}
    [v]_{A_{p}^{\sigma}}\leq \mathfrak C_0 [w]_{A_q^{\varsigma}}^{1+\frac{\sigma(p-1)}{\varsigma(q-1)}} \quad \text{and} \quad \mathfrak C_0 \perdef \left \{ \begin{array}{lr} 
    \mathfrak c_n^{1+\varsigma(q-1)} (1+\varsigma(q-1)), 
    & \varsigma \geq 1, \, p = 1,  \\
    & \\
     \mathfrak c_n^{1+\varsigma(q-1)} \frac{1+\varsigma(q-1)}{1-\varsigma}, 
    & \varsigma < 1, \, p = 1,  \\ & \\
    \mathfrak c_n ^{1+\varsigma(q-1)+\sigma(p-1)+\frac{\sigma(p-1)}{\varsigma(q-1)}} \frac{\varsigma (q-1)}{\sigma(p-1)}, & p > 1.
\end{array} \right.
\end{equation*}
Observe that $V\geq w^{\frac{p_{\alpha}}{p}\left(\frac{p-1}{q-1}+\frac{q-p}{q(q-1)}\right)}W^{\frac{p_{\alpha}}{p}\left(1-\frac{p}{q}\right)}\nu^{\alpha p_{\alpha}}\chi_{\{|g_{\varrho}|>y\}} = W\chi_{\{|g_{\varrho}|>y\}}$.

Now, estimate \eqref{eqoffup02} for $p>1$ should be replaced by
\begin{align*}
 v(\{|f|>t\}) & \leq w(\{|f|>t\})^{p/q} \left \| \frac{M(w^{\frac{\varsigma}{1+\varsigma(q-1)}}W^{\frac{\varsigma (q-1)}{1+\varsigma(q-1)}}\chi_{\{|g_{\varrho}|>y\}})^{\frac{1+\varsigma (q-1)}{\varsigma q}}}{w} \right \|_{L^{q'}(w)}^{\frac{q-p}{q-1}},
\end{align*}
for $p = 1$ and $\varsigma \geq 1$ by
\begin{align*}
 v(\{|f|>t\}) & \leq q w(\{|f|>t\})^{1/q} \left \| \frac{M(\varpi_0^{-1}W\chi_{\{|g_{\varrho}|>y\}})}{\varpi_0^{-q}\varpi_1} \right \|_{L^{1,\infty}(w)}^{1/q'},
\end{align*}
and for $p = 1$ and $0<\varsigma < 1$ by
\begin{align*}
 v(\{|f|>t\}) & \leq w(\{|f|>t\})^{1/q} \left \| \frac{M(w^{1/q}W^{1/q'}\varpi_1^{\frac{\varsigma -1}{q}}\chi_{\{|g_{\varrho}|>y\}})}{w\varpi_1^{\frac{\varsigma -1}{q}}} \right \|_{L^{q'}(w)}.
\end{align*}

Since $w^{\varsigma}\in A_{1+\varsigma(q-1)}$, $w^{\frac{1}{1-q}}\in A_{1+\frac{1}{\varsigma(q-1)}}$, with $[w^{\frac{1}{1-q}}]_{A_{1+\frac{1}{\varsigma(q-1)}}}=[w]_{A_q^{\varsigma}}^{\frac{1}{\varsigma(q-1)}}$, and by the classical Buckley's bound for $M$ (see \cite[Theorem 2.5]{buckley}, \cite[Theorem 3.11]{carlos}),
\begin{align*}
\begin{split}
    \left \| M(w^{\frac{\varsigma}{1+\varsigma(q-1)}}W^{\frac{\varsigma (q-1)}{1+\varsigma(q-1)}}\chi_{\{|g_{\varrho}|>y\}}) \right \|_{L^{1+\frac{1}{\varsigma(q-1)}}(w^{\frac{1}{1-q}})} & \leq \mathfrak c_n (1+\varsigma(q-1))[w]_{A_q^{\varsigma}} \\ & \times W(\{|g_{\varrho}|>y\})^{\frac{\varsigma(q-1)}{1+\varsigma(q-1)}}.
    \end{split}
\end{align*}

Similarly, for $0<\varsigma <1$, since $\varpi_0^{1-q}\varpi_1^{\varsigma} \in A_q$, $\varpi_0 \varpi_1^{\frac{\varsigma}{1-q}}\in A_{q'}$, with $[\varpi_0 \varpi_1^{\frac{\varsigma}{1-q}}]_{A_{q'}}=[\varpi_0^{1-q}\varpi_1^{\varsigma}]_{A_q}^{\frac{1}{q-1}}$, so
\begin{align*}
    \left \| M(w^{1/q}W^{1/q'}\varpi_1^{\frac{\varsigma -1}{q}}\chi_{\{|g_{\varrho}|>y\}}) \right \|_{L^{q'}(\varpi_0 \varpi_1^{\frac{\varsigma}{1-q}})} \leq \mathfrak c_n^{q+\frac{1}{\varsigma}+\varsigma(q-1)} q [w]_{A_q^{\varsigma}}^{1+\frac{1}{\varsigma}}
W(\{|g_{\varrho}|>y\})^{1/q'}.
\end{align*}

Also, for $\varsigma \geq 1$, $w=(w^{\varsigma})^{1/\varsigma}\in A_{1+\varsigma(q-1)}$, 
 so in virtue of \cite[Theorem 2]{prp} (see also \cite[Theorem 1.4]{CUMP}),
\begin{align*}
\left \| \frac{M(\varpi_0^{-1}W\chi_{\{|g_{\varrho}|>y\}})}{\varpi_0^{-q}\varpi_1} \right \|_{L^{1,\infty}(w)}\leq \mathscr E_{\varsigma,q}^n([w]_{A_q^{\varsigma}}) W(\{|g_{\varrho}|>y\}),
\end{align*}
where $\mathscr E_{\varsigma,q}^n : [1,\infty)\longrightarrow [0,\infty)$ is an increasing function that depends only on $q, \varsigma$, and the dimension $n$. Once again, the required optimal weak-type versions of \cite[Theorem 2]{prp} and \cite[Theorem 7]{prp} are unavailable.

Now, we can replace \eqref{eqoffup03} by
\begin{align*}
\begin{split}
    v (\{|f|>t\}) &\leq 
 \phi([w]_{A_q^{\varsigma}}) W(\{|g_{\varrho}|>y\})^{1-\frac{p}{q}} w(\{|f|>t\})^{p/q},
    \end{split}
\end{align*}
and \eqref{eqoffup3} by
\begin{align}\label{eqweakoffup1}
\begin{split}
   \left \| f\right \|_{L^{p}(v)} & \leq \left(\frac{p}{q} \phi([w]_{A_q^{\varsigma}}) \right)^{1/p} 
W(\{|g_{\varrho}|>y\})^{\frac{1}{p}-\frac{1}{q}}   \left \| f \right \|_{L^{q,p}(w)},
   \end{split}
\end{align}
with
\begin{equation*}
   \phi([w]_{A_q^{\varsigma}}) \perdef \left \{ \begin{array}{lr} 
    q \mathscr E_{\varsigma,q}^n([w]_{A_q^{\varsigma}})^{1/q'}, 
    & \varsigma \geq 1, \, p = 1,  \\
    & \\
     \mathfrak c_n^{q+\frac{1}{\varsigma}+\varsigma(q-1)} q [w]_{A_q^{\varsigma}}^{1+\frac{1}{\varsigma}}, 
    & \varsigma < 1, \, p = 1,  \\ & \\
    \left(\mathfrak c_n (1+\varsigma(q-1))[w]_{A_q^{\varsigma}}\right)^{\left(1-\frac{p}{q}\right)\left(1+\frac{1}{\varsigma(q-1)}\right)}, & p > 1.
\end{array} \right.
\end{equation*}

Finally, if we follow the proof of Theorem~\ref{offup} performing the previous changes and keeping track of the constants, we conclude that \eqref{eqweakoffupc} holds, with
\begin{align*}
   \Psi( \xi) & =  \left(\frac{p}{q} \phi(\xi) \right)^{1/p}
\psi ( \mathfrak C_0 \xi ^{1+\frac{\sigma(p-1)}{\varsigma(q-1)}}), \quad \xi \geq 1.
\end{align*}
\end{proof}

\begin{remark}\label{rmkrdfup}
Alternatively, if $1\leq p <q < \infty$, then we can take 
\begin{equation*}
    v \perdef w^{\frac{p-1}{q-1}} \mathscr R '(w^{\frac{1}{q} \cdot \frac{q'}{(1+\varsigma(q-1))'}}W^{\frac{1}{(1+\varsigma(q-1))'}}\chi_{\{|g_{\varrho}|>y\}})^{\left(1-\frac{p}{q} \right)(1+\varsigma(q-1))'},
\end{equation*}
where for a measurable function $h\in L^{1+\frac{1}{\varsigma(q-1)}}(w^{\frac{1}{1-q}})$,
\begin{equation*}
    \mathscr R'h \perdef \sum_{k=0}^{\infty} \frac{M^k (|h|)}{2^k\Vert M \Vert_{L^{1+\frac{1}{\varsigma(q-1)}}(w^{\frac{1}{1-q}})}^k}
\end{equation*}
is the Rubio de Francia's iteration algorithm (see \cite{duoextrapol,rdfrubio}). In virtue of \cite[Lemma 2.2]{duoextrapol}, we have that $|h|\leq \mathscr R'h$, $\Vert \mathscr R'h \Vert_{L^{1+\frac{1}{\varsigma(q-1)}}(w^{\frac{1}{1-q}})}\leq 2 \Vert h \Vert_{L^{1+\frac{1}{\varsigma(q-1)}}(w^{\frac{1}{1-q}})}$, and $\mathscr R'h \in A_1$, with $[\mathscr R'h]_{A_1}\leq 2 \Vert M \Vert_{L^{1+\frac{1}{\varsigma(q-1)}}(w^{\frac{1}{1-q}})} \leq 2\mathfrak c_n (1+\varsigma(q-1))[w]_{A_q^{\varsigma}}$. 
Moreover, applying \cite[Lemma 2.1]{duoextrapol}, we get that $v \in A_{p}^{\sigma}$, with $[v]_{A_{p}^{\sigma}}\leq \left(2\mathfrak c_n (1+\varsigma(q-1))\right)^{1-\frac{\sigma(p-1)}{\varsigma(q-1)}}[w]_{A_q^{\varsigma}}$. Hence, we can argue as before to conclude that \eqref{eqweakoffupc} holds, with
\begin{align*}
   \Psi( \xi) & =  \left(\frac{p}{q}  \right)^{1/p} 2^{{\left(\frac{1}{p}-\frac{1}{q}\right)\left(1+\frac{1}{\varsigma(q-1)}\right)}} \psi (\left(2 \mathfrak c_n (1+\varsigma(q-1))\right)^{1-\frac{\sigma(p-1)}{\varsigma(q-1)}}\xi), \quad \xi \geq 1.
\end{align*}
\end{remark}

\begin{remark}\label{rmkmixedextrapolfixup}
Note that given $r>0$, if in \eqref{eqweakoffuph} we replace $\Vert f \Vert_{L^{p}(v)}$ by $\Vert f \Vert_{L^{p,r}(v)}$, then we can replace estimate \eqref{eqweakoffup1} by
\begin{equation*}
    \Vert f \Vert_ {L^{p,r}(v)} \leq  \left(\frac{p}{q} \right)^{1/r}  \phi([w]_{A_q^{\varsigma}}) ^{1/p} 
W(\{|g_{\varrho}|>y\})^{\frac{1}{p}-\frac{1}{q}} \left \| f \right \|_{L^{q,r}(w)},
\end{equation*}
and follow the proof of Theorem~\ref{weakoffup} to conclude that 
\begin{equation*}
  \Vert g \Vert_{L^{q_\alpha,\infty}(W)}  \leq \left(\frac{p}{q} \right)^{\frac{1}{r}-\frac{1}{p}}\Psi ([w]_{{A}_{q}^{\varsigma}}) \left \| f \right \|_{L^{q,r}(w)}.
\end{equation*}
\end{remark}

A more familiar presentation of Theorem~\ref{weakoffup} follows assuming that $0<\sigma \leq 1$, in which case we can fix an exponent $q>p$ and choose $\varsigma \perdef \frac{\sigma p}{\sigma p+(1-\sigma)q}>0$. This restriction on $\sigma$ appears in $A_{\vec P}$ extrapolation.

\begin{corollary}\label{corollweakoffup}
Fix $0 \leq \alpha<\infty$, and let $\nu$ be a positive, measurable function. Given measurable functions $f$ and $g$, suppose that for some exponents $1 \leq p<\infty$ and $0<\sigma \leq 1$, and every $v\in A_{p}^{\sigma}$,
\begin{equation*}
\Vert g \Vert_{L^{p_\alpha,\infty}(V)} \leq \psi([ v]_{{A}_{p}^{\sigma}}) \Vert f \Vert_{L^{p}(v)},
\end{equation*}
where $\frac{1}{p_\alpha} = \frac{1}{p} + \alpha$, $V = v^{p_{\alpha} / p} \nu ^{\alpha p_\alpha}$, and $\psi:[1,\infty) \longrightarrow [0,\infty)$ is an increasing function. Then, for every finite exponent $q \geq p$, and every $w\in {A}_{q}^{\varsigma}$, 
\begin{equation*}
\Vert g \Vert_{L^{q_\alpha,\infty}(W)} \leq \Psi ([w]_{{A}_{q}^{\varsigma}}) \Vert f \Vert_{L^{q,p}(w)},
\end{equation*}
where $\varsigma = \frac{\sigma p}{\sigma p+(1-\sigma)q}$, $\frac{1}{q_\alpha} = \frac{1}{q} + \alpha$, $W = w^{q_{\alpha} / q} \nu ^{\alpha q_\alpha}$, and $\Psi :[1,\infty) \longrightarrow [0,\infty)$ is an increasing function.
\end{corollary}

\section{Multi-variable extrapolation}\label{s6}

This section presents our theorems on multi-variable mixed and restricted weak-type extrapolation. 

\subsection{Main results}\hfill\vspace{2.5mm}

 The following general extrapolation scheme is a combination of Theorems \ref{offdown}, \ref{offup}, and \ref{weakoffdown}, and Corollary~\ref{corollweakoffup}.
 
\begin{theorem}\label{multiextrapoldowntotal}
Fix $1 \leq N_1,\dots,N_m \in \mathbb N \mathcup \{ \infty \}$, and $0\leq \ell_{\mathcal R} \leq \ell \leq \ell_{\mathfrak M} \leq  m$. Given measurable functions $f_1, \dots, f_m$, and $g$, suppose that for some exponents $1\leq p_1,\dots,p_m < \infty$, $\frac{1}{p}=\frac{1}{p_1}+\dots+\frac{1}{p_m}$, and all weights $v_i \in \widehat{A}_{p_i,N_i+1}$, $i=1,\dots,\ell_{\mathcal R}$, $v_i \in \widehat{A}_{p_i,N_i}$, $i=\ell_{\mathcal R}+1,\dots,\ell$, and $v_i \in {A}_{p_i}$, $i=\ell+1,\dots,m$,
\begin{align}\label{eqmultiextrapoldowntotalh}
\begin{split}
  \left \| g \right \|_{L^{p,\infty}(v_1^{p/p_1} \dots v_m^{p/p_m})} & \leq \varphi (\Vert v_1 \Vert_{\widehat{A}_{p_1,N_1+1}},
  \dots
  ,\Vert v_{\ell_{\mathcal R}+1} \Vert_{\widehat{A}_{p_{\ell_{\mathcal R}+1},N_{\ell_{\mathcal R}+1}}}, \dots, [ v_{\ell+1} ]_{{A}_{p_{\ell+1}}},\dots
  ) \\ & \times \left(\prod_{i=1}^{\ell} \left \| f_i \right \|_{L^{p_i,1}(v_i)} \right) \left( \prod_{i=\ell+1}^m \left \| f_i \right \|_{L^{p_i}(v_i)}\right),
  \end{split}
\end{align}
where $\varphi:[1,\infty)^m\longrightarrow [0,\infty)$ is a function increasing in each variable.
Then, for all finite exponents $1\leq q_1 \leq p_1, \dots, q_{\ell_{\mathcal R}+1}\geq p_{\ell_{\mathcal R}+1},\dots,1< q_{\ell+1} \leq p_{\ell+1},\dots,q_{\ell_{\mathfrak M}+1}\geq p_{\ell_{\mathfrak M}+1},\dots$, $\frac{1}{q}=\frac{1}{q_1}+\dots+\frac{1}{q_m}$, and all weights $w_i \in \widehat{A}_{q_i,N_i}$, $i=1,\dots,\ell$, and $w_i \in {A}_{q_i}$, $i=\ell+1,\dots,m$, 
\begin{align}\label{eqmultiextrapoldowntotalc}
\begin{split}
        \left \| g \right \|_{L^{q,\infty}(w_1^{q/q_1}\dots w_m^{q/q_m})} & \leq \Phi (\Vert w_1 \Vert_{\widehat{A}_{q_1,N_1}},\dots,[ w_{\ell+1} ]_{{A}_{q_{\ell+1}}},\dots) \\ & \times \left( \prod_{i=1}^{\ell} \left \| f_i \right \|_{L^{q_i,\min  \left \{1,\frac{q_i}{p_i} \right\}}(w_i)}\right) \left(\prod_{i=\ell+1}^{m} \left \| f_i \right \|_{L^{q_i,\min  \{p_i,q_i \}}(w_i)}\right),
\end{split}
\end{align}
where $\Phi:[1,\infty)^m\longrightarrow [0,\infty)$ is a function increasing in each variable. If for some $1\leq i \leq \ell_{\mathcal R}$, $q_i=1$, then we can take $N_i=0$.
\end{theorem}

\begin{proof}
   We are going to prove this theorem in $m$ consecutive steps, one for each variable. The argument is simple, but the notation is appalling. 
   
   At step $1\leq i \leq m$, we take $p^{(i)}\perdef p_i$, $q^{(i)}\perdef q_i$,  $f^{(i)}\perdef f_i$, $g^{(i)}\perdef g$, $v^{(i)}\perdef v_i$, $w^{(i)}\perdef w_i$,
    $$\alpha^{(i)}\perdef \left(\sum_{1\leq j < i}  \frac{1}{q_j} \right) + \left(\sum_{i<j\leq m} \frac{1}{p_j}\right), \quad \text{ and } \quad \nu^{(i)}\perdef \left( \prod_{1\leq j < i} w_j^{\frac{1}{\alpha^{(i)}q_j}} \right) \left( \prod_{i < j \leq m} v_j^{\frac{1}{\alpha^{(i)} p_j}} \right).$$
    
If $1\leq i \leq \ell_{\mathcal R}$, we pick $N^{(i)}\perdef N_i$, and
\begin{align*}
    \begin{split}
\psi^{(i)}(\xi) & \perdef \left( \prod_{1\leq j < i} 
\mathscr E_j(\dots,\Vert w_{i-1} \Vert_{\widehat{A}_{q_{i-1},N_{i-1}}},\xi,\Vert v_{i+1} \Vert_{\widehat{A}_{p_{i+1},N_{i+1}+1}},
\cdots)
\right) \\ & \times 
\varphi (\dots,4\Vert w_{i-1} \Vert_{\widehat{A}_{q_{i-1},N_{i-1}}}^{\mathfrak e_{i-1}},\xi,\Vert v_{i+1} \Vert_{\widehat{A}_{p_{i+1},N_{i+1}+1}},
\cdots
  ) \\ & \times \left(\prod_{1\leq j < i} \left \| f_j \right \|_{L^{q_j,\frac{q_j}{p_j}}(w_j)}\right) \left(\prod_{i<j \leq \ell} \left \| f_j \right \|_{L^{p_j,1}(v_j)} \right) \prod_{\ell<j\leq m} \left \| f_j \right \|_{L^{p_j}(v_j)}, \quad \xi \geq 1,
\end{split}
\end{align*}
and apply Theorem~\ref{offdown}. Here, for $1\leq j \leq \ell_{\mathcal R}$, if $N_j<\infty$, then $\mathfrak e_j \perdef \frac{q_j}{p_j}$, and $\mathscr E_j$ is as in \eqref{eqoffdown9}, and if $N_j=\infty$, then $\mathfrak e_j \perdef 1$, and $\mathscr E_j$ is as in \eqref{eqoffdown11}, taking into account Proposition~\ref{weightcombi1} to handle $\nu^{(i)}$; that is,
\begin{equation*}
[\nu^{(i)}]_{A_{\max\{q_1,\dots,q_{i-1},p_{i+1},\dots,p_m\}} ^{\mathcal R}} \leq \mathfrak C_i \left( \prod_{1\leq j < i} [w_j]_{A_{q_j} ^{\mathcal R}}^{\frac{1}{\alpha^{(i)}q_j}} \right) \left( \prod_{i < j \leq m} [v_j]_{A_{p_j} ^{\mathcal R}}^{\frac{1}{\alpha^{(i)} p_j}} \right), \quad \mathfrak C_i \geq 1.
\end{equation*}
Alternatively, one could also use \cite[Lemma 4.1.11]{thesis}.

If $\ell_{\mathcal R}+1\leq i \leq \ell$, we pick $N^{(i)}\perdef N_i$, and
\begin{align*}
    \begin{split}
\psi^{(i)}(\xi) & \perdef \left( \prod_{1\leq j \leq \ell_{\mathcal R}} \mathscr E_j(\dots,\Vert w_{i-1} \Vert_{\widehat{A}_{q_{i-1},N_{i-1}}},\xi,\Vert v_{i+1} \Vert_{\widehat{A}_{p_{i+1},N_{i+1}}},\cdots)\right) \\ & \times \left(\prod_{\ell_{\mathcal R}< j < i}  \widetilde{\mathscr E}_j(\dots,\Vert w_{i-1} \Vert_{\widehat{A}_{q_{i-1},N_{i-1}}},\xi,\Vert v_{i+1} \Vert_{\widehat{A}_{p_{i+1},N_{i+1}}},\cdots) \right) \\ & \times \varphi(
\cdots, \mathfrak c_n \mathfrak C_{p_{i-1},q_{i-1}}^1\left( 2\Vert w_{i-1} \Vert_{\widehat{A}_{q_{i-1},N_{i-1}}}\right)^{q_{i-1}/p_{i-1}},\xi,\Vert v_{i+1} \Vert_{\widehat{A}_{p_{i+1},N_{i+1}}},\cdots
  ) \\ & \times \left(\prod_{1\leq j < i} \left \| f_j \right \|_{L^{q_j,\min  \left \{1,\frac{q_j}{p_j} \right\}}(w_j)} \right) \left(\prod_{i<j \leq \ell} \left \| f_j \right \|_{L^{p_j,1}(v_j)} \right) \prod_{\ell <j\leq m} \left \| f_j \right \|_{L^{p_j}(v_j)}, \quad \xi\geq 1,
\end{split}
\end{align*}
and apply Theorem~\ref{offup}. Here, for $\ell_{\mathcal R}+1\leq j \leq \ell$, if $N_j<\infty$, then $\widetilde{\mathscr E}_j$ is as in \eqref{eqoffup7}, and if $N_j=\infty$, then  $\widetilde{\mathscr E}_j$ is as in \eqref{eqoffup8}, handling $\nu^{(i)}$ as before.

If $\ell+1\leq i \leq \ell_{\mathfrak M}$, we pick $\sigma^{(i)} \perdef 1 \defper \varsigma^{(i)}$, and
\begin{align*}
    \begin{split}
\psi^{(i)}(\xi) & \perdef \left( \prod_{1\leq j \leq \ell_{\mathcal R}} \mathscr E_j(\cdots,[ w_{i-1}]_{{A}_{q_{i-1}}},\xi,[ v_{i+1} ]_{{A}_{p_{i+1}}},\dots)\right) \\ & \times \left(\prod_{\ell_{\mathcal R}< j \leq \ell}  \widetilde{\mathscr E}_j(\cdots,[ w_{i-1}]_{{A}_{q_{i-1}}},\xi,[ v_{i+1} ]_{{A}_{p_{i+1}}},\dots) \right) \\ & \times \left( \prod_{\ell< j < i} 2^{1-\frac{q_j}{p_j}}\mathfrak C^{(j)} \right) 
\varphi(
\cdots,\widetilde{\mathfrak C}_0 ^{(i-1)}[ w_{i-1} ]_{{A}_{q_{i-1}}}^{\frac{p_{i-1}-1}{q_{i-1}-1}} ,\xi,[ v_{i+1} ]_{{A}_{p_{i+1}}},\dots
  ) \\ & \times \left(\prod_{1\leq j \leq \ell} \left \| f_j \right \|_{L^{q_j,\min  \left \{1,\frac{q_j}{p_j} \right\}}(w_j)} \right) \left(\prod_{\ell< j < i} \left \| f_j \right \|_{L^{q_j}(w_j)} \right) \prod_{i< j \leq m} \left \| f_j \right \|_{L^{p_j}(v_j)}, \quad \xi \geq 1,
\end{split}
\end{align*}
and apply Theorem~\ref{weakoffdown} and Remark~\ref{rmkrubiodefrancia} to get better constants.

Finally, if $\ell_{\mathfrak M}+1\leq i \leq m$, we pick $\sigma^{(i)} \perdef 1 \defper \varsigma^{(i)}$, and for $\xi \geq 1$,
\begin{align*}
    \begin{split}
\psi^{(i)}(\xi) & \perdef \left( \prod_{1\leq j \leq \ell_{\mathcal R}} \mathscr E_j(\cdots,[ w_{i-1}]_{{A}_{q_{i-1}}},\xi,[ v_{i+1} ]_{{A}_{p_{i+1}}},\dots)\right) \\ & \times \left(\prod_{\ell_{\mathcal R}< j \leq \ell}  \widetilde{\mathscr E}_j(\cdots,[ w_{i-1}]_{{A}_{q_{i-1}}},\xi,[ v_{i+1} ]_{{A}_{p_{i+1}}},\dots) \right) \\ & \times \left( \prod_{\ell< j \leq \ell_{\mathfrak M}} 2^{1-\frac{q_j}{p_j}}\mathfrak C^{(j)} \right) \left( \prod_{\ell_{\mathfrak M}< j < i}\left(\frac{p_j}{q_j}  \right)^{1/p_j} 2^{q_j '{\left(\frac{1}{p_j}-\frac{1}{q_j}\right)}} \right)\\ & \times 
\varphi(
\cdots,\left(2 \mathfrak c_n q_{i-1} \right)^{1-\frac{p_{i-1}-1}{q_{i-1}-1}}[ w_{i-1} ]_{{A}_{q_{i-1}}},\xi,[ v_{i+1} ]_{{A}_{p_{i+1}}},\dots
  ) \\ & \times \left(\prod_{1\leq j \leq \ell} \left \| f_j \right \|_{L^{q_j,\min  \left \{1,\frac{q_j}{p_j} \right\}}(w_j)} \right) \left(\prod_{\ell< j < i} \left \| f_j \right \|_{L^{q_j,\min  \{p_j,q_j \}}(w_j)}\right)  \prod_{i< j \leq m} \left \| f_j \right \|_{L^{p_j}(v_j)},
\end{split}
\end{align*}
and apply Corollary~\ref{corollweakoffup} and Remark~\ref{rmkrdfup} to get better constants.

Hence, \eqref{eqmultiextrapoldowntotalc} holds, with
\begin{align*}
    \begin{split}
\Phi( \vec \xi) & = \left( \prod_{1\leq j \leq \ell_{\mathcal R}} \mathscr E_j( \vec \xi)\right) \left(\prod_{\ell_{\mathcal R}< j \leq \ell}  \widetilde{\mathscr E}_j( \vec \xi)\right) \\ & \times \left( \prod_{\ell< j \leq \ell_{\mathfrak M}} 2^{1-\frac{q_j}{p_j}}\mathfrak C^{(j)} \right)  \left( \prod_{\ell_{\mathfrak M}< j \leq m}\left(\frac{p_j}{q_j}  \right)^{1/p_j} 2^{q_j '{\left(\frac{1}{p_j}-\frac{1}{q_j}\right)}} \right)\\ & \times 
\varphi (4\xi_1^{\mathfrak e_{1}},
\dots,\mathfrak c_n \mathfrak C_{p_{\ell_{\mathcal R}+1},q_{\ell_{\mathcal R}+1}}^1\left( 2\xi_{\ell_{\mathcal R}+1}\right)^{q_{\ell_{\mathcal R}+1}/p_{\ell_{\mathcal R}+1}},\dots,\\ & \hspace{27pt}\widetilde{\mathfrak C}_0 ^{(\ell+1)}\xi_{\ell+1}^{\frac{p_{\ell+1}-1}{q_{\ell+1}-1}}, \dots, \left(2 \mathfrak c_n q_{\ell_{\mathfrak M}+1} \right)^{1-\frac{p_{\ell_{\mathfrak M}+1}-1}{q_{\ell_{\mathfrak M}+1}-1}}\xi_{\ell_{\mathfrak M}+1},\dots), \quad \vec \xi \in [1,\infty)^m.
\end{split}
\end{align*}
\end{proof}

\begin{remark}
Note that in the case $\ell=0$, we obtain an alternative proof of the weak-type extrapolation schemes in \cite[Theorem 3.12]{casto} and \cite[Theorem 6.1]{gramar}, and the one that follows from \cite[Theorem 6.1]{duoextrapol}.
\end{remark}

We have presented Theorem~\ref{multiextrapoldowntotal} in its general form, for $(m+1)$-tuples of functions $(f_1,\dots,f_m,g)$. In the next corollary, we deduce the corresponding extrapolation result for $m$-variable operators.

\begin{corollary}\label{cextrapolone}
Let $T$ be an $m$-variable operator defined for suitable measurable functions. Fix $1 \leq N_1,\dots,N_m \in \mathbb N \mathcup \{ \infty \}$, and $0\leq \ell_{\mathcal R} \leq \ell \leq \ell_{\mathfrak M} \leq  m$, and suppose that for some exponents $1\leq p_1,\dots,p_m < \infty$, $\frac{1}{p}=\frac{1}{p_1}+\dots+\frac{1}{p_m}$, and all weights $v_i \in \widehat{A}_{p_i,N_i+1}$, $i=1,\dots,\ell_{\mathcal R}$, $v_i \in \widehat{A}_{p_i,N_i}$, $i=\ell_{\mathcal R}+1,\dots,\ell$, and $v_i \in {A}_{p_i}$, $i=\ell+1,\dots,m$,
\begin{equation}\label{cextrapoloneh}
    T:L^{p_1,1}(v_1)\times \dots \times  L^{p_{\ell},1}(v_{\ell}) \times L^{p_{\ell+1}}(v_{\ell+1}) \times \dots \times L^{p_m}(v_m)\longrightarrow L^{p,\infty}(v_1^{p/p_1} \dots v_m^{p/p_m}),
\end{equation}
with constant bounded as in \eqref{eqmultiextrapoldowntotalh}. Then, for all finite exponents $1\leq q_1 \leq p_1$$, \dots$, $q_{\ell_{\mathcal R}+1}\geq p_{\ell_{\mathcal R}+1},\dots,1< q_{\ell+1} \leq p_{\ell+1},\dots,q_{\ell_{\mathfrak M}+1}\geq p_{\ell_{\mathfrak M}+1},\dots$, $\frac{1}{q}=\frac{1}{q_1}+\dots+\frac{1}{q_m}$, and all weights $w_i \in \widehat{A}_{q_i,N_i}$, $i=1,\dots,\ell$, and $w_i \in {A}_{q_i}$, $i=\ell+1,\dots,m$, 
\begin{equation}\label{cextrapolonec}
    T:\left( \prod_{i=1}^{\ell} L^{q_i,\min  \left \{1,\frac{q_i}{p_i} \right\}}(w_i)\right) \times \left(\prod_{i=\ell+1}^{m} L^{q_i,\min  \{p_i,q_i \}}(w_i)\right)\longrightarrow L^{q,\infty}(w_1^{q/q_1}\dots w_m^{q/q_m}),
\end{equation}
with constant bounded as in \eqref{eqmultiextrapoldowntotalc}. If for some $1\leq i \leq \ell_{\mathcal R}$, $q_i=1$, then we can take $N_i=0$.
\end{corollary}

\begin{proof}
Given $f_i \in L^{q_i,\min  \left \{1,\frac{q_i}{p_i} \right\}}(w_i)$, $i=1,\dots,\ell$, and $f_i \in L^{q_i,\min  \{p_i,q_i \}}(w_i)$, $i=\ell+1,\dots,m$, take $g \perdef |T(f_1,\dots,f_m)|$, and for every natural number $\varrho \geq 1$, let $g_{\varrho} \perdef \min \{g,\varrho\}\chi_{B(0,\varrho)}$. We can perform this step under the vague assumption that $T$ is defined for suitable measurable functions, including $f_1,\dots,f_m$. Note that for every $\varrho\geq 1$, $g_{\varrho} \leq g$, so from \eqref{cextrapoloneh} we deduce \eqref{eqmultiextrapoldowntotalh} for $g_{\varrho}$, and by Theorem~\ref{multiextrapoldowntotal}, we obtain \eqref{eqmultiextrapoldowntotalc} for $g_{\varrho}$. Finally, we get \eqref{cextrapolonec} by taking the supremum over all $\varrho \geq 1$, since $g_{\varrho} \uparrow g$ and 
\begin{equation*}
    \left \| g \right \|_{L^{q,\infty}(w_1^{q/q_1}\dots w_m^{q/q_m})} = \sup_{\varrho\geq 1}  \left \| g_{\varrho} \right \|_{L^{q,\infty}(w_1^{q/q_1}\dots w_m^{q/q_m})}.
\end{equation*}
\end{proof}

\begin{remark}
If the operator $T$ is initially defined for some nice functions, say bounded functions with compact support or simple functions, and \eqref{cextrapoloneh} only holds for these, then we can recover \eqref{cextrapolonec} in full generality via a standard density argument, assuming that $T$ is multi-sub-linear (see \cite[Proposition 7.2.3]{grafmod}).
\end{remark}

\subsection{A one-weight scheme}\hfill\vspace{2.5mm}

For simplicity, suppose that $\ell_{\mathcal R}=0$ and $\ell=m=2$. Although from Theorem~\ref{multiextrapoldowntotal} we can deduce a one-weight conclusion, we need to assume a two-weight hypothesis to get it because, in general, the weights $v_1$ and $v_2$ that we chose in \eqref{eqoffup0} are different. However, for $q_1>p_1>1$ and $q_2>p_2>1$ such that $\frac{p_1-1}{q_1-1}=\frac{p_2-1}{q_2-1}$, we can have $v_1=v_2$. That is, these weights can coincide if the points $(p_1,p_2)$ and $(q_1,q_2)$ in $(1,\infty)^2$ lay on a straight line passing through the point $(1,1)$. Equivalently, for some $\gamma>0$, the points $P=(\frac{1}{p_1},\frac{1}{p_2})$ and $Q=(\frac{1}{q_1},\frac{1}{q_2})$ in $(0,1)^2$ belong to the graph of the function $F_{\gamma}(x)\perdef \frac{x}{(1-\gamma)x+\gamma}$, defined for $0<x\leq 1$. Adding this assumption in Theorem~\ref{multiextrapoldowntotal} allows us to obtain the following one-weight theorem, and the corresponding extrapolation scheme for multi-variable operators. See Figure~\ref{figextrapolupmod} for a pictorial representation of these results.

\begin{theorem}\label{multiextrapolup}
Fix integers $N_1,\dots,N_m \geq 1$. Given measurable functions $f_1,\dots,f_m$, and $g$, suppose that for some exponents $p_1= \dots = p_m = 1$ or $1 < p_1,\dots,p_m < \infty$, $\frac{1}{p}=\frac{1}{p_1}+\dots+\frac{1}{p_m}$, and every weight $v \in \bigcap_{i=1}^m \widehat{A}_{p_i,N_i}$, 
\begin{equation}\label{eqmultiextrapoluph}
    \left \| g \right \|_{L^{p,\infty}(v)} \leq \varphi (\Vert v \Vert_{\widehat{A}_{p_1,N_1}},\dots,\Vert v \Vert_{\widehat{A}_{p_m,N_m}}) \prod_{i=1}^m \left \| f_i \right \|_{L^{p_i,1}(v)},
\end{equation}
where $\varphi:[1,\infty)^m\longrightarrow [0,\infty)$ is a function increasing in each variable. Then, for all finite exponents $q_1 \geq p_1, \dots,q_m \geq p_m$ such that $\frac{p_1-1}{q_1-1}=\dots = \frac{p_m-1}{q_m-1}$, $\frac{1}{q}=\frac{1}{q_1}+\dots+\frac{1}{q_m}$, and every weight $w \in \bigcap_{i=1}^m \widehat{A}_{q_i,N_i}$,
\begin{align}\label{eqmultiextrapolupc}
\begin{split}
        \left \| g \right \|_{L^{q,\infty}(w)} & \leq \Phi (\Vert w\Vert_{\widehat{A}_{q_1,N_1}},\dots,\Vert w \Vert_{\widehat{A}_{q_m,N_m}}) \prod_{i=1}^m \left \| f_i \right \|_{L^{q_i,1}(w)},
\end{split}
\end{align}
where $\Phi:[1,\infty)^m\longrightarrow [0,\infty)$ is a function that increases in each variable. The same result is valid if for some index $0\leq \ell < m$, $N_{\ell +1}= \dots = N_m = \infty$ and $N_1,\dots,N_\ell < \infty$.
\end{theorem}

\begin{proof}
We will follow the steps of the proof of Theorem~\ref{offup}, skipping most of the tedious computations. 

Note that if for some $1\leq i \leq m$, $q_i=p_i$, then $q_1=p_1,\dots,q_m=p_m$ and there is nothing to prove, so we may assume that $q_i>p_i$, $i=1,\dots,m$. Fix a natural number $\varrho\geq 1$, and let $g_\varrho \perdef |g|\chi_{B(0,\varrho)}$. We will prove \eqref{eqmultiextrapolupc} for the tuple $(f_1,\dots,f_m,g_\varrho)$. Pick a weight $w \in \bigcap_{i=1}^m \widehat{A}_{q_i,N_i}$ and fix $y>0$ such that $\lambda_{g_\varrho}^w(y)\neq 0$.

In order to apply \eqref{eqmultiextrapoluph}, we want to find a weight $v \in \bigcap_{i=1}^m \widehat{A}_{p_i,N_i}$ such that $\lambda_{g_\varrho} ^w(y) \leq \lambda_{g_\varrho} ^v(y)$. 
 If  $1< p_1,\dots,p_m < \infty$, then the restrictions on the exponents allow us to take any $1\leq j \leq m$, and choose
\begin{align*}
    v & \perdef 
       w^{\frac{p_j-1}{q_j-1}}M(w \chi_{\{|g_\varrho|>y\}})^{\frac{q_j-p_j}{q_j-1}}.
\end{align*}
In virtue of Lemma~\ref{pesos5}, we see that for $i=1,\dots,m$, $v \in \widehat A_{p_i,N_i}$, with 
$$
 \Vert v \Vert_{\widehat A_{p_i,N_i}} \lesssim_n \left(\frac{q_i-1}{p_i-1} \right)^{1/p_i} \Vert w \Vert_{\widehat{A}_{q_i,N_i}} ^{q_i/p_i}.
$$

If $p_1= \dots = p_m=1$, then we argue as follows: for $q_0 \perdef \min \{q_1,\dots,q_m\}$, and $N \perdef \max \{N_1,\dots,N_m\}$, 
$\bigcap_{i=1}^m \widehat{A}_{q_i,N_i} 
\subseteq \widehat A_{q_0,N}$, so we can find measurable functions $h_1,\dots,h_{N} \in L^1_{loc}(\mathbb R^n)$, parameters $\theta_1, \dots, \theta_{N} \in (0,1]$, with $\theta_1+ \dots + \theta_{N} = 1$, and a weight $u\in A_1$ such that $w = \left(\prod_{j=1}^{N} (Mh_j)^{\theta_j} \right)^{1-q_0}u$, with $[u]_{A_1}^{1/q_0} \leq (1+\frac{1}{q_0}) \Vert w \Vert_{\widehat A_{q_0,N}}$. In particular, for $i=1,\dots,m$,
\begin{equation}\label{eqmultiextrapolup5}
w =\left(M(\chi_{\mathbb R^n})^{1-\frac{q_0-1} {q_{i}-1}}\prod_{j=1}^{N} (Mh_j)^{\frac{\theta_j(q_0-1)}{q_{i}-1} } \right)^{1-q_i}u.
\end{equation}
This step is crucial to guarantee that we can apply Theorem~\ref{thetasawyer} $m$ times, one for each of the exponents $q_1,\dots,q_m$, and always working with the same weight $u\in A_1$.

Now, write $q_\infty \perdef \max \{q_1,\dots,q_m\}$, and choose $\theta \perdef 1-\frac{1}{5q_{\infty}}$, $\tau \perdef 1+\frac{1}{2^{n+1}[u]_{A_1}}$, and $\mu \perdef  1-\frac{1-\theta}{\tau}$, and define $w_{\theta} \perdef w u^{\theta -1}$ and
\begin{align*}
    v & \perdef 
       M^{\mu}(w_{\theta} \chi_{\{|g_\varrho|>y\}})u^{\tau(1-\mu)}.
\end{align*}
In virtue of \eqref{eqoffup5}, we see that $v \in A_{1}$, with 
 $$
    [v] _{A_1}\lesssim_n \frac{\Vert w \Vert_{\widehat A_{q_0,N}}^{q_0}}{1-\theta} \lesssim q_\infty \max_{1\leq k \leq m } \left \{\Vert w \Vert_{\widehat A_{q_k,N_k}} \right \}^{q_0}.
$$

Observe that, in any case, $v \geq w \chi_{\{|g_\varrho|>y\}}$, so \eqref{eqmultiextrapoluph} implies that
\begin{align}\label{eqmultiextrapolup2}
\begin{split}
\lambda_{g_\varrho} ^w(y) &  \leq  \lambda_{g_\varrho} ^v(y) \leq \frac{1}{y^p} \varphi(\Vert v \Vert_{\widehat{A}_{p_1,N_1}},\dots,\Vert v \Vert_{\widehat{A}_{p_m,N_m}})^p \prod_{i=1}^m \left \| f_i \right \|_{L^{p_i,1}(v)}^p \\ & \leq \frac{\widetilde{\varphi}^p}{y^p}  
\prod_{i=1}^m \left \| f_i \right \|_{L^{p_i,1}(v)}^p,  
\end{split}
\end{align}
where $\widetilde{\varphi}$ is given by
$$
 \left \{ \begin{array}{lr}
   \varphi (\mathfrak c_n \left(\frac{q_1-1}{p_1-1}  \Vert w\Vert_{\widehat{A}_{q_1,N_1}}^{q_1}\right)^{1/p_1},\dots,\mathfrak c_n \left(\frac{q_m-1}{p_m-1} \Vert w \Vert_{\widehat{A}_{q_m,N_m}}^{q_m} \right)^{1/p_m}),  & 1<p_j < \infty, \\
   & \\
   \varphi(\mathfrak c_n q_\infty \max\limits_{ 1 \leq k \leq m } \left \{\Vert w \Vert_{\widehat A_{q_k,N_k}} \right \}^{q_0},\dots,\mathfrak c_n q_\infty  \max\limits_{ 1 \leq k \leq m } \left \{\Vert w \Vert_{\widehat A_{q_k,N_k}} \right \}^{q_0}),  &  p_j = 1.
\end{array} \right.
$$

For $i=1,\dots,m$, we want to replace $\Vert f_i \Vert_{L^{p_i,1}(v)}$ by $\Vert f_i \Vert_{L^{q_i,1}(w)}$ in \eqref{eqmultiextrapolup2}. Applying H\"older's inequality with exponent $\frac{q_i}{p_i}>1$, we obtain that for every $t>0$,
\begin{align*}
\lambda_{f_i}^{v}(t) \leq \frac{q_i}{p_i} w(\{|f_i|>t\})^{p_i/q_i} \times \left \{\begin{array}{lr}
   \left \| \frac{M(w \chi_{\{|g_{\varrho}|>y\}})}{w} \right \|_{L^{q_i',\infty}(w)}^{\frac{q_i-p_i}{q_i-1}},   &  1<p_i < \infty, \\
  & \\
    \left \| \frac{M^{\mu}(w_{\theta}\chi_{\{|g_{\varrho}|>y\}})}{w_{\theta}} \right \|_{L^{q_i',\infty}(w)},  & p_i = 1.
\end{array} \right.
\end{align*}

Now, \cite[Theorem 8]{prp} gives us that
$$
\left \| \frac{M(w \chi_{\{|g_{\varrho}|>y\}})}{w} \right \|_{L^{q_i',\infty}(w)} \lesssim_{n,q_i} [w]_{A_{q_i}^{\mathcal R}}^{q_i+1} w(\{|g_\varrho|>y\})^{1/q_i '},
$$
and from \eqref{eqmultiextrapolup5} and Theorem~\ref{thetasawyer}, we deduce that
$$
\left \| \frac{M^{\mu}(w_{\theta}\chi_{\{|g_{\varrho}|>y\}})}{w_{\theta}} \right \|_{L^{q_i',\infty}(w)} \lesssim_{n,q_i,q_\infty}  
[u]_{A_1} [w]_{A_{q_i}^{\mathcal R}}^{2} w(\{|g_\varrho|>y\})^{1/q_i '}.
$$

Thus, in virtue of Theorem~\ref{aprgorro}, 
\begin{equation*}
     \lambda_{f_i}^{v}(t) \leq \frac{q_i}{p_i}   \phi_i^{\frac{q_i-p_i}{q_i-1}} w(\{|g_\varrho|>y\})^{1-\frac{p_i}{q_i}} w(\{|f_i|>t\})^{p_i/q_i},
\end{equation*}
with
$$
\phi_i \perdef \left \{\begin{array}{lr}
  \mathfrak C_{q_i}^n \left(N_i \Vert w \Vert_{\widehat A_{q_i,N_i}} \right)^{q_i+1} ,  &  1<p_i < \infty, \\
  & \\
 \mathfrak C_{q_i,q_\infty}^n 
 \left(N_i \Vert w \Vert_{\widehat A_{q_i,N_i}} \right)^{2} \max\limits_{1\leq k \leq m } \left \{\Vert w \Vert_{\widehat A_{q_k,N_k}} \right\}^{q_0},  &  p_i = 1,
\end{array} \right.
$$
for some constants $ \mathfrak C_{q_i}^n,\mathfrak C_{q_i,q_\infty}^n \geq 1$, and hence,
\begin{align}\label{eqmultiextrapolup3}
    \begin{split}
        \left \| f_i \right \|_{L^{p_i,1}(v)} \leq  \left(\frac{p_i}{q_i}\right)^{1/p_i '} \phi_i^{\frac{1}{p_i}\cdot\frac{q_i-p_i}{q_i-1}}w(\{|g_\varrho|>y\})^{\frac{1}{p_i}-\frac{1}{q_i}} \Vert f_i \Vert_{L^{q_i,1}(w)}.
    \end{split}
\end{align}

Combining the estimates \eqref{eqmultiextrapolup2} and \eqref{eqmultiextrapolup3}, we have that
\begin{equation*}
    \lambda_{g_\varrho} ^w(y) \leq \frac{1}{y^p} \Phi (\Vert w \Vert_{\widehat{A}_{q_1,N_1}},\dots,\Vert w \Vert_{\widehat{A}_{q_m,N_m}})^p \left(\prod_{i=1}^m \left \| f_i \right \|_{L^{q_i,1}(w)}^p\right) \lambda_{g_\varrho} ^w(y) ^{1-\frac{p}{q}},
\end{equation*}
with
\begin{align*}
    \Phi (\Vert w \Vert_{\widehat{A}_{q_1,N_1}},\dots,\Vert w \Vert_{\widehat{A}_{q_m,N_m}}) & = \left(\prod_{i=1}^m \left(\frac{p_i}{q_i}\right)^{1/p_i '} \phi_i^{\frac{1}{p_i}\cdot\frac{q_i-p_i}{q_i-1}}\right)   \widetilde{\varphi},
\end{align*}
and the desired result follows.

Now, suppose that for some index $0\leq \ell < m$, $N_{\ell +1}= \dots = N_m = \infty$ and $N_1,\dots,N_\ell < \infty$. By hypothesis, we have that for every weight $v\in \bigcap_{i=1}^m \widehat{A}_{p_i,N_i}$, 
\begin{equation}\label{eqmultiextrapolup7}
    \left \| g \right \|_{L^{p,\infty}(v)} \leq \varphi (\Vert v \Vert_{\widehat{A}_{p_1,N_1}},\dots,\Vert v \Vert_{\widehat{A}_{p_\ell,N_\ell}},\Vert v \Vert_{\widehat{A}_{p_{\ell+1},\infty}},\dots,\Vert v \Vert_{\widehat{A}_{p_m,\infty}}) \prod_{i=1}^m \left \| f_i \right \|_{L^{p_i,1}(v)}.
\end{equation} 
Pick a weight $w\in \bigcap_{i=1}^m \widehat A_{q_i,N_i}$. For $i=\ell +1, \dots, m$, we can find an integer $N_i ^0 \geq 1$ such that $w\in \widehat A_{q_i,N_i^0}$, with $\Vert w \Vert_{\widehat A_{q_i,\infty}} \leq N_i^0\Vert w \Vert_{\widehat A_{q_i,N_i^0}} \leq 2 \Vert w \Vert_{\widehat A_{q_i,\infty}}$. From \eqref{eqmultiextrapolup7}, we deduce that for every weight $v\in \left(\bigcap_{i=1}^\ell \widehat{A}_{p_i,N_i}\right) \cap \left(\bigcap_{i=\ell+1}^m \widehat{A}_{p_i,N_i^0}\right)$,
\begin{align*}
\begin{split}
    \left \| g \right \|_{L^{p,\infty}(v)} & \leq \varphi (
    \dots,\Vert v \Vert_{\widehat{A}_{p_\ell,N_\ell}},N_{\ell+1}^0\Vert v \Vert_{\widehat{A}_{p_{\ell+1},N_{\ell+1}^0}},\dots
    )  \prod_{i=1}^m \left \| f_i \right \|_{L^{p_i,1}(v)},
    \end{split}
\end{align*}
and applying Theorem~\ref{multiextrapolup} for $N_1,\dots,N_\ell,N_{\ell+1}^0,\dots,N_m^0$, we conclude that 
\begin{equation*}
 \left \| g \right \|_{L^{q,\infty}(w)} 
 \leq  \Phi(
 \dots,\Vert w \Vert_{\widehat{A}_{q_\ell,N_\ell}},\Vert w \Vert_{\widehat{A}_{q_{\ell+1},\infty}},\dots
 )\prod_{i=1}^m \left \| f_i \right \|_{L^{q_i,1}(w)},
\end{equation*}
with 
$\Phi(\vec \xi)$ defined for $\vec \xi \in [1,\infty)^m$ as
\begin{align*}
\begin{split}
    & \left(\prod_{i=1}^{m} \left(\frac{p_i}{q_i}\right)^{1/p_i '} \times \left \{\begin{array}{lr}
  \mathfrak C_{q_i}^n \Upsilon_i ^{q_i+1} ,  &  1<p_i < \infty, \\
  & \\ \mathfrak C_{q_i,q_\infty}^n 
 \Upsilon_i^2 
  \max \left \{\dots,\xi_{\ell},2\xi_{\ell+1},\dots \right \}^{q_0},  &  p_i= 1,
\end{array} \right\} ^{\frac{1}{p_i}\frac{q_i-p_i}{q_i-1}}\right) \\ & \times  \left \{ \begin{array}{lr}
   \varphi (\dots,\mathfrak c_n \left(\frac{q_\ell-1}{p_\ell-1}  \xi_\ell ^{q_\ell}\right)^{1/p_\ell},\mathfrak c_n \left(\frac{q_{\ell+1}-1}{p_{\ell+1}-1} (2 \xi_{\ell+1} )^{q_{\ell+1}}\right)^{1/p_{\ell+1}},\dots),  & 1<p_j < \infty, \\
   & \\
   \varphi(
   \dots,\mathfrak c_n q_\infty \max \left \{\dots,\xi_{\ell},2\xi_{\ell+1},\dots \right \}^{q_0},\dots
   ),  &  p_j= 1,
\end{array} \right\}
\end{split}
\end{align*}
where $\Upsilon_i \perdef N_i \xi_i$ for $1\leq i \leq \ell$, and $\Upsilon_i \perdef 2 \xi_i$ for $\ell+1\leq i \leq m$.
\end{proof}

\begin{figure}[th]
\centering
\begin{tikzpicture}[scale=6]
\fill[gray!40] (0,0) rectangle (1,1);
\draw[very thick, dotted,draw=black,->] (0,0) -- (0,1.25);
\draw[very thick, dotted,draw=black,->] (0,0) -- (1.25,0);
\draw[very thick, dotted, draw=black] (1,0) -- (1,1) -- (0,1);

\draw[->,thick,color=black,smooth,domain=1:0.759746926647957852] plot(\x,{\x/(10-9*\x)});

\draw[->,thick,color=black,smooth,domain=0.78:0.2591033192698828835] plot(\x,{\x/(10-9*\x)});
\draw[->,thick,color=black,smooth,domain=0:0.2591033192698828835] plot({1-\x/(10-9*\x)},{1-\x});
\draw[thick,color=black,smooth,domain=0.28:0] plot(\x,{\x/(10-9*\x)});
\draw[fill, black] (0.8919345770811255471,0.4521644725175936574) circle [radius=0.025] node [anchor=south east] at (0.8919345770811255471, 0.4521644725175936574) {$P$};
\draw[fill, black] (0.54783552748240634261,0.10806542291887445287) circle [radius=0.025] node [anchor=south east] at (0.54783552748240634261, 0.10806542291887445287) {$Q$};

\draw[->,thick,color=black] (1,1) -- (0.8535533905932737,0.8535533905932737);
\draw[->,thick,color=black] (0.86,0.86) -- (0.5,0.5);
\draw[->,thick,color=black] (0.6,0.6) -- ($(0,0)+{2-sqrt(2)}*(0.25,0.25)$);
\draw[thick,color=black] (1/6,1/6) -- (0,0);
\draw[fill, black] ($(0,0)+sqrt(2)*(0.5,0.5)$) circle [radius=0.025] node [anchor=south] at ($(0,0.02)+sqrt(2)*(0.5,0.5)$) {$P$};
\draw[fill, black] ($(0,0)+{2-sqrt(2)}*(0.5,0.5)$) circle [radius=0.025] node [anchor=north ] at ($(0,-0.02)+{2-sqrt(2)}*(0.5,0.5)$) {$Q$};

\draw[->,thick,color=black,smooth,domain=1:0.876100656900704586857] plot(\x,{\x/(50-49*\x)});
\draw[->,thick,color=black,smooth,domain=0.89:0.54783552748240634261] plot(\x,{\x/(50-49*\x)});
\draw[->,thick,color=black,smooth,domain=0:0.54783552748240634261] plot({1-\x/(50-49*\x)},{1-\x});
\draw[thick,color=black,smooth,domain=0.56:0] plot(\x,{\x/(50-49*\x)});
\draw[fill, black] (0.9548307487772641567,0.297150199712955036) circle [radius=0.025] node [anchor= south east] at (0.9548307487772641567,0.297150199712955036) {$P$};
\draw[fill, black] (0.70284980028704496399,0.04516925122273584332) circle [radius=0.025] node [anchor=south  east] at (0.70284980028704496399,0.04516925122273584332) {$Q$};

\draw[->,thick,color=black,smooth,domain=1:0.633974596215561353236] plot(\x,{\x/(3-2*\x)});
\draw[->,thick,color=black,smooth,domain=0.65:0.2110322627132408914] plot(\x,{\x/(3-2*\x)});
\draw[->,thick,color=black,smooth,domain=0:0.2110322627132408914] plot({1-\x/(3-2*\x)},{1-\x});
\draw[thick,color=black,smooth,domain=0.25:0] plot(\x,{\x/(3-2*\x)});
\draw[fill, black] (0.8098218198733462114,0.5866759071727946824) circle [radius=0.025] node [anchor=south east] at (0.8098218198733462114,0.5866759071727946824) {$P$};
\draw[fill, black] (0.41332409282720531758,0.19017818012665378855) circle [radius=0.025] node [anchor=south  east] at (0.41332409282720531758,0.19017818012665378855) {$Q$};

\draw[->,thick,color=black,smooth,domain=1:0.633974596215561353236] plot({\x/(3-2*\x)},\x);
\draw[->,thick,color=black,smooth,domain=0.65:0.2110322627132408914] plot({\x/(3-2*\x)},\x);
\draw[->,thick,color=black,smooth,domain=0:0.2110322627132408914] plot({1-\x},{1-\x/(3-2*\x)});
\draw[thick,color=black,smooth,domain=0.25:0] plot({\x/(3-2*\x)},\x);
\draw[fill, black] (0.5866759071727946824,0.8098218198733462114) circle [radius=0.025] node [anchor=north west] at (0.5866759071727946824,0.8098218198733462114) {$P$};
\draw[fill, black] (0.19017818012665378855,0.41332409282720531758) circle [radius=0.025] node [anchor=north west] at (0.19017818012665378855,0.41332409282720531758) {$Q$};

\draw[->,thick,color=black,smooth,domain=1:0.876100656900704586857] plot({\x/(50-49*\x)},\x);
\draw[->,thick,color=black,smooth,domain=0.89:0.54783552748240634261] plot({\x/(50-49*\x)},\x);
\draw[->,thick,color=black,smooth,domain=0:0.54783552748240634261] plot({1-\x},{1-\x/(50-49*\x)});
\draw[thick,color=black,smooth,domain=0.56:0] plot({\x/(50-49*\x)},\x);
\draw[fill, black] (0.297150199712955036,0.9548307487772641567) circle [radius=0.025] node [anchor=north west] at (0.297150199712955036,0.9548307487772641567) {$P$};
\draw[fill, black] (0.04516925122273584332,0.70284980028704496399) circle [radius=0.025] node [anchor=north west] at (0.04516925122273584332,0.70284980028704496399) {$Q$};

\draw[->,thick,color=black,smooth,domain=1:0.759746926647957852] plot({\x/(10-9*\x)},\x);
\draw[->,thick,color=black,smooth,domain=0:0.2591033192698828835] plot({1-\x},{1-\x/(10-9*\x)});
\draw[->,thick,color=black,smooth,domain=0.78:0.2591033192698828835] plot({\x/(10-9*\x)},\x);
\draw[thick,color=black,smooth,domain=0.28:0] plot({\x/(10-9*\x)},\x);
\draw[fill, black] (0.4521644725175936574,0.8919345770811255471) circle [radius=0.025] node [anchor= north west] at (0.4521644725175936574,0.8919345770811255471) {$P$};
\draw[fill, black] (0.10806542291887445287,0.54783552748240634261) circle [radius=0.025] node [anchor=north west] at (0.10806542291887445287,0.54783552748240634261) {$Q$};

\draw[white,fill] (0,0) circle [radius=0.025];
\draw[black,thick] (0,0) circle [radius=0.025] node [anchor=north east] at (0, 0) {$(0,0)$};
\draw[white,fill] (1,0) circle [radius=0.025];
\draw[black,thick] (1,0) circle [radius=0.025] node [anchor=north west] at (1, 0) {$(1,0)$};
\draw[white,fill] (0,1) circle [radius=0.025];
\draw[black,thick] (0,1) circle [radius=0.025] node [anchor=south east] at (0, 1) {$(0,1)$};
\draw[black,fill] (1,1) circle [radius=0.025];
\draw[thick, black] (1,1) circle [radius=0.025] node [anchor=south west] at (1, 1) {$(1,1)$};
\end{tikzpicture}
\rule{.8\textwidth}{.4pt}
\caption{Pictorial representation of Theorem~\ref{multiextrapolup} and Corollary~\ref{corollmultiextrapolup} for $m=2$.}
\label{figextrapolupmod}
\end{figure}
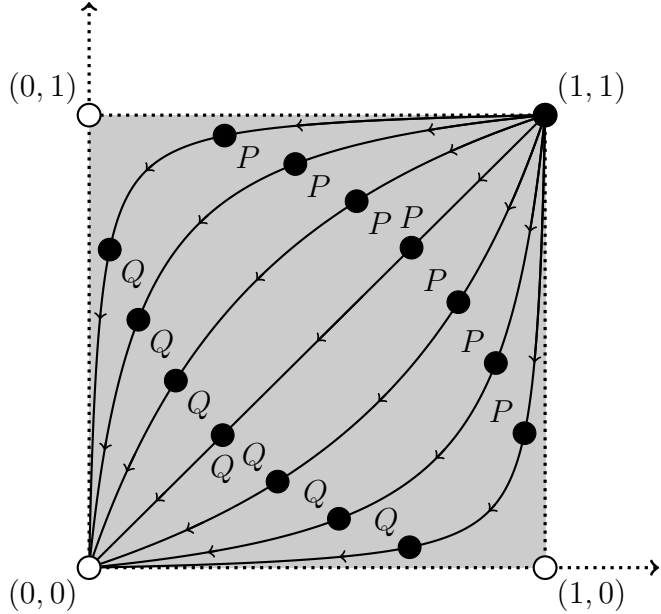

As usual, from Theorem~\ref{multiextrapolup} we can obtain the corresponding extrapolation result for multi-variable operators arguing as in the proof of Corollary~\ref{cextrapolone}. 

\begin{corollary}\label{corollmultiextrapolup}
Let $T$ be an $m$-variable operator defined for suitable measurable functions. Fix $1 \leq N_1,\dots,N_m \in \mathbb N \mathcup \{ \infty \}$. Suppose that for some exponents $p_1= \dots = p_m = 1$ or $1 < p_1,\dots,p_m < \infty$, $\frac{1}{p}=\frac{1}{p_1}+\dots+\frac{1}{p_m}$, and every weight $v \in \bigcap_{i=1}^m \widehat{A}_{p_i,N_i}$, 
\begin{equation*}
    T:L^{p_1,1}(v)\times \dots \times  L^{p_{m},1}(v) \longrightarrow L^{p,\infty}(v),
\end{equation*}
with constant bounded as in \eqref{eqmultiextrapoluph}. Then, for all finite exponents $q_1 \geq p_1, \dots,$\break $q_m \geq p_m$ such that $\frac{p_1-1}{q_1-1}=\dots = \frac{p_m-1}{q_m-1}$, $\frac{1}{q}=\frac{1}{q_1}+\dots+\frac{1}{q_m}$, and every weight $w \in \bigcap_{i=1}^m \widehat{A}_{q_i,N_i}$,
\begin{equation*}
    T:L^{q_1,1}(w)\times \dots \times  L^{q_{m},1}(w) \longrightarrow L^{q,\infty}(w),
\end{equation*}
with constant bounded as in \eqref{eqmultiextrapolupc}.
\end{corollary}

\begin{remark}
    One-weight strong-type extrapolation for multi-variable operators was studied in \cite[Theorem 1.2]{gramar}, assuming an alternative condition on the exponents; that is, $\frac{p_1}{q_1}=\dots=\frac{p_m}{q_m}$. An analog for restricted weak-type operators in the downwards case was obtained in \cite[Theorem 1.3]{laura}.
\end{remark}

\section{Applications}

In this section, we apply the extrapolation results previously introduced to produce mixed and restricted weak-type bounds for classical operators.

\subsection{Weak-type operators}\hfill\vspace{2.5mm}

We will demonstrate that, under suitable conditions, a weak-type hypothesis for a multi-sub-linear operator $T$ is more than enough to run an extrapolation argument yielding the full range of weighted restricted weak-type bounds for such an operator. In particular, we extend outside the Banach range the well-known equivalence between weak-type for characteristic functions and restricted weak-type (see \cite{BS,moon,stws}). Also, we show that weighted restricted weak-type conditions are stronger than the usual weighted weak-type ones.

The next result follows applying Theorems \ref{multiextrapoldowntotal} and \ref{epsdelta}, Corollary~\ref{cextrapolone}, and the standard extension argument in \cite[Page 256]{BS} if $q>1$.

\begin{theorem}\label{extrapolchar}
Let $T$ be a multi-sub-linear operator defined for suitable measurable functions. Suppose that for some exponents $1\leq p_1,\dots,p_m < \infty$, $\frac{1}{p}=\frac{1}{p_1}+\dots+\frac{1}{p_m}$, and all weights $v_i \in \widehat{A}_{p_i}$, $i=1,\dots,m$,
\begin{equation*}
     \left \| T(\chi_{E_1},\dots,\chi_{E_m}) \right \|_{L^{p,\infty}(v_1^{p/p_1} \dots v_m^{p/p_m})}\leq \varphi (\Vert v_1 \Vert_{\widehat{A}_{p_1}},
  \dots,\Vert v_{m} \Vert_{\widehat{A}_{p_{m}}}) \prod_{i=1}^m v_i (E_i)^{1/p_i},
\end{equation*}
for all measurable sets $E_1,\dots,E_m \subseteq \mathbb R^n$, where $\varphi:[1,\infty)^m\longrightarrow [0,\infty)$ is a function increasing in each variable. Then, for all exponents $1\leq q_1,\dots,q_m<\infty$, $\frac{1}{q}=\frac{1}{q_1}+\dots+\frac{1}{q_m}$, and all weights $w_i \in \widehat{A}_{q_i,\infty}$, $i=1,\dots,m$, 
\begin{equation*}
    T: L^{q_1,1}(w_1)\times \dots \times  L^{q_m,1}(w_m)\longrightarrow L^{q,\infty}(w_1^{q/q_1}\dots w_m^{q/q_m}),
\end{equation*}
with monotonically increasing dependence of the constant on $(\Vert w_i \Vert_{\widehat{A}_{q_i,\infty}})_{1\leq i \leq m}$, provided that $q>1$ or $T$ is $(\varepsilon,\delta)$-atomic approximable or iterative $(\varepsilon,\delta)$-atomic approximable.
\end{theorem}

\begin{remark}
    This result provides positive evidence for the conjecture that given $p>1$, and $v \in \widehat{A}_{p,\infty}$, there exists a weight $v_0 \in \widehat{A}_{p}$ such that $v_0 \eqsim v$, with implicit constants depending only on $p$, $\Vert v\Vert_{\widehat{A}_{p,\infty}}$, and the dimension $n$.
\end{remark}

\begin{remark}
    Under the hypotheses of Theorem~\ref{extrapolchar}, and assuming that $T$ is $(\varepsilon,\delta)$-atomic approximable or iterative $(\varepsilon,\delta)$-atomic approximable, we can extrapolate down to the endpoint $(1,\dots,1,\frac{1}{m})$ with Theorems~\ref{multiextrapoldowntotal} and \ref{epsdelta}, and in virtue of the classical multi-variable weak-type Rubio de Francia's extrapolation in \cite[Theorem 6.1]{gramar}, we deduce that for all exponents $1 < q_1,\dots,q_m<\infty$, $\frac{1}{q}=\frac{1}{q_1}+\dots+\frac{1}{q_m}$, and all weights $w_i \in {A}_{q_i}$, $i=1,\dots,m$, 
\begin{equation*}
    T: L^{q_1}(w_1)\times \dots \times  L^{q_m}(w_m)\longrightarrow L^{q,\infty}(w_1^{q/q_1}\dots w_m^{q/q_m}),
\end{equation*}
with monotonically increasing dependence of the constant on
$([w_i]_{A_{q_i}})_{1\leq i\leq m}$. \end{remark}

Recently, the following class of operators has been extensively studied in \cite{cao}.

\begin{definition}
Given a Banach space $\mathbb B$, we say that an $m$-variable operator $T$ is a $\mathbb B$-\textit{valued multi-linear bounded oscillation operator} if there is a $\mathbb B$-valued $m$-linear (or $m$-sub-linear if $\mathbb B=\mathbb R$) operator $\mathcal T$ such that for every $x \in \mathbb R^n$, $T(\vec f)(x) = \Vert \mathcal T(\vec f)(x) \Vert_{\mathbb B}$, and there exist constants $C_1,C_2>0$ such that for all functions $f_1,\dots,f_m \in L^1(\mathbb R^n)$:
\begin{enumerate}
\item [($a$)] For every cube $Q \subseteq \mathbb R^n$, there exists a cube $P\subseteq \mathbb R^n$ so that $Q \subsetneq P$, and
$$
\sup_{x\in Q} \Vert \mathcal T(\vec f \chi_{5P})(x)- \mathcal T(\vec f \chi_{5Q})(x) \Vert_{\mathbb B} \leq C_1 \prod_{i=1}^m \avgint_{5P} |f_i|.
$$

\item [($b$)] For every cube $Q \subseteq \mathbb R^n$,
$$
\sup_{x,y\in Q} \left \Vert \left(\mathcal T(\vec f)- \mathcal T(\vec f \chi_{5Q})\right)(x) - \left(\mathcal T(\vec f)- \mathcal T(\vec f \chi_{5Q})\right)(y) \right \Vert_{\mathbb B} \leq C_2 \prod_{i=1}^m \sup_{P\supseteq Q} \avgint_{P} |f_i|.
$$
\end{enumerate}
\end{definition}

It follows immediately from \cite[Theorem 1.7]{cao} that we can transfer the  weak-type $(1,\dots,1,\frac{1}{m})$ bounds for $\mathcal M$ in \cite[Theorem 3.3]{LOPTT} to general $\mathbb B$-valued multi-linear bounded oscillation operators, and we can apply our extrapolation scheme in Corollary~\ref{cextrapolone} to obtain the corresponding mixed and restricted weak-type estimates.

\begin{theorem}\label{multiosc}
    Let $T$ be a $\mathbb B$-valued multi-linear bounded oscillation operator defined for suitable measurable functions, and suppose that
    $$
T:L^1(\mathbb R^n) \times \dots \times L^1(\mathbb R^n) \longrightarrow L^{\frac{1}{m},\infty}(\mathbb R^n).
    $$
Then, for all exponents $1\leq q_1,\dots,q_m < \infty$, $\frac{1}{q}=\frac{1}{q_1}+\dots+\frac{1}{q_m}$, and all weights $w_i \in \widehat{A}_{q_i,\infty}$, $i=1,\dots,m$,
\begin{equation*}
    T:L^{q_1,1} (w_1) \times \dots \times L^{q_m,1} (w_m) \longrightarrow L^{q,\infty}(w_1^{q/q_1}\dots w_m^{q/q_m}),
\end{equation*}
with monotonically increasing dependence of the constant on $(\Vert w_i\Vert_{\widehat{A}_{q_i,\infty}})_{1\leq i \leq m}$. Likewise, we also get mixed-type inequalities for $T$.
\end{theorem}

Many well-known multi-variable operators satisfy the hypotheses of Theorem~\ref{multiosc}, meaning that we automatically cover the full range of mixed and restricted weak-type bounds for them. Among these, we find multi-linear Littlewood-Paley square operators (see \cite[Theorem 2.13]{cao}), multi-linear Fourier integral operators (see \cite[Theorem 2.20]{cao}), higher order Calder\'on commutators 
(see \cite[Theorem 2.24]{cao}), and maximally modulated multi-linear singular integrals (see \cite[Theorem 2.27]{cao}). The multi-sub-linear maximal operator $\mathcal M$ (see \cite[Theorem 2.1]{cao}), and $m$-variable $\omega$-Calder\'on-Zygmund operators with $\omega$ satisfying the Dini condition (see \cite[Theorem 2.5]{cao}) also satisfy the hypotheses of Theorem~\ref{multiosc}, but we managed to produce better results for them in \cite[Theorem 10]{prp}, \cite[Theorem 5.2.7]{thesis}, and \cite[Remark 5.2.9]{thesis}.

\subsection{Product-type operators and commutators}\hfill\vspace{2.5mm}

We start with the following result, that gives us mixed-type bounds for products of one-variable operators.

\begin{proposition}\label{producttype}
Let $T_1,\dots,T_{m}$ be one-variable operators defined for suitable measurable functions. For $i=2,\dots,m$, suppose that for some $p_i>1$, and every weight $v_i \in A_{p_i}$,
\begin{equation*}
    T_i: L^{p_i,1}(v_i) \longrightarrow L^{p_i,1}(v_i),
\end{equation*}
with constant bounded by $\varphi_i([v_i]_{A_{p_i}})$, where $\varphi_i:[1,\infty) \longrightarrow [0,\infty)$ is an increasing function. Suppose also that for some $p_1\geq 1$, and every weight $v_1 \in A_{p_1}^{\mathcal R}$,
\begin{equation*}
    T_1: L^{p_1,1}(v_1) \longrightarrow L^{p_1,\infty}(v_1),
\end{equation*}
with constant bounded by $\varphi_1([v_1]_{A_{p_1}^{\mathcal R}})$, where $\varphi_1:[1,\infty) \longrightarrow [0,\infty)$ is an increasing function. If $\frac{1}{p_1}+\dots+\frac{1}{p_m}=\frac{1}{p}<1$, then for all weights $w_1 \in A_{p_1}^{\mathcal R}$, and $w_i \in A_{p_i}$, $i=2,\dots,m$, and all suitable measurable functions $f_1,\dots,f_{m}$, 
\begin{equation*}
    \Vert T_1f_1 \dots T_mf_m \Vert_{L^{p,\infty}(w_1^{p/p_1}\dots w_m^{p/p_m})} \leq \Phi([w_1]_{ A_{p_1}^{\mathcal R}},[w_2]_{A_{p_2}},\dots) \prod_{i=1}^m\Vert f_i \Vert_{L^{p_i,1}(w_i)}, 
\end{equation*}
where $\Phi:[1,\infty)^m \longrightarrow [0,\infty)$ is a function increasing in each variable.
\end{proposition}

\begin{proof}
Writing $w \perdef w_1^{p/p_1}\dots w_m^{p/p_m}$ and applying \cite[Lemma 2.2.1]{thesis}, we get that
\begin{align*}
   \Vert T_1f_1 \dots T_mf_m \Vert_{L^{p,\infty}(w)}  & \leq \frac{2^{2m-2} p'}{p_2\dots p_{m}} \left(\prod_{i=2}^{m}\Vert T_if_i \Vert_{L^{p_i,1}(w_i)} \right)\Vert T_1f_1 \Vert_{L^{p_1,\infty}(w_1)} \\ & \leq \frac{2^{2m-2} p'}{p_2\dots p_{m}} \left(\prod_{i=2}^{m}\varphi_i([w_i]_{A_{p_i}}) \right)\varphi_1([w_1]_{A_{p_1}^{\mathcal R}})\prod_{i=1}^m \Vert f_i \Vert_{L^{p_i,1}(w_i)}.
\end{align*}
\end{proof}

\begin{remark}\label{rmkproduct}
    For $i=2,\dots,m$, if $T_i$ is sub-linear, and for some $p_i>1$, and every weight $v_i\in A_{p_i}^{\mathcal R}$,
\begin{equation*}
    T_i: L^{p_i,1}(v_i) \longrightarrow L^{p_i,\infty}(v_i),
\end{equation*}
with constant bounded by $\varphi_i([v_i]_{A_{p_i}^{\mathcal R}})$, then in virtue of \cite[Theorem 3.1.9]{thesis}, $T_i$ satisfies the hypotheses of Proposition~\ref{producttype}.
\end{remark}

For operators as in Remark~\ref{rmkproduct}, we should be able to extend Proposition~\ref{producttype} for $p_1,\dots,p_m\geq 1$, without restrictions on $p$, and assuming that for $1\leq \ell \leq m$, $w_i \in A_{p_i}^{\mathcal R}$, $i=1,\dots,\ell$, and $w_i \in A_{p_i}$, $i=\ell+1,\dots,m$. This question is still open due to the lack of a H\"older-type inequality for Lorentz spaces with the change of measures (see \cite[Page 27]{thesis} or \cite{crp}) and a complete mixed-type generalization of \cite[Lemma 2.2.1]{thesis}, although we managed to do the job for the particular case of the point-wise product of Hardy-Littlewood maximal operators (see \cite[Theorem 3]{prp} and \cite[Remark 2.4.2]{thesis}). Fortunately, we can use Corollary~\ref{cextrapolone}, and Remarks \ref{rmkmixedextrapolfix} and \ref{rmkmixedextrapolfixup}, to improve the conclusion of Proposition~\ref{producttype}. 

\begin{theorem}\label{producttypeextrapol}
Let $T_1,\dots,T_m$ be sub-linear operators defined for suitable measurable functions. For $i=1,\dots,m$, suppose that for some $p_i>1$, and every weight $v_i\in A_{p_i}^{\mathcal R}$,
\begin{equation*}
    T_i: L^{p_i,1}(v_i) \longrightarrow L^{p_i,\infty}(v_i),
\end{equation*}
with constant bounded by $\varphi_i([v_i]_{A_{p_i}^{\mathcal R}})$, where $\varphi_i:[1,\infty) \longrightarrow [0,\infty)$ is an increasing function. Consider the operator 
\begin{equation*}
    T^{\otimes}(f_1,\dots,f_m)(x)\perdef T_1f_1(x)\dots T_mf_m(x), \quad x \in \mathbb R^n,
\end{equation*}
defined for suitable measurable functions $f_1,\dots,f_m$. If $\frac{1}{p_1}+\dots+\frac{1}{p_m}=\frac{1}{p}<1$, then for all exponents $1\leq q_1 < \infty$, $1< q_2,\dots,q_{m}< \infty$, and $\frac{1}{q}=\frac{1}{q_1}+\dots+\frac{1}{q_m}$, and all weights $w_1 \in \widehat {A}_{q_1,\infty}$, and $w_i \in {A}_{q_i}$, $i=2,\dots,m$,
\begin{equation*}
    T^{\otimes}: L^{q_1,\min \left\{1, \frac{q_1}{p_1}\right\}}(w_1) \times \dots \times L^{q_m,\min \left \{1, \frac{q_m}{p_m} \right\}}(w_m)  \longrightarrow L^{q,\infty}(w_1^{q/q_1}\dots w_m^{q/q_m}),
\end{equation*}
with monotonically increasing dependence of the constant on $\Vert w_1\Vert_{\widehat{A}_{q_1,\infty}}$, and $[w_i]_{A_{q_i}}$, $i=2,\dots,m$. If $q>1$, then we can replace the space $L^{q_i,\min \left\{1, \frac{q_i}{p_i}\right\}}(w_i)$ by $L^{q_i,1}(w_i)$, $i=1,\dots,m$.
\end{theorem}


Let us recall some basic concepts on commutators (see \cite{corowe,gcrf,LOPTT}). Given a function $f\in L^1_{loc}(\mathbb R^n)$, the \textit{sharp maximal operator} $M^{\#}$ is defined by
\begin{equation*}
    M^{\#}f(x) \perdef \sup_{Q \ni x} \frac{1}{|Q|} \int_Q \left |f- \avgint_Q f \right |, \quad x \in \mathbb R^n.
\end{equation*}

If $b\in L^1_{loc}(\mathbb R^n)$ is such that $M^{\#}b \in L^{\infty}(\mathbb R^n)$, we say that $b$ is a function of \textit{bounded mean oscillation}, and we denote by $BMO$ the class of all these functions. For $b\in BMO$, we write 
\begin{equation*}
    \Vert b \Vert_{BMO} \perdef \Vert M^{\#}b \Vert_{L^{\infty}(\mathbb R^n)}.
\end{equation*}

Given an $m$-variable operator $T$ defined for measurable functions on $\mathbb R^n$, and measurable functions $b_1,\dots,b_m$, with $\vec b = (b_1,\dots,b_m)$, the $m$-\textit{variable commutators} $[\vec b,T]_i$, $i=1,\dots,m$, are formally defined for measurable functions $f_1,\dots,f_m$ by
\begin{align*}
    [\vec b,T]_i (f_1,\dots,f_m)(x) & \perdef b_i(x) T(f_1,\dots,f_m)(x) \\ & - T(f_1,\dots,f_{i-1},b_if_i,f_{i+1},\dots,f_m)(x),\quad x\in \mathbb R^n. 
\end{align*}
In particular, if $T= T^{\otimes}$, then
\begin{equation*}
    [\vec b,T^{\otimes}]_i (f_1,\dots,f_m)(x) = [b_i,T_i]f_i (x) \prod_{j\neq i} T_jf_j(x), \quad x \in \mathbb R^n.
\end{equation*}

Hence, these operators are, in fact, product-type operators, and we can follow the approach of Theorem~\ref{producttypeextrapol} to prove weighted bounds for them, using known estimates for commutators of one-variable operators, as we show in the next result.

\begin{corollary}\label{commutextrapol}
Fix $1\leq i \leq m$, and let $T_i$ be a linear operator such that for every weight $u\in A_2$,
\begin{equation*}
    T_i : L^2(u) \longrightarrow  L^2(u),
\end{equation*}
with constant bounded by $\phi_i([u]_{A_{2}})$, where $\phi_i:[1,\infty) \longrightarrow [0,\infty)$ is an increasing function. Also, for every $1\leq j \neq i \leq m$, let $T_j$ be a sub-linear operator such that for some $p_j>1$, and every weight $v_j \in A_{p_j}^{\mathcal R}$,
\begin{equation*}
    T_j: L^{p_j,1}(v_j) \longrightarrow L^{p_j,\infty}(v_j),
\end{equation*}
with constant bounded by $\varphi_j([v_j]_{A_{p_j}^{\mathcal R}})$, where $\varphi_j:[1,\infty) \longrightarrow [0,\infty)$ is an increasing function. Let $b_i \in BMO$, and fix an index $\ell \neq i$. If $\sum_{j\neq i} \frac{1}{p_j}<1$, then for all finite exponents $\frac{1}{p_i}<1-\sum_{j\neq i} \frac{1}{p_j}$, $1\leq q_{\ell} < \infty$, $1< q_1,\dots,q_{\ell-1},q_{\ell+1},\dots,q_{m}< \infty$, and $\frac{1}{q}=\frac{1}{q_1}+\dots+\frac{1}{q_m}$, and all weights $w_{\ell} \in \widehat {A}_{q_{\ell},\infty}$, and $w_j \in {A}_{q_j}$, $1\leq j \neq \ell \leq m$, 
\begin{equation*}
   [\vec b,T^{\otimes}]_i: L^{q_1,\min \left\{1, \frac{q_1}{p_1}\right\}}(w_1) \times \dots \times L^{q_m,\min \left \{1, \frac{q_m}{p_m} \right\}}(w_m)  \longrightarrow L^{q,\infty}(w_1^{q/q_1}\dots w_m^{q/q_m}),
\end{equation*}
with constant bounded by $\Phi_{\vec b}([ w_1]_{{A}_{q_1}},\dots,\Vert w_{\ell}\Vert_{\widehat{A}_{q_{\ell},\infty}},[ w_{\ell+1}]_{{A}_{q_{\ell+1}}},\dots)$, and as usual, $\Phi_{\vec b}:[1,\infty)^m \longrightarrow [0,\infty)$ is a function increasing in each variable. If $q>1$, then we can replace the space $L^{q_j,\min \left\{1, \frac{q_j}{p_j}\right\}}(w_j)$ by $L^{q_j,1}(w_j)$, $j=1,\dots,m$.
\end{corollary}

\begin{proof}
In virtue of \cite[Corollary 3.3]{cppcomm}, we have that for $1<r<\infty$, and every weight $v_i \in A_r$,
\begin{equation*}
    [b_i,T_i] : L^r(v_i) \longrightarrow L^r(v_i),
\end{equation*}
with constant bounded by
\begin{equation*}
    \varphi_i([v_i]_{A_r}) \perdef c_{n,r} \phi_i (C_{n,r}[v_i]_{A_r}^{\max \left \{1,\frac{1}{r-1}\right\}}) [v_i]_{A_r}^{\max \left \{1,\frac{1}{r-1}\right\}} \Vert b_i \Vert_{BMO}.
\end{equation*}

In particular,
\begin{equation*}
    [b_i,T_i] : L^r(v_i) \longrightarrow L^{r,\infty}(v_i),
\end{equation*}
with constant also bounded by $\varphi_i([v_i]_{A_r})$, and arguing as in the proof of \cite[Theorem 3.1.9]{thesis}, we get that for every weight $v_i \in A_r$, 
\begin{equation*}
     [b_i,T_i] : L^{r,1}(v_i) \longrightarrow L^{r,1}(v_i),
\end{equation*}
with constant bounded by $\widetilde c_{n,r} [v_i]_{A_r}^{\frac{2}{r-1}} \varphi_i(\widetilde C_{n,r}[v_i]_{A_r}^2)$.

Choosing $1< r \perdef p_i < \infty$, and writing $\frac{1}{p}\perdef \frac{1}{p_1}+\dots+\frac{1}{p_m}<1$, we deduce the desired result applying Proposition~\ref{producttype}, Remark~\ref{rmkproduct}, Corollary~\ref{cextrapolone}, and Remarks \ref{rmkmixedextrapolfix} and \ref{rmkmixedextrapolfixup}. 

Since $[\vec b,T^{\otimes}]_i$ is multi-sub-linear, the conclusion for $q>1$ follows from the standard extension argument in \cite[Page 256]{BS}.
\end{proof}

\begin{remark}
    It is worth mentioning that the function $\Phi_{\vec b}$ that we obtain is of the form
$\Phi_{\vec b} = \Vert b_i \Vert_{BMO} \Phi$, where $ \Phi:[1,\infty)^m \longrightarrow [0,\infty)$ is a function increasing in each variable and independent of $\vec b$.
\end{remark}

\subsection{\texorpdfstring{Averaging operators and $NBV$ Fourier multipliers}{Averaging operators and NBV Fourier multipliers}}\hfill\vspace{2.5mm}

We have seen in Theorem~\ref{producttypeextrapol} that, sometimes, we can use extrapolation techniques to avoid the application of some H\"older-type inequalities for Lorentz spaces. In the next result, we will see that we can also use extrapolation theorems to overcome the lack of Minkowski's integral inequality for the Lorentz quasi-norm $\Vert \cdot \Vert_{L^{q,\infty}(w)}$ when $q\leq 1$.

\begin{theorem}\label{averagingops}
Let $\{ T_1^{t_1}\}_{t_1\in \mathbb R}, \dots, \{ T_m^{t_m}\}_{t_m\in \mathbb R}$ be families of sub-linear operators defined for suitable measurable functions. For $i=1,\dots,m$, suppose that for some $p_i>1$, every $t_i\in \mathbb R$, and every weight $v_i \in A_{p_i}^{\mathcal R}$,
\begin{equation*}
    T_i^{t_i} : L^{p_i,1}(v_i) \longrightarrow L^{p_i,\infty}(v_i),
\end{equation*}
with constant bounded by $\varphi_i([v_i]_{A_{p_i}^{\mathcal R}})$, where $\varphi_i:[1,\infty) \longrightarrow [0,\infty)$ is an increasing function independent of $t_i$. For a measure $\mu$ on $\mathbb R^m$ such that $|\mu|(\mathbb R^m)<\infty$, consider the averaging operator
\begin{equation*}
    T_{\mu}(f_1,\dots,f_m)(x)\perdef \int_{\mathbb R^m} T_1^{t_1}f_1(x) \dots T_m^{t_m}f_m (x) d \mu (t_1,\dots,t_m), \quad x \in \mathbb R^n,
\end{equation*}
defined for suitable measurable functions $f_1,\dots,f_m$. If $\frac{1}{p_1}+\dots+\frac{1}{p_m}=\frac{1}{p}<1$, then for all exponents $1\leq q_1 < \infty$, $1< q_2,\dots,q_{m}< \infty$,  and $\frac{1}{q}=\frac{1}{q_1}+\dots+\frac{1}{q_m}$, and all weights $w_1 \in \widehat {A}_{q_1,\infty}$, and $w_i \in {A}_{q_i}$, $i=2,\dots,m$,
\begin{equation*}
    T_{\mu}: L^{q_1,\min \left\{1, \frac{q_1}{p_1}\right\}}(w_1) \times \dots \times L^{q_m,\min \left \{1, \frac{q_m}{p_m} \right\}}(w_m)  \longrightarrow L^{q,\infty}(w_1^{q/q_1}\dots w_m^{q/q_m}),
\end{equation*}
with monotonically increasing dependence of the constant on $\Vert w_1\Vert_{\widehat{A}_{q_1,\infty}}$, and $[w_i]_{A_{q_i}}$, $i=2,\dots,m$. If $q>1$, then we can replace the space $L^{q_i,\min \left\{1, \frac{q_i}{p_i}\right\}}(w_i)$ by $L^{q_i,1}(w_i)$, $i=1,\dots,m$.
\end{theorem}

\begin{proof}
Since $p>1$, in virtue of Minkowski's integral inequality (see \cite[Proposition 2.1]{schep} and \cite[Theorem 4.4]{barza}), we have that for all weights  $v_1 \in {A}_{p_1}^{\mathcal R}$, $v_i \in {A}_{p_i}$, $i=2,\dots,m$, and $v \perdef v_1^{p/p_1}\dots v_m^{p/p_m}$,
\begin{equation*}
   \Vert T_{\mu}(f_1,\dots,f_m) \Vert_{L^{p,\infty}(v)} \leq p' \int_{\mathbb R^m} \Vert T_1^{t_1}f_1 \dots T_m^{t_m}f_m  \Vert_{L^{p,\infty}(v)} d|\mu|(t_1,\dots,t_m),
\end{equation*}
and applying Proposition~\ref{producttype}, we get that
\begin{equation*}
     \Vert T_{\mu}(f_1,\dots,f_m) \Vert_{L^{p,\infty}(v)} \leq p' |\mu|(\mathbb R^m) \Phi([v_1]_{ A_{p_1}^{\mathcal R}},[v_2]_{A_{p_2}},\dots) \prod_{i=1}^m \Vert f_i \Vert_{L^{p_i,1}(v_i)},
\end{equation*}
where $\Phi:[1,\infty)^m\longrightarrow [0,\infty)$ is a function that increases in each variable. The desired result follows from Corollary~\ref{cextrapolone}, taking into account Remarks \ref{rmkmixedextrapolfix} and \ref{rmkmixedextrapolfixup}, and the standard extension argument in \cite[Page 256]{BS} if $q>1$.
\end{proof}

Now, let us recall some classical definitions from \cite[Chapter 8]{rudin}.

\begin{definition}
Given a function $f:\mathbb R \longrightarrow \mathbb R$, we say that $f$ is of \textit{bounded variation} if
\begin{equation*}
   V(f) \perdef \lim_{x \rightarrow \infty} \sup \sum_{j=1}^N |f(x_j)-f(x_{j-1})| \in \mathbb R,
\end{equation*}
where the supremum is taken over all $N$ and over all choices of $x_0,\dots,x_N$ such that $-\infty < x_0 <x_1 <\dots < x_N=x < \infty$. We call $V(f)$ the \textit{total variation} of $f$. The class of all functions $f$ of bounded variation will be denoted by $BV(\mathbb R)$. 

We say that a function $f \in BV(\mathbb R)$ is \textit{normalized} if 
$\lim_{x \rightarrow -\infty}f(x) = 0$. The class of these functions will be denoted by $NBV(\mathbb R)$. 

We say that a function $f:\mathbb R \longrightarrow \mathbb R$ is \textit{absolutely continuous} if for every $\varepsilon>0$, there exists $\delta>0$ such that
\begin{equation*}
    \sum_{j=1}^N (b_j-a_j) < \delta \quad \text{implies} \quad \sum_{j=1}^N |f(b_j)-f(a_j)|< \varepsilon, 
\end{equation*}
whenever $(a_1,b_1),\dots, (a_N,b_N)$ are disjoint segments. The class of all such functions will be denoted by $AC(\mathbb R)$. 
\end{definition}

Let us focus our attention to \cite[Corollary 3.8]{duo}. This result tells us that for a function $\mathfrak m \in NBV(\mathbb R)$ that is right-continuous at every point of $\mathbb R$, we can write
\begin{equation}\label{nbvm}
    \mathfrak m(\xi)=\int_{-\infty}^{\xi} d \mathfrak m(t)=\int _{\mathbb R} \chi_{(-\infty,\xi)}(t) d \mathfrak m(t)=\int _{\mathbb R} \chi_{(t,\infty)}(\xi) d \mathfrak m(t), \quad \xi \in \mathbb R,
\end{equation}
where $d \mathfrak m$ denotes the Lebesgue-Stieltjes measure associated with $\mathfrak m$. Therefore, the linear multiplier operator $T_{\mathfrak m}$ given by 
\begin{equation*}
    \widehat{T_{\mathfrak m} f} (\xi) \perdef \mathfrak m(\xi) \widehat f (\xi) = \int _{\mathbb R} \chi_{(t,\infty)}(\xi) \widehat f (\xi) d \mathfrak m(t) 
    , \quad \xi \in \mathbb R,
\end{equation*}
initially defined for Schwartz functions $f$ on $\mathbb R$, can be written as
\begin{equation*}
    T_{\mathfrak m} f(x) = \int _{\mathbb R} S_{t,\infty} f (x) d \mathfrak m(t), \quad x \in \mathbb R,
\end{equation*}
where
\begin{equation*}
   S_{t,\infty} f(x) \perdef \frac{1}{2}f(x) + \frac{i}{2} e^{2 \pi i t x } H(e^{-2 \pi i t \cdot }f)(x) \defper \frac{1}{2}f(x) + \frac{i}{2} e^{2 \pi i t x } H_t f(x).
\end{equation*}

As usual, $H$ denotes the \textit{Hilbert transform} on $\mathbb R$, defined as
\begin{equation*}
    Hf(x)\perdef \frac{1}{\pi} \lim_{\varepsilon \rightarrow 0^+} \int_{\{y\in \mathbb R \,  : \, |x-y|>\varepsilon \}} \frac{f(y)}{x-y}dy, \quad x \in \mathbb R,
\end{equation*}
and for a Schwartz function $f:\mathbb R \longrightarrow \mathbb R$, we denote by $\widehat f$ its \textit{Fourier transform}, given by 
\begin{equation*}
    \widehat f (\xi) \perdef \int_{\mathbb R} f(y)e^{-2 \pi i y \xi}d y, \quad \xi \in \mathbb R.
\end{equation*}

Since for $1<p<\infty$, $H:L^p(\mathbb R) \longrightarrow L^p(\mathbb R)$ with constant bounded by $c_p$, in virtue of Minkowski's integral inequality we conclude that
\begin{equation*}
    \Vert T_{\mathfrak m} f \Vert_{L^{p}(\mathbb R)} \leq \int _{\mathbb R} \Vert S_{t,\infty} f\Vert_{L^p(\mathbb R)} d|\mathfrak m|(t) \leq \frac{1+c_p}{2} V(\mathfrak m) \Vert f\Vert_{L^p(\mathbb R)},
\end{equation*}
and $\mathfrak m$ is an \textit{$L^p$ Fourier multiplier} for every $1<p<\infty$.

Inspired by this result, let us take a measure $\mu$ on $\mathbb R^m$ such that $|\mu|(\mathbb R^m)<\infty$, and define the function 
\begin{align}\label{twoem}
\begin{split}
   \mathfrak  m_{\mu}(\xi_1,\dots,\xi_m) & \perdef \int_{\{(r_1,\dots,r_m)\in \mathbb R^m \, : \, r_1 < \xi_1, \dots, r_m < \xi_m \}} d \mu (r_1,\dots,r_m) \\ & = \int_{\mathbb R^m} \left(\prod_{j=1}^m \chi_{(-\infty,\xi_j)}(r_j) \right) d \mu (r_1,\dots,r_m) \\ & = \int_{\mathbb R^m} \left( \prod_{j=1}^m \chi_{(r_j,\infty)}(\xi_j) \right) d \mu (r_1,\dots,r_m), 
    \end{split}
\end{align}
for $\xi_1,\dots,\xi_m \in \mathbb R$. It is clear that $\Vert \mathfrak m_{\mu}\Vert_{L^{\infty}(\mathbb R^m)}\leq |\mu|(\mathbb R^m)<\infty$, so it makes sense to consider the $m$-linear multiplier operator
\begin{equation*}
    T_{{\mathfrak m_{\mu}}}(\vec f)(x)\perdef \int_{\mathbb R} \dots \int_{\mathbb R} \mathfrak m_{\mu}(\vec \xi) \widehat f_1 (\xi_1)\dots \widehat f_m (\xi_m) e^{2 \pi i x (\xi_1+\dots+\xi_m)}d\xi_1 \dots d\xi_m, \quad x \in \mathbb R,
\end{equation*}
initially defined for Schwartz functions $f_1,\dots,f_m$. 

Arguing as we did in the linear case, and applying Fubini's theorem, we have that
\begin{align*}
    \begin{split}
        T_{{\mathfrak m_{\mu}}} (\vec f)(x) & = \int_{\mathbb R^m} \left( \prod_{j=1}^m \int_{\mathbb R} \chi_{(r_j,\infty)}(\xi_j)\widehat f_j (\xi_j) e^{2 \pi i x \xi_j} d \xi_j \right) d \mu (r_1,\dots,r_m) \\ & = \int_{\mathbb R^m} \left( \prod_{j=1}^m S_{r_j,\infty}f_j(x)\right) d \mu (r_1,\dots,r_m),
    \end{split}
\end{align*}
so $T_{{\mathfrak m_{\mu}}}$ is, in fact, an $m$-variable averaging operator, and we can follow the approach of Theorem~\ref{averagingops} to prove weighted bounds for it, exploiting known restricted weak-type bounds for the Hilbert transform, as we show in the next theorem, which generalizes and extends \cite[Corollary 1.4]{laura} and \cite[Corollary 1.6]{laura}, being this last result primarily based on Corollary 3.3.14, Proposition 3.4.1, and Theorem 3.4.7 in \cite{thesis}.

\begin{theorem}\label{bimultipliers}
Given exponents $1\leq q_1 <\infty$, $1<q_2,\dots,q_m<\infty$, and $\frac{1}{q}=\frac{1}{q_1}+\dots+\frac{1}{q_m}$, and weights $w_1 \in \widehat A_{q_1,\infty}$, $w_i \in A_{q_i}$, $i=2,\dots,m$, and $w=w_1^{q/q_1}\dots w_m^{q/q_m}$, 
\begin{equation}\label{eqbimultipliersc1}
    T_{{\mathfrak m_{\mu}}}:L^{q_1,1}(w_1) \times \prod_{i=2}^m L^{q_i,\min \left \{1, \frac{q_i}{p_i} \right\}}(w_i)  \longrightarrow L^{q,\infty}(w_1^{q/q_1}\dots w_m^{q/q_m}),
\end{equation}
for all exponents $1<p_2,\dots,p_m<\infty$ such that $\frac{1}{p_2}+\dots+\frac{1}{p_m}<1$, with constant bounded by $\Phi( \Vert w_1 \Vert_{\widehat{A}_{q_1,\infty}}, [w_2]_{A_{q_2}},\dots)$, where $\Phi:[1,\infty)^m \longrightarrow [0,\infty)$ is a function increasing in each variable. If $q>1$, then we can replace the space $L^{q_i,\min \left\{1, \frac{q_i}{p_i}\right\}}(w_i)$ by $L^{q_i,1}(w_i)$, $i=2,\dots,m$. Moreover, for all exponents $1\leq q_1,\dots,q_m<\infty$, $\frac{1}{q}=\frac{1}{q_1}+\dots+\frac{1}{q_m}$, and every weight $u \in \bigcap_{i=1}^m \widehat A_{q_i,\infty}$,
\begin{equation}\label{eqbimultipliersc2}
    T_{{\mathfrak m_{\mu}}}: L^{q_1,1}(u)\times \dots \times L^{q_m,1}(u) \longrightarrow L^{q,\infty}(u),
\end{equation}
with constant bounded by $\widetilde \Phi( \Vert u \Vert_{\widehat{A}_{q_1,\infty}}, \dots, \Vert u \Vert_{\widehat{A}_{q_m,\infty}})$, where $\widetilde \Phi:[1,\infty)^m \longrightarrow [0,\infty)$ is a function increasing in each variable.
\end{theorem}

\begin{proof}
It follows from \cite[Theorem 10]{prp} and \cite[Theorem 1.1]{lophilbert} that for every $p\geq 1$, and every weight $v\in A_{p}^{\mathcal R}$, $H:L^{p,1}(v) \longrightarrow L^{p,\infty}(v)$, with constant bounded by
\begin{equation*}
\phi([v]_{A_{p}^{\mathcal R}}) \perdef \left \{
    \begin{array}{lr}
       c_{p}[v]_{A_{p}^{\mathcal R}}^{p+1},  &  p>1, \\
       c [v]_{A_{1}}(1+\log [v]_{A_{1}}),  & p=1,
    \end{array}
    \right.
\end{equation*}
 so for every $\sigma \in \mathbb R$, and every $h \in L^{p,1}(v)$,
\begin{equation*}
    \Vert S_{\sigma,\infty} h \Vert_{L^{p,\infty}(v)} \leq \Vert h \Vert_{L^{p,\infty}(v)} + \Vert H_{\sigma} h \Vert_{L^{p,\infty}(v)} \leq \left( \frac{1}{p}+\phi([v]_{A_{p}^{\mathcal R}})\right) \Vert h \Vert_{L^{p,1}(v)}.
\end{equation*}

Choosing exponents $1<p_2,\dots,p_m<\infty$ such that $\frac{1}{p_2}+\dots+\frac{1}{p_m}<1$, and $0<\frac{1}{p_1}<1-\sum_{i=2}^m \frac{1}{p_i}$, and applying Theorem~\ref{averagingops}, we get that for all weights $u_1 \in A_{1}$, $w_i \in A_{q_i}$, $i=2,\dots,m$, with $\frac{1}{q_0}\perdef 1+ \frac{1}{q_2}+\dots+\frac{1}{q_m}$,
\begin{equation}\label{eqbimultipliers1}
    T_{{\mathfrak m_{\mu}}}:L^{1,\frac{1}{p_1}}(u_1) \times \prod_{i=2}^m L^{q_i,\min \left \{1, \frac{q_i}{p_i} \right\}}(w_i)  \longrightarrow L^{q_0,\infty}(u_1^{ q_0}w_2^{q_0/q_2}\dots w_m^{q_0/q_m}),
\end{equation}
with suitable control of the constant.

It was proved in \cite{laura} that $T_{{\mathfrak m_{\mu}}}$ can be approximated by some iterative $(\varepsilon,\delta)$-atomic operators $T_{\mathfrak m_{k_1,\dots,k_m}}$ associated to the multipliers $\mathfrak m_{k_1,\dots,k_m} (\vec \xi) \perdef {\mathfrak m_{\mu}}(\vec \xi) e^{-\sum_{j=1}^m \frac{\xi_j^2}{k_j}}$, and these also satisfy \eqref{eqbimultipliers1} uniformly on $k_1,\dots,k_m \in \mathbb N \setminus \{0\}$, so in virtue of Fatou's lemma and Theorem~\ref{lin}, we can replace the space $L^{1,\frac{1}{p_1}}(u_1)$ by $L^{1}(u_1)$ in \eqref{eqbimultipliers1}, and extrapolating the first variable with Corollary~\ref{cextrapolone}, we establish \eqref{eqbimultipliersc1}, taking into account the standard extension argument in \cite[Page 256]{BS} if $q>1$.

Finally, \eqref{eqbimultipliersc2} follows immediately from the one-weight weak-type $(1,\dots,1,\frac{1}{m})$ bounds for $T_{{\mathfrak m_{\mu}}}$ in \cite[Corollary 1.4]{laura} and our extrapolation scheme in Corollary~\ref{corollmultiextrapolup}.
\end{proof}

\begin{remark}
For simplicity, fix $m=2$. Observe that for $1\leq p_1, p_2 <\infty$, $\frac{1}{p}=\frac{1}{p_1}+\frac{1}{p_2}<1$, $v_1 \in A_{p_1}^{\mathcal R}$, $v_2\in A_{p_2}^{\mathcal R}$, and $v=v_1^{p/p_1}v_2^{p/p_2}$, and in virtue of \cite[Lemma 2.2.1]{thesis}, we get that for $\sigma,\tau \in \mathbb R$,
\begin{align*}
&\Vert (S_{\sigma,\infty}f) (S_{\tau,\infty}g) \Vert_{L^{p,\infty}(v)} \\ & \leq \Vert fg \Vert_{L^{p,\infty}(v)} + \Vert f H_{\tau}g \Vert_{L^{p,\infty}(v)} + \Vert gH_{\sigma}f \Vert_{L^{p,\infty}(v)}  +\Vert (H_{\sigma}f) (H_{\tau}g) \Vert_{L^{p,\infty}(v)} \\ & \leq \varphi([v_1]_{A_{p_1}^{\mathcal R}},[v_2]_{A_{p_2}^{\mathcal R}}) \Vert f \Vert_{L^{p_1,1}(v_1)}\Vert g \Vert_{L^{p_2,1}(v_2)}+\Vert (H_{\sigma}f) (H_{\tau}g) \Vert_{L^{p,\infty}(v)},
\end{align*}
where $\varphi:[1,\infty)^2 \longrightarrow [0,\infty)$ is a function increasing in each variable and independent of $\sigma$ and $\tau$. Hence, if we could prove restricted weak-type bounds for the point-wise product of Hilbert transforms, we would be able to transfer them to the operator $T_{{\mathfrak m_{\mu}}}$ using our extrapolation results, arguing as in the proof of Theorem~\ref{bimultipliers}.
\end{remark}

If we take right-continuous functions $\mathfrak m_1,\dots,\mathfrak m_m \in NBV(\mathbb R)$, we can easily construct a function like \eqref{twoem} by merely considering their product, since by \eqref{nbvm},
\begin{align*}
   (\mathfrak m_1 \otimes \dots \otimes \mathfrak m_m) (\xi_1,\dots,\xi_m) & \perdef \mathfrak m_1(\xi_1) \dots \mathfrak m_m(\xi_m) \\ & = \int_{\mathbb R} \dots \int_{\mathbb R}  \left(\prod_{j=1}^m\chi_{(r_j,\infty)}(\xi_j) \right) d\mathfrak m_1(r_1) \dots d\mathfrak m_m(r_m),
\end{align*}
and 
\begin{equation*}
    \Vert \mathfrak m_1 \otimes \dots \otimes \mathfrak m_m \Vert_{L^{\infty}(\mathbb R^m)} \leq \int_{\mathbb R} \dots \int_{\mathbb R} d|\mathfrak m_1|(r_1) \dots d|\mathfrak m_m|(r_m) = V(\mathfrak m_1)\dots V(\mathfrak m_m) < \infty.
\end{equation*}

The following result, which is a combination of Theorems 8.17 and 8.18 in \cite{rudin}, will allow us to construct another simple yet more elaborate example of a function like \eqref{twoem}, along with many examples of functions in $NBV(\mathbb R)$. 

\begin{theorem}\label{thrudin}
If $\psi \in L^1(\mathbb R)$, and for every $x\in \mathbb R$, 
\begin{equation*}
    f(x) \perdef \int_{-\infty}^x \psi(r)dr,
\end{equation*}
then $f\in NBV(\mathbb R)\cap AC(\mathbb R)$, and $f'=\psi$ almost everywhere. Conversely, if $f\in NBV(\mathbb R) \cap AC(\mathbb R)$, then $f$ is differentiable almost everywhere, $f'\in L^1(\mathbb R)$, and for every $x\in \mathbb R$,
\begin{equation*}
    f(x)=\int_{-\infty}^x f'(r)dr.
\end{equation*}
\end{theorem}

An immediate consequence of Theorem~\ref{thrudin} is the next lemma.


\begin{lemma}\label{integral}
For a function $\mathfrak m\in NBV(\mathbb R) \cap AC(\mathbb R)$, and for all $\xi_1,\dots,\xi_m \in \mathbb R$,
\begin{equation*}
    \mathfrak{m}_* (\xi_1,\dots,\xi_m) \perdef \mathfrak m(\min\{\xi_1,\dots,\xi_m\}) = \int _{\mathbb R} \left(\prod_{j=1}^m \chi_{(r,\infty)}(\xi_j) \right) \mathfrak m'(r)dr,
\end{equation*}
with $\Vert  \mathfrak{m}_* \Vert_{L^{\infty}(\mathbb R^m)} \leq \Vert \mathfrak m' \Vert_{L^1(\mathbb R)}<\infty$.
\end{lemma}
\begin{proof}
By Theorem~\ref{thrudin}, we have that
\begin{align*}
   \mathfrak m(\min\{\xi_1,\dots,\xi_m\})&= \int_{-\infty}^{\min\{\xi_1,\dots,\xi_m\}} \mathfrak m'(r)dr \\ & = \int _{\mathbb R} \left(\prod_{j=1}^m \chi_{(-\infty,\xi_j)}(r) \right) \mathfrak m'(r)dr = \int _{\mathbb R} \left( \prod_{j=1}^m \chi_{(r,\infty)}(\xi_j) \right) \mathfrak m'(r)dr.
\end{align*}
\end{proof}

The function $\mathfrak m_*$ is an example of a function like \eqref{twoem} with the measure $\mu$ restricted to $\mathbb R$. More generally, we can take a subset $E \subseteq \mathbb R^m$, and a measure $\nu$ on $E$ such that $|\nu|(E)<\infty$, and consider the function
\begin{equation}\label{funcme}
   \mathfrak m_{\nu,E}(\xi_1,\dots,\xi_m) \perdef \nu (E \cap R_{\xi_1,\dots,\xi_m}), \quad \Vert \mathfrak m_{\nu,E} \Vert_{L^{\infty}(\mathbb R^m)}\leq |\nu|(E)<\infty,
\end{equation}
with $R_{\xi_1,\dots,\xi_m} \perdef \{(r_1,\dots,r_m)\in \mathbb R^m : r_1 < \xi_1, \dots, r_m < \xi_m\}$. This example was suggested to us by M. J. Carro. See Figure~\ref{figme} for a pictorial representation of it.

\begin{figure}[ht]
\def\svgwidth{0.5\linewidth}
{\hspace{-2.8cm}
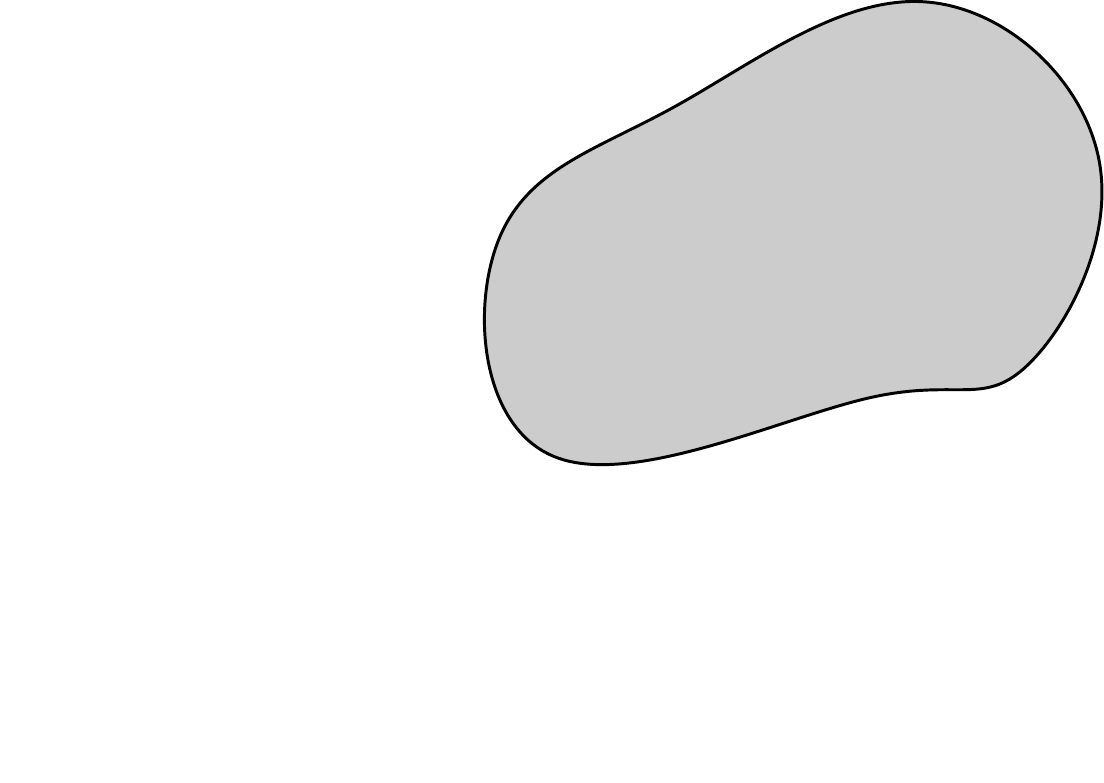
\hspace{2.8cm}
}
\rule{.8\textwidth}{.4pt}
\caption[Pictorial representation of the function in \eqref{funcme}]{Pictorial representation of the function in \eqref{funcme} for $m=2$.}
\label{figme}
\end{figure}

In the particular case when $E=(0,1)^m$, the unit $m$-cube, and $\nu$ is the Lebesgue measure on $E$, we obtain that
\begin{equation*}
  \mathfrak  m_{\nu,E}(\xi_1,\dots,\xi_m) = \prod_{j=1}^m \min \{1,\xi_j\} \chi_{\{r\in \mathbb R \, : \, r > 0\}}(\xi_j).
\end{equation*}

We get a more elaborate example if we consider the unit half-ball, 
\begin{equation*}
    E=\{(x_1,\dots,x_m)\in \mathbb R^m : x_m \geq 0, x_1^2+\dots+x_m^2 \leq 1\},
\end{equation*}
and $\nu$ the Lebesgue measure on $E$ (see \cite[Page 112]{thesis}). 

\section{\texorpdfstring{To $A_{\vec{P}}$ and beyond}{To AP and beyond}}

In this section, we adapt our techniques to establish multi-variable weak-type extrapolation theorems for tuples of measures analogous to the ones presented in \cite{multiap2,multiap1,zoe}, omitting classical constructions of weights involving Rubio de Francia's iteration algorithm. We will deduce our extrapolation schemes from one-variable off-diagonal results obtained with an approach that differs fundamentally from \cite[Theorem 3.1]{multiap1} and \cite[Theorem 5.1]{duoextrapol}. 

This time, the underlying measure won't be the Lebesgue measure as before, but a more general doubling measure $\mu$. Sometimes, we will invoke statements that appear in the literature for the Lebesgue measure but are also valid for $\mu$, with almost identical proofs, and apply them in the form that suits our needs without prior notice. 

\subsection{Weighted weights}\hfill\vspace{2.5mm}\label{weiweights}\label{ss81}

Before proceeding, let us push our understanding of $A_p^{\mathcal R}$ further. The following is a restricted weak-type version of \cite[Lemma 2.1]{CUMP}.

\begin{lemma}\label{ApRA1}
    Let $1\leq p <\infty$, and $u\in A_1$. If $v \in A_p^{\mathcal R}(u)$, then $uv \in A_p^{\mathcal R}$, with
    $$
[uv]_{A_p^{\mathcal R}} \leq [u]_{A_1} [v]_{A_p^{\mathcal R}(u)}.
    $$
\end{lemma}

\begin{proof}
    This result follows directly from the definitions of $[uv]_{A_p^{\mathcal R}}$ and $[v]_{A_p^{\mathcal R}(u)}$, since for every cube $Q\subseteq \mathbb R^n$, 
    $$
\frac{1}{|Q|}\left(\int_Q uv \right)^{1/p} \leq  \frac{ [u]_{A_1}}{u(Q)}\left(\int_Q v du \right)^{1/p} \essinf_{x \in Q} u(x),
    $$
and
$$
 \Vert \chi_Q u^{-1}v^{-1}\Vert_{L^{p',\infty}(uv)} \essinf_{x \in Q} u(x) \leq \Vert \chi_Q v^{-1}\Vert_{L^{p',\infty}(vdu)}.
$$
\end{proof}

\begin{remark}
  We showed in \cite[Lemma 2]{prp} that if $u\in A_{\infty}$, then $uv \in A_{\infty}$ if, and only if $v \in A_{\infty}(u)$, so it would be interesting to study a version of Lemma~\ref{ApRA1} relating $[u]_{ A_q^{\mathcal R}}$, $[v]_{A_p^{\mathcal R}(u)}$, and $[uv]_{A_r^{\mathcal R}}$ for suitable exponents $1\leq p,q,r<\infty$ other than those above.
\end{remark}

The next result is an extension of \cite[Lemma 2.12]{cs}, and the proof is similar.

\begin{lemma}\label{pesa1}
Let $\mu$ be a positive, 
locally finite, doubling Borel measure on $\mathbb R^n$, with doubling constant $1\leq C_{\mu}<\infty$, and fix $0 < \delta < 1$. For a $\mu$-measurable function $0 \neq h\in L^1_{loc}(\mathbb R^n,\mu)$ such that $M_{\mu}h < \infty$ $\mu$-a.e., and $w \in A_1(\mu)$, let $v=(M_{\mu}h)^{1-\delta} w^{\delta}$. Then, $v\in A_1(\mu)$, and
\begin{equation}\label{eqpesa11}
[v]_{A_1(\mu)} \leq  \frac{\kappa C_{\mu}^{\kappa} }{\delta} [w]_{A_1(\mu)},
\end{equation}
where $\kappa > 1$ is a universal constant. In particular, if $d\mu(y)=u(y)dy$ and $u\in A_1$, then $(u M_{u}h)^{1-\delta} (uw)^{\delta} \in A_1$, with constant independent of $h$.
\end{lemma}

\begin{proof}
    In virtue of \cite[Theorem 1.1]{hpr}, taking
    $$
D_{\mu} \perdef \log_2 C_{\mu}, \quad \text{ and } \quad r \perdef 1+\frac{1}{6 \cdot 800^{D_{\mu}}[w]_{A_1(\mu)}}, 
    $$
    we have that $w^r \in A_1(\mu)$, with
    $$
[w^r]_{A_1(\mu)} \leq (2\cdot 4 ^{D_{\mu}}[w]_{A_1(\mu)})^r \leq 4 e^{1/e} 16^{D_{\mu}} [w]_{A_1(\mu)}.  
    $$

For $p\perdef \frac{r}{\delta}>1$, $0<p'(1-\delta)<1$, and from \cite[Proposition 2.10]{kurki} and \cite[(3.11)]{hpr} (see also \cite[Exercise 2.1.1]{grafclas}), we deduce that $(M_{\mu} h)^{p'(1-\delta)} \in A_1(\mu)$, with
\begin{equation}\label{eqpesa13}
[(M_{\mu} h)^{p'(1-\delta)}]_{A_1(\mu)} \leq \frac{\kappa_0 C_{\mu}^{\kappa_0}}{1-p'(1-\delta)},
\end{equation}
where $\kappa_0 > 1$ is a universal constant; $\kappa_0=6$ works.

Applying H\"older's inequality with exponent $p$, we get that for every cube $Q \subseteq \mathbb R^n$ such that $0<\mu(Q) <\infty$, 
and $\mu$-a.e. $x\in Q$,
\begin{align*}
    \frac{1}{\mu(Q)} \int_Q v d \mu & \leq \left(\frac{1}{\mu(Q)} \int_Q w^r d \mu \right) ^{1/p} \left( \frac{1}{\mu(Q)} \int_Q (M_{\mu} h)^{p'(1-\delta)} d \mu\right)^{1/p'} \\ & \leq (4 e^{1/e} 16^{D_{\mu}} [w]_{A_1(\mu)})^{1/p} w(x)^{\delta} \left(\frac{\kappa_0 C_{\mu}^{\kappa_0}}{1-p'(1-\delta)} \right)^{1/p'} M_{\mu}h(x)^{1-\delta} \\ & \leq 28 e^{1/e} \kappa_0 \frac{ C_{\mu}^{14+\kappa_0}}{\delta} [w]_{A_1(\mu)} v(x), 
\end{align*}
and \eqref{eqpesa11} holds for $\kappa=28 e^{1/e} \kappa_0$. 

Finally, for $u\in A_1$, it follows from \eqref{eqpesa11}, Lemma~\ref{ApRA1}, and \cite[Proposition 7.1.5]{grafclas} that $uv\in A_1$, with
\begin{equation*}
[(u M_{u}h)^{1-\delta} (uw)^{\delta}]_{A_1} \leq \frac{2^{n\kappa}\kappa}{\delta} [u]_{A_1} ^{\kappa+1}[w]_{A_1(u)}.
\end{equation*}
\end{proof}

We present a novel generalization of Theorem 2.7 and Corollary 2.8 in \cite{cgs}, following a different approach.

\begin{lemma}\label{wweights1}
Let $\mu$ be a positive, 
locally finite, doubling Borel measure on $\mathbb R^n$, with doubling constant $1\leq C_{\mu}<\infty$, and fix $1<p<\infty$. For a $\mu$-measurable function $0 \neq h\in L^1_{loc}(\mathbb R^n,\mu)$ such that $M_{\mu}h < \infty$ $\mu$-a.e., and $v\in A_1(\mu)$, let $w=(M_{\mu}h)^{1-p}v$. Then, $w \in A_p^{\mathcal R}(\mu)$, and
$$
[w]_{A_p^{\mathcal R}(\mu)} \leq  (\kappa C_{\mu}^{\kappa})^{1/p'} [v]_{A_1(\mu)}^{1/p},
$$
where $\kappa>1$ is a universal constant. In particular, if $d\mu(y)=u(y)dy$ and $u\in A_1$, then $(M_{u}h)^{1-p}uv \in A_p^{\mathcal R}$, with constant independent of $h$.
\end{lemma}

\begin{proof}
We first discuss the case $v=1$; that is, $w=(M_{\mu}h)^{1-p}$, by adapting an unpublished argument for $d \mu(y)=dy$ by C. P\'erez (\cite{cpapr}). 

From the definition of $[w]_{A_p^{\mathcal R}(\mu)}$ and an elementary estimate, we get that
\begin{align}\label{eqwweights11}
\begin{split}
[w]_{A_p^{\mathcal R}(\mu)}^p & = \sup_{Q} \left(\frac{1}{\mu(Q)}\int_Q w d\mu \right) \frac{\Vert \chi_Q (M_{\mu}h)^{p-1}\Vert_{L^{p',\infty}(w d \mu)}^p}{\mu(Q)^{p-1}} \\ &  = \sup_{Q} \left(\frac{1}{\mu(Q)}\int_Q w d\mu \right) \left(\frac{\Vert \chi_Q M_{\mu}h \Vert_{L^{p,\infty}(w d \mu)}^{p}}{\mu(Q)}\right)^{p-1}
     \\ & = \sup_{Q} \left(\frac{1}{\mu(Q)}\int_Q w d\mu \right) \left(\sup_{t>0} \frac{t^p}{\mu(Q)} \int_{\{x\in Q \, : \, M_{\mu}h(x)>t\}} (M_{\mu}h)^{1-p} d\mu \right)^{p-1} \\ & \leq \sup_{Q} \left(\frac{1}{\mu(Q)}\int_Q w d\mu\right) \left( \sup_{t>0} \frac{t \mu(\{x\in Q : M_{\mu}h(x)>t\})}{\mu(Q)} \right)^{p-1}.
\end{split}
\end{align}

It is well-known (see \cite[Page 160]{gcrf}) that for every cube $Q\subseteq \mathbb R^n$ such that $0<\mu(Q) <\infty$, and $\mu$-a.e. $x\in Q$,
$$
M_{\mu}(h \chi_{\mathbb R^n \setminus 3Q})(x) \leq C_{\mu}^2 \, \mu \text{-} \essinf_{y\in Q} M_{\mu}(h \chi_{\mathbb R^n \setminus 3Q})(y) \defper J_{\mu,Q},
$$
so $M_{\mu}h(x) \leq M_{\mu}(h \chi_{3Q})(x)+J_{\mu,Q}$. Hence, for $t>0$,
\begin{align}\label{eqwweights12}
\begin{split}
 &\frac{t \mu(\{x\in Q : M_{\mu}h(x)>t\})}{\mu(Q)}   \leq \frac{t \mu(\{x\in Q :  M_{\mu}(h \chi_{3Q})(x)+J_{\mu,Q} >t\})}{\mu(Q)} \\ & \leq  \frac{t \mu(\{x\in Q :  J_{\mu,Q} >\frac{t}{2}\})}{\mu(Q)} + \frac{t \mu(\{x\in Q :  M_{\mu}(h \chi_{3Q})(x) >\frac{t}{2}\})}{\mu(Q)}\defper I+ II.
 \end{split}
\end{align}

Now,
\begin{equation}\label{eqwweights13}
I \leq 2 J_{\mu,Q} \leq 2 C_{\mu}^2 \, \mu \text{-} \essinf_{y\in Q} M_{\mu}h(y),
\end{equation}
and 
\begin{align}\label{eqwweights14}
\begin{split}
II & \leq \frac{t\mu \left(\left\{M_{\mu}(h \chi_{3Q}) > \frac{t}{2}\right\}\right)}{\mu(Q)}  \leq \frac{2}{\mu(Q)} \Vert M_{\mu}(h \chi_{3Q}) \Vert_{L^{1,\infty}(\mu)} \\ & \leq \frac{\kappa_0 C_{\mu}^{\kappa_0} }{\mu(Q)}\int_{3Q} |h| d \mu \leq \frac{\kappa_0 C_{\mu}^{\kappa_0+2} }{\mu(3Q)}\int_{3Q} |h| d \mu \leq \kappa_0 C_{\mu}^{\kappa_0+2} \mu \text{-}\essinf_{y\in Q} M_{\mu}h(y), 
\end{split}
\end{align}
with $\kappa_0 \geq 2$ universal. Note that in the third inequality of \eqref{eqwweights14}, we have used the weak-type $(1,1)$ bound for $M_{\mu}$ in \cite[Exercise 2.1.1]{grafclas}.

Combining \eqref{eqwweights11}, \eqref{eqwweights12}, \eqref{eqwweights13}, and \eqref{eqwweights14}, we obtain that
\begin{align*}
[w]_{A_p^{\mathcal R}(\mu)}^p & \leq \sup_{Q} \left(\frac{1}{\mu(Q)}\int_Q (M_{\mu}h)^{1-p} d\mu\right) \big(2\kappa_0 C_{\mu}^{\kappa_0+2} \mu \text{-}\essinf_{y\in Q} M_{\mu}h(y)\big)^{p-1} \\ & \leq  (2\kappa_0 C_{\mu}^{\kappa_0+2})^{p-1}.
\end{align*}

For the general case; that is, $w=(M_{\mu}h)^{1-p}v$, with $v\in A_1(\mu)$, note that 
$$
\mu \text{-}\esssup_{y\in Q} v(y)^{\frac{1}{1-p}} \leq \big(\mu \text{-}\essinf_{y\in Q} v(y)\big)^{\frac{1}{1-p}} \leq \left([v]_{A_1(\mu)} \frac{\mu(Q)}{\int_Q v d\mu}\right)^{\frac{1}{p-1}}\defper K^{-1},
$$
so for $t>0$,
$$
\mu(\{x\in Q : M_{\mu}h(x)v(x)^{\frac{1}{1-p}}>t\}) \leq \mu(\{x\in Q : M_{\mu}h(x)>K t\}),
$$
and we can reuse \eqref{eqwweights12}, \eqref{eqwweights13}, and \eqref{eqwweights14} to conclude that $w \in A_p^{\mathcal R}(\mu)$, with
\begin{align*}
[w]_{A_p^{\mathcal R}(\mu)}^p &   = \sup_{Q} \left(\frac{1}{\mu(Q)}\int_Q w d\mu \right) \left(\frac{\Vert \chi_Q v^{\frac{1}{1-p}} M_{\mu}h \Vert_{L^{p,\infty}(w d \mu)}^{p}}{\mu(Q)}\right)^{p-1} \\ & \leq \sup_{Q} \left(\frac{1}{\mu(Q)}\int_Q w d\mu\right)  \left(\sup_{t>0} \frac{t \mu(\{x\in Q : M_{\mu}h(x)v(x)^{\frac{1}{1-p}}>t\})}{\mu(Q)} \right)^{p-1} \\ & \leq \sup_{Q} \left(\frac{1}{\mu(Q)}\int_Q w d\mu\right)  \left(\sup_{\tau>0} \frac{K^{-1} \tau \mu(\{x\in Q : M_{\mu}h(x)>\tau\})}{\mu(Q)} \right)^{p-1} \\ & \leq [v]_{A_1(\mu)} \sup_{Q} \left(\frac{1}{\int_Q v d\mu }\int_Q (M_{\mu}h)^{1-p} v d\mu\right) \big(2\kappa_0 C_{\mu}^{\kappa_0+2} \mu \text{-}\essinf_{y\in Q} M_{\mu}h(y)\big)^{p-1} \\ & \leq  (2\kappa_0 C_{\mu}^{\kappa_0+2})^{p-1} [v]_{A_1(\mu)}.
\end{align*}

Finally, if $d\mu(y)=u(y)dy$ and $u\in A_1$, it follows from Lemma~\ref{ApRA1} that $uw \in A_p^{\mathcal R}$, with
$$
[(M_{u}h)^{1-p}uv]_{A_p^{\mathcal R}} \leq (\mathfrak c_n^{\kappa_0} [u]_{A_1}^{\kappa_0+2})^{1/p'} [u]_{A_1} [v]_{A_1(u)}^{1/p}.
$$
\end{proof}

\begin{remark}\label{rmkwweights1}
    It is worth mentioning that for $u$, $v$, and $h$ as above, and $0<\delta<1$, $(M_u h)^{\delta (1-p)}uv \in A_p$, $1\leq p <\infty$. Indeed, in virtue of the argument we used in \eqref{eqpesa13}, \cite[Theorem 4.2]{david}, and \cite[Lemma 2.1]{CUMP}, 
$$
[(M_u h)^{\delta (1-p)}uv]_{A_p} \leq \left(\frac{\mathfrak c_n ^{\kappa}[u]_{A_1}^{\kappa}}{1-\delta}\right)^{p-1} [u]_{A_1}^p [v]_{A_1(u)},
$$
with $\kappa >1$ universal.
\end{remark}

Here, there is another new extension of Theorem 2.7 and Corollary 2.8 in \cite{cgs}.

\begin{lemma}\label{wweights2}
Let $u \in A_1$, and fix $1<p<\infty$. For a function $0 \neq h\in L^1_{loc}(u)$ such that $M_{u}h < \infty$ a.e., and $v\in A_1$, let $w=(u M_{u}h)^{1-p}v$. Then, $w \in A_p^{\mathcal R}$, and
$$
[w]_{A_p^{\mathcal R}} \leq \phi_n([u]_{A_1})^{1/p'} [v]_{A_1}^{1/p},
$$
where $\phi_n: [1,\infty) \longrightarrow [0,\infty)$ is an increasing function depending only on the dimension $n$.
\end{lemma}

\begin{proof}
    As in the proof of Lemma~\ref{wweights1}, we first discuss the case $w=(u M_{u}h)^{1-p}$ by adapting the argument for $u=1$ by C. P\'erez (\cite{cpapr}). We have that
    \begin{align}\label{eqwweights21}
    \begin{split}
        [w]_{A_p^{\mathcal R}}^p & = \sup_{Q} \left(\avgint_Q w\right) \left(\frac{\Vert \chi_Q u M_u h \Vert_{L^{p,\infty}(w)}^p}{|Q|}\right)^{p-1} \\ & = \sup_{Q} \left(\avgint_Q w\right)  \left(\sup_{t>0} \frac{t^p}{|Q|} \int_{\{x\in Q \, : \, u(x)M_{u}h(x)>t\}} (u M_{u}h)^{1-p}  \right)^{p-1} \\ & \leq \sup_{Q} \left(\avgint_Q w\right) \left(\sup_{t>0} \frac{t |\{x\in Q : u(x)M_{u}h(x)>t\}|}{|Q|} \right)^{p-1}.
    \end{split}
    \end{align}


   We know from \cite[Page 160]{gcrf} and \cite[Proposition 7.1.5]{grafclas} that, for every cube $Q\subseteq \mathbb R^n$, and a.e. $x\in Q$, 
$$
M_{u}(h \chi_{\mathbb R^n \setminus 3Q})(x) \leq 3^n [u]_{A_1} \essinf_{y\in Q} M_{u}(h \chi_{\mathbb R^n \setminus 3Q})(y) \defper J_{u,Q},
$$
so $M_{u}h(x) \leq M_{u}(h \chi_{3Q})(x)+J_{u,Q}$. Hence, for $t>0$,
\begin{align}\label{eqwweights22}
\begin{split}
 \frac{ t |\{x\in Q : u(x)M_{u}h(x)>t\}|}{|Q|} & 
 \leq \frac{t |\{x\in Q : u(x)J_{u,Q} >\frac{t}{2}\}|}{|Q|} \\ & + \frac{t |\{x\in Q :  u(x) M_{u}(h \chi_{3Q})(x) > \frac{t}{2}\}|}{|Q|} \defper I+II.
\end{split}
\end{align}

For the first term, we get that
\begin{align}\label{eqwweights23}
\begin{split}
I &= \frac{1}{|Q|}\int_{\{x\in Q \, : \,  u(x)J_{u,Q} >\frac{t}{2}\}} t \leq \frac{2J_{u,Q}}{|Q|} \int_{\{x\in Q \, : \,  u(x)J_{u,Q} >\frac{t}{2}\}} u \leq 2J_{u,Q} \avgint_{Q} u \\ & \leq 2 \cdot 3^n [u]_{A_1}^2 \big(\essinf_{y\in Q} u(y) \big) \big(\essinf_{y\in Q} M_{u}h(y)\big)  \leq 2\cdot 3^n [u]_{A_1}^2  \essinf_{y\in Q} u(y) M_{u}h(y).
\end{split}
\end{align}

For the second term, we apply the Sawyer-type inequality in \cite[Theorem 1]{prp} (see also \cite[Theorem 1.3]{CUMP}) to obtain that
\begin{align}\label{eqwweights24}
\begin{split}
II & \leq \frac{t}{|Q|} \left| \left\{x \in \mathbb R^n :  \frac{M_{u}(h \chi_{3Q})(x)}{u(x)^{-1}} > \frac{t}{2}\right\}\right| \leq \frac{2}{|Q|} \left \Vert \frac{M_{u}(h \chi_{3Q})}{u^{-1}} \right \Vert_{L^{1,\infty}(\mathbb R^n)} \\ & \leq  \frac{\mathscr E_{1,1}^n([u]_{A_1})}{|Q|} \int_{3Q} |h| u  \leq 3^n \mathscr E_{1,1}^n([u]_{A_1}) [u]_{A_1} \frac{u(Q)}{|Q|} \cdot \frac{1}{u(3Q)}\int_{3Q} |h| u \\ & \leq 3^n \mathscr E_{1,1}^n([u]_{A_1}) [u]_{A_1}^2 \big(\essinf_{y\in Q} u(y) \big) \big(\essinf_{y\in Q} M_u h(y)\big) \\ & \leq 3^n \mathscr E_{1,1}^n([u]_{A_1}) [u]_{A_1}^2 \essinf_{y\in Q} u(y) M_u h(y),
\end{split}
\end{align}
where $\mathscr E_{1,1}^n: [1,\infty) \longrightarrow [0,\infty)$ is an increasing function that depends only on $n$.

Combining \eqref{eqwweights21}, \eqref{eqwweights22}, \eqref{eqwweights23}, and \eqref{eqwweights24}, we deduce that
\begin{align*}
[w]_{A_p^{\mathcal R}}^p & \leq \sup_{Q} \left(\avgint_Q (u M_{u}h)^{1-p} \right) \left(3^n (\mathscr E_{1,1}^n([u]_{A_1})+2)[u]_{A_1}^2 \essinf_{y\in Q} u(y) M_u h(y)\right)^{p-1} \\ & \leq  \left(3^n (\mathscr E_{1,1}^n([u]_{A_1})+2)[u]_{A_1}^2 \right)^{p-1}.
\end{align*}

Now, for $v\in A_1$,  
$
\esssup_{y\in Q} v(y)^{\frac{1}{1-p}} 
\leq \left([v]_{A_1} \frac{|Q|}{ v(Q)}\right)^{\frac{1}{p-1}}\defper K^{-1},
$
and for $t>0$,
$$
|\{x\in Q : u(x) M_{u}h(x)v(x)^{\frac{1}{1-p}}>t\}| \leq |\{x\in Q : u(x) M_{u}h(x)>K t\}|.
$$
Thus, if $w=(uM_uh)^{1-p}v$, then by \eqref{eqwweights22}, \eqref{eqwweights23}, and \eqref{eqwweights24},
\begin{align*}
[w]_{A_p^{\mathcal R}}^p &   = \sup_{Q} \left(\avgint_Q w \right) \left(\frac{\Vert \chi_Q v^{\frac{1}{1-p}} u M_{u}h \Vert_{L^{p,\infty}(w)}^{p}}{|Q|}\right)^{p-1} \\ & \leq \sup_{Q} \left(\avgint_Q w \right)  \left(\sup_{t>0} \frac{t |\{x\in Q : u(x)M_{u}h(x)v(x)^{\frac{1}{1-p}}>t\}|}{|Q|} \right)^{p-1} \\ & \leq \sup_{Q} \left(\avgint_Q w \right)  \left(\sup_{\tau>0} \frac{K^{-1} \tau |\{x\in Q : u(x) M_{u}h(x)>\tau\}|}{|Q|} \right)^{p-1}  \\ & \leq \phi_n ([u]_{A_1})^{p-1} [v]_{A_1} \sup_{Q} \left(\frac{1}{v(Q)}\int_Q (uM_{u}h)^{1-p} v \right) \big(\essinf_{y\in Q} u(y) M_u h(y)\big)^{p-1} \\ & \leq \phi_n ([u]_{A_1})^{p-1} [v]_{A_1},
\end{align*}
with
$$
\phi_n(\xi) = 3^n (\mathscr E_{1,1}^n(\xi)+2)\xi^2, \quad \xi \geq 1.
$$
\end{proof}

\begin{remark}
    As in Remark~\ref{rmkwweights1}, for $u$, $v$, and $h$ as in Lemma~\ref{wweights2}, and $0<\delta<1$, $(u (M_u h)^{\delta})^{1-p}v \in A_p$, $1\leq p <\infty$, with 
$$
[(u (M_u h)^{\delta} )^{1-p}v]_{A_p} \leq \left(\frac{\mathfrak c_n ^{\kappa}[u]_{A_1}^{\kappa+1}}{1-\delta}\right)^{p-1} [v]_{A_1},
$$
and $\kappa >1$ universal.
\end{remark}

It is clear that a factorization of $A_p^{\mathcal R}$ has to cover a plethora of unusual weights, and once again, we have to ponder the following question: is $\widehat A_p$ enough?

Let us point out that the argument we used in the proof of Lemma~\ref{wweights1} gives us that for $p>1$, and $w\in \widehat A_p(\mu)$,
$$
[w]_{A_p^{\mathcal R}(\mu)} \leq \sup_{Q} \left(\frac{1}{\mu(Q)}\int_Q w d\mu \right) ^{1/p} \left(\frac{\Vert \chi_Q w^{1-p'} \Vert_{L^{1,\infty}( \mu)}}{\mu(Q)}\right)^{1/p'} < \infty.
$$
Does this last condition characterize $A_p^{\mathcal R}(\mu)$?

\subsection{One-variable extrapolation remastered}\hfill\vspace{2.5mm}\label{ss82}

We are now ready to develop weak-type analogs of \cite[Theorem 3.1]{multiap1}, starting with downwards extrapolation. 

\begin{theorem}\label{APoffdown}
Fix $0 \leq \alpha<\infty$, and let $\nu$ be a positive, measurable function such that for $\mu=\nu^{\frac{\alpha}{1+\alpha}}$, $\mu(x)dx$ is 
a locally finite, doubling Borel measure on $\mathbb R^n$, with doubling constant $1\leq C_{\mu}< \infty$. Given measurable functions $f$ and $g$, suppose that for some exponent $1< p<\infty$, and every positive, measurable function $v$ such that $V \mu^{-1} \in A_{(1+\alpha)p_{\alpha}}(\mu)$,
\begin{equation}\label{eqAPoffdownh}
\Vert g \Vert_{L^{p_\alpha,\infty}(V)} \leq \psi([V \mu^{-1}]_{ A_{(1+\alpha)p_{\alpha}}(\mu)}) \Vert f \Vert_{L^{p}(v)},
\end{equation}
where $\frac{1}{p_\alpha} = \frac{1}{p} + \alpha$, $V = v^{p_{\alpha} / p} \nu ^{\alpha p_\alpha}$, and $\psi:[1,\infty) \longrightarrow [0,\infty)$ is an increasing function. Then, for every exponent $1< q \leq p$, and every positive, measurable function $w$ such that $W \mu^{-1} \in A_{(1+\alpha)q_{\alpha}}(\mu)$,
\begin{equation}\label{eqAPoffdownc}
\Vert g \Vert_{L^{q_\alpha,\infty}(W)} \leq \Psi ([W \mu^{-1}]_{ A_{(1+\alpha)q_{\alpha}}(\mu)}) \Vert f \Vert_{L^{q}(w)},
\end{equation}
where $\frac{1}{q_\alpha} = \frac{1}{q} + \alpha$, $W = w^{q_{\alpha} / q} \nu ^{\alpha q_\alpha}$, and $\Psi :[1,\infty) \longrightarrow [0,\infty)$ is an increasing function.
\end{theorem}

\begin{proof}
We will prove this statement adapting the proof of Theorem~\ref{weakoffdown}. If $q=p$, then there is nothing to prove, so we may assume that $q<p$. Pick $w$ such that $W \mu^{-1} \in A_{(1+\alpha)q_{\alpha}}(\mu)$. We may also assume that $\Vert f \Vert_{L^{q}(w)} < \infty$. Fix $y>0$ and $\gamma >0$. We have that
\begin{align}\label{eqAPoffdown1}
\begin{split}
\lambda_g ^{W} (y) & =\int_{\{|g|>y\}}W \leq \lambda ^{W} _{\mathscr Z} (\gamma y) + \int_{\{|g|>y\}} \left( \frac{\gamma y}{\mathscr Z} \right)^{p_\alpha - q_\alpha}W \defper I + II,
\end{split}
\end{align}
where 
\begin{equation}\label{eqAPoffdown6}
\mathscr Z \perdef \left(\frac{v^{p_{\alpha}/p}\nu^{\alpha p_{\alpha}}}{W}\right)^{\frac{1}{q_{\alpha}-p_{\alpha}}} \quad \text{ and } \quad v \perdef \left(M_{\mu}(|f|^{\frac{1+\alpha q}{1+\alpha}}w^{\frac{\alpha}{1+\alpha}}\mu^{-1})\mu \right)^{\frac{(q-p)(1+\alpha)}{1+\alpha q}}w^{\frac{1+\alpha p}{1+\alpha q}}.\end{equation}

For the term $I$ in \eqref{eqAPoffdown1}, we obtain that
\begin{align}\label{eqAPoffdown2}
\begin{split}
    I  & \leq \frac{\Vert \mathscr Z \Vert_{L^{q_{\alpha},\infty}(W)}^{q_{\alpha}}}{(\gamma y)^{q_{\alpha}}}   = \frac{1}{(\gamma y)^{q_{\alpha}}} \left \Vert M_{\mu} (|f|^{\frac{1+\alpha q}{1+\alpha}}w^{\frac{\alpha}{1+\alpha}}\mu^{-1}) \right \Vert_{L^{(1+\alpha)q_{\alpha},\infty}(W \mu^{-1} d\mu )}^{(1+\alpha)q_{\alpha}} \\ &  \leq  \frac{10^{(1+\alpha)q_{\alpha}\log_2 C_{\mu}} }{(\gamma y)^{q_{\alpha}}} [W \mu^{-1}]_{ A_{(1+\alpha)q_{\alpha}}(\mu)} \left \Vert f \right \Vert_{L^{q}(w )}^{q},
    \end{split}
\end{align}
where in the last inequality we have used the weak-type bound for the Hardy-Littlewood maximal operator $M_{\mu}$ in \cite[(3.11)]{hpr}.

For the term $II$ in \eqref{eqAPoffdown1}, we argue as follows: since $W \mu^{-1} \in A_{(1+\alpha)q_{\alpha}}(\mu)$, we can find functions $\varpi_0,\varpi_1 \in A_1(\mu)$ such that $W \mu^{-1}=\varpi_0^{\frac{1-q}{1+\alpha q}}\varpi_1$, with
$$
[\varpi_0]_{A_1(\mu)} \leq (\kappa C_{\mu}^{\kappa})^{(1+\alpha)q'}  [W \mu^{-1}]_{A_{(1+\alpha)q_{\alpha}}(\mu)}^{\frac{1+\alpha q}{q-1}}, \, \, [\varpi_1]_{A_1(\mu)} \leq (\kappa C_{\mu}^{\kappa})^{(1+\alpha)q_{\alpha}}  [W \mu^{-1}]_{A_{(1+\alpha)q_{\alpha}}(\mu)},
$$
and $\kappa > 1$ universal. The construction of such functions combines the argument in the proof of \cite[Lemma 3.18]{carlos} and the Buckley-type bounds for $M_{\mu}$ in \cite[Theorem 1.3]{hpr}. In virtue of Lemma~\ref{pesa1} and \cite[Theorem 4.2]{david}, if we write 
\begin{align*}
V \mu^{-1} & = M_{\mu} (|f|^{\frac{1+\alpha q}{1+\alpha}}w^{\frac{\alpha}{1+\alpha}}\mu^{-1})^{\frac{(q-p)(1+\alpha)}{(1+\alpha p)(1+\alpha q)}} W \mu^{-1} \\ & = \left(M_{\mu} (|f|^{\frac{1+\alpha q}{1+\alpha}}w^{\frac{\alpha}{1+\alpha}}\mu^{-1})^{\frac{(p-q)(1+\alpha)}{(p-1)(1+\alpha q)}} \varpi_0^{\frac{(q-1)(1+\alpha p)}{(p-1)(1+\alpha q)}}\right) ^{\frac{1-p}{1+\alpha p}}\varpi_1,
\end{align*}
then we see that $V \mu^{-1} \in A_{(1+\alpha)p_{\alpha}}(\mu)$, with 
\begin{align*}
   [V \mu^{-1}]_{ A_{(1+\alpha)p_{\alpha}}(\mu)} & \leq \left(\kappa C_{\mu}^{\kappa} \frac{(p-1)(1+\alpha q)}{(q-1)(1+\alpha p)} [\varpi_0]_{A_1(\mu)}\right)^{\frac{p-1}{1+\alpha p}}[\varpi_{1}]_{A_1(\mu)} \\ & \leq \mathfrak C_0 [W \mu^{-1}]_{A_{(1+\alpha)q_{\alpha}}(\mu)}^{1+\frac{(p-1)(1+\alpha q)}{(q-1)(1+\alpha p)}},
\end{align*}
where
$$
\mathfrak C_0 \perdef (\kappa C_{\mu}^{\kappa})^{(1+\alpha)q_{\alpha}} \left((\kappa C_{\mu}^{\kappa})^{1+(1+\alpha)q'}  \frac{(p-1)(1+\alpha q)}{(q-1)(1+\alpha p)} \right)^{\frac{p-1}{1+\alpha p}}  ,
$$
so by \eqref{eqAPoffdownh}, we get that
\begin{align}\label{eqAPoffdown4}
\begin{split}
II & \leq \frac{\gamma ^{p_{\alpha}}}{(\gamma y)^{q_{\alpha}}} \Vert g \Vert_{L^{p_{\alpha},\infty}(v^{p_{\alpha}/p}\nu^{\alpha p_{\alpha}})}^{p_{\alpha}} \leq \frac{\gamma ^{p_{\alpha}}}{(\gamma y)^{q_{\alpha}}} \psi (\mathfrak C_0 [W \mu^{-1}]_{A_{(1+\alpha)q_{\alpha}}(\mu)}^{1+\frac{(p-1)(1+\alpha q)}{(q-1)(1+\alpha p)}}) ^{p_{\alpha}} \left \| f \right \|_{L^{p}(v)}^{p_{\alpha}},
\end{split}
\end{align}
and
\begin{align}\label{eqAPoffdown5}
\begin{split}
    \Vert f \Vert_{L^{p}(v)}= \left(\int_{\mathbb R^n} |f|^{p} \left(M_{\mu}(|f|^{\frac{1+\alpha q}{1+\alpha}}w^{\frac{\alpha}{1+\alpha}}\mu^{-1})\mu \right)^{\frac{(q-p)(1+\alpha)}{1+\alpha q}}w^{\frac{1+\alpha p}{1+\alpha q}}\right)^{1/p} \leq \Vert f \Vert_{L^{q}(w)}^{q/p}.
\end{split}
\end{align}

Finally, if we argue as in the 
proof of Theorem~\ref{offdown}, we can combine \eqref{eqAPoffdown1}, \eqref{eqAPoffdown2}, \eqref{eqAPoffdown4}, and  \eqref{eqAPoffdown5} to conclude that \eqref{eqAPoffdownc} holds, with
\begin{equation*}
\Psi (\xi) = \mathfrak C_1 \xi^{\frac{1}{q}-\frac{1}{p}}\psi (\mathfrak C_0 \xi^{1+\frac{(p-1)(1+\alpha q)}{(q-1)(1+\alpha p)}}), \quad \xi\geq 1,
\end{equation*}
where 
\begin{equation*}
   \mathfrak C_1 \perdef \left(\frac{p_{\alpha}}{p_{\alpha}-q_{\alpha}}\right)^{1/q_{\alpha}}\left( \frac{p_{\alpha}-q_{\alpha}}{q_{\alpha}}\right)^{1/p_{\alpha}}  C_{\mu}^{4(1+\alpha)q_{\alpha}\left(\frac{1}{q}-\frac{1}{p}\right)}.
\end{equation*}
\end{proof}

\begin{remark}\label{rmkAprubiodefrancia}
As in Remark~\ref{rmkrubiodefrancia}, for $q>1$, we can take 
\begin{equation*}
    v \perdef \left(\mathscr R_{\mu}(|f|^{\frac{1+\alpha q}{1+\alpha}}w^{\frac{\alpha}{1+\alpha}}\mu^{-1})\mu \right)^{\frac{(q-p)(1+\alpha)}{1+\alpha q}}w^{\frac{1+\alpha p}{1+\alpha q}},
\end{equation*}
where for a measurable function $h\in L^{(1+\alpha)q_{\alpha}}(W \mu^{-1} d\mu)$,
\begin{equation*}
    \mathscr R_{\mu} h \perdef \sum_{k=0}^{\infty} \frac{M_{\mu}^k (|h|)}{2^k \Vert M_{\mu} \Vert_{L^{(1+\alpha)q_{\alpha}}(W \mu^{-1} d\mu)}^k},
\end{equation*}
and rewrite the proof of Theorem~\ref{APoffdown} to conclude that \eqref{eqAPoffdownc} holds, with
\begin{equation*}
\Psi (\xi) = 2^{(1+\alpha)q_{\alpha} \left(\frac{1}{q}-\frac{1}{p}\right)} \mathfrak C \psi (\widetilde{\mathfrak C}_0 \xi^{\frac{(p-1)(1+\alpha q)}{(q-1)(1+\alpha p)} }), \quad \xi\geq 1,
\end{equation*}
where $\mathfrak C$ is as in \eqref{constantc}, and
$
\widetilde{\mathfrak C}_0 \perdef (\kappa C_{\mu}^{\kappa} (1+\alpha)q')^{(1+\alpha)(p_{\alpha}-q_{\alpha})}.
$
\end{remark}

The proof of Theorem~\ref{APoffdown} extends to the case $q=1$ almost verbatim, taking proper care of the terms with divisions by $q-1$.

\begin{theorem}\label{A1offdown}
Fix $0 \leq \alpha<\infty$, and let $\nu$ be a positive, measurable function such that for $\mu=\nu^{\frac{\alpha}{1+\alpha}}$, $\mu(x)dx$ is 
a locally finite, doubling Borel measure on $\mathbb R^n$, with doubling constant $1\leq C_{\mu}< \infty$. Given measurable functions $f$ and $g$, suppose that for some exponent $1< p<\infty$, and every positive, measurable function $v$ such that $V \mu^{-1} \in \widehat A_{(1+\alpha)p_{\alpha}}(\mu)$,
\begin{equation*}
\Vert g \Vert_{L^{p_\alpha,\infty}(V)} \leq \psi(\Vert V \mu^{-1} \Vert_{ \widehat A_{(1+\alpha)p_{\alpha}}(\mu)}) \Vert f \Vert_{L^{p,1}(v)},
\end{equation*}
where $\frac{1}{p_\alpha} = \frac{1}{p} + \alpha$, $V = v^{p_{\alpha} / p} \nu ^{\alpha p_\alpha}$, and $\psi:[1,\infty) \longrightarrow [0,\infty)$ is an increasing function. Then, for every positive, measurable function $w$ such that $W \mu^{-1} \in A_{1}(\mu)$,
\begin{equation*}
\Vert g \Vert_{L^{\frac{1}{1+\alpha},\infty}(W)} \leq p^{3-\frac{1}{p}-p}\mathfrak C_1 [W \mu^{-1}]_{ A_{1}(\mu)}^{1/p'} \psi ([W \mu^{-1}]_{ A_{1}(\mu)}^{\frac{1}{(1+\alpha)p_{\alpha}}}) \Vert f \Vert_{L^{1,\frac{1}{p}}(w)},
\end{equation*}
where $W = w^{\frac{1}{1+\alpha} } \mu$.
\end{theorem}

\begin{remark}
    Note that for $\alpha = 0$ and $\nu = 1$, we recover \cite[Theorem 2.11]{cgs}.
\end{remark}

As usual, we will now focus on the upwards extrapolation.

\begin{theorem}\label{APoffup}
Fix $0 \leq \alpha<\infty$, and let $\nu$ be a positive, measurable function such that for $\mu=\nu^{\frac{\alpha}{1+\alpha}}$, $\mu(x)dx$ is 
a locally finite, doubling Borel measure on $\mathbb R^n$, with doubling constant $1\leq C_{\mu}< \infty$. Given measurable functions $f$ and $g$, suppose that for some exponent $1 \leq p<\infty$, and every positive, measurable function $v$ such that $V \mu^{-1} \in A_{(1+\alpha)p_{\alpha}}(\mu)$,
\begin{equation}\label{eqAPoffuph}
\Vert g \Vert_{L^{p_\alpha,\infty}(V)} \leq \psi([V \mu^{-1}]_{ A_{(1+\alpha)p_{\alpha}}(\mu)}) \Vert f \Vert_{L^{p}(v)},
\end{equation}
where $\frac{1}{p_\alpha} = \frac{1}{p} + \alpha$, $V = v^{p_{\alpha} / p} \nu ^{\alpha p_\alpha}$, and $\psi:[1,\infty) \longrightarrow [0,\infty)$ is an increasing function. Then, for every finite exponent $q \geq p$, and every positive, measurable function $w$ such that $W \mu^{-1} \in A_{(1+\alpha)q_{\alpha}}(\mu)$,
\begin{equation}\label{eqAPoffupc}
\Vert g \Vert_{L^{q_\alpha,\infty}(W)} \leq \Psi ([W \mu^{-1}]_{ A_{(1+\alpha)q_{\alpha}}(\mu)}) \Vert f \Vert_{L^{q,p}(w)},
\end{equation}
where $\frac{1}{q_\alpha} = \frac{1}{q} + \alpha$, $W = w^{q_{\alpha} / q} \nu ^{\alpha q_\alpha}$, and $\Psi :[1,\infty) \longrightarrow [0,\infty)$ is an increasing function.
\end{theorem}

\begin{proof}
We will adapt the proofs of Theorems \ref{offup} and \ref{weakoffup}. If $q=p$, then there is nothing to prove, so we may assume that $q>p$. 

Pick $w$ such that $W \mu^{-1} \in A_{(1+\alpha)q_{\alpha}}(\mu)$. By duality, $w^{\frac{1}{1-q}}\mu^{-1} \in A_{(1+\alpha)q'}(\mu)$, with $[w^{\frac{1}{1-q}}\mu^{-1}]_{A_{(1+\alpha)q'}(\mu)}= [W \mu^{-1}]_{A_{(1+\alpha)q_{\alpha}}(\mu)}^{\frac{1+\alpha q}{q-1}}$, and it follows from \cite[Theorem 1.2]{hpr} that 
$w^{\frac{1}{1-q}}\mu^{-1} \in A_{(1+\alpha)q'-\varepsilon}(\mu)$, with $\varepsilon \perdef \frac{1+\alpha q}{\kappa C_{\mu}^{\kappa}(q-1)}[W \mu^{-1}]_{A_{(1+\alpha)q_{\alpha}}(\mu)}^{-1}$, and
$$
[w^{\frac{1}{1-q}}\mu^{-1}]_{A_{(1+\alpha)q'-\varepsilon}(\mu)}\leq 2^{\frac{1+\alpha q}{q-1}}  C_{\mu}^{2(1+\alpha)q'} [W \mu^{-1}]_{A_{(1+\alpha)q_{\alpha}}(\mu)}^{\frac{1+\alpha q}{q-1}}.
$$

Note that $W\in L^1_{loc}(\mathbb R^n)$. Let $g_{\varrho} \perdef |g|\chi_{B(0,\varrho)}$, with $\varrho \geq 1$, and $y>0$ such that $\lambda_{g_{\varrho}}^{W}(y)\neq 0$. If we take $\beta \perdef 1+\alpha -\frac{\varepsilon}{q'}$, and
\begin{equation}\label{eqAPoffup2}
    v \perdef \left \{ \begin{array}{lr} 
   M_{\mu}(w^{\frac{1}{q \beta}}W^{\frac{1}{q' \beta}} \chi_{\{|g_{\varrho}|>y\}})^{\beta}, 
    & p = 1,  \\
    & \\
     M_{\mu}(w^{\frac{1}{q(1+\alpha) }} W^{\frac{1}{q'(1+\alpha)}}\chi_{\{|g_{\varrho}|>y\}})^{\frac{(q-p)(1+\alpha)}{q-1}}w^{\frac{p-1}{q-1}},
     & p > 1,
\end{array} \right.
\end{equation}
then we can write 
\begin{equation*}
   V \mu^{-1} = \left \{ \begin{array}{lr} 
   M_{\mu}(w^{\frac{1}{q \beta}}W^{\frac{1}{q' \beta}}\chi_{\{|g_{\varrho}|>y\}})^{\frac{\beta}{1+\alpha}} , 
    & p = 1,  \\
    & \\
      M_{\mu}(w^{\frac{1}{q(1+\alpha) }} W^{\frac{1}{q'(1+\alpha)}}\chi_{\{|g_{\varrho}|>y\}})^{\frac{(q-p)(1+\alpha)}{(q-1)(1+\alpha p)}} (W\mu^{-1})^{\frac{(p-1)(1+\alpha q)}{(q-1)(1+\alpha p)}},
     & p > 1.
\end{array} \right.
\end{equation*}

Choose $\varpi_0,\varpi_1 \in A_1(\mu)$ such that $W \mu^{-1}=\varpi_0^{\frac{1-q}{1+\alpha q}}\varpi_1$, with
$$
[\varpi_0]_{A_1(\mu)} \leq (\kappa C_{\mu}^{\kappa})^{(1+\alpha)q'}  [W \mu^{-1}]_{A_{(1+\alpha)q_{\alpha}}(\mu)}^{\frac{1+\alpha q}{q-1}}, \, \, [\varpi_1]_{A_1(\mu)} \leq (\kappa C_{\mu}^{\kappa})^{(1+\alpha)q_{\alpha}}  [W \mu^{-1}]_{A_{(1+\alpha)q_{\alpha}}(\mu)},
$$
and $\kappa > 1$ universal. In virtue of Lemma~\ref{pesa1} and \cite[Theorem 4.2]{david}, we see that $V \mu^{-1} \in A_{(1+\alpha)p_{\alpha}}(\mu)$, with 
\begin{align*}
   [V \mu^{-1}]_{ A_{(1+\alpha)p_{\alpha}}(\mu)} & \leq \mathfrak C_0 [W \mu^{-1}]_{A_{(1+\alpha)q_{\alpha}}(\mu)}^{1+\frac{(p-1)(1+\alpha q)}{(q-1)(1+\alpha p)}},
\end{align*}
and
\begin{equation*}
    \mathfrak C_0 \perdef \left \{ \begin{array}{lr} 
    \kappa C_{\mu}^{\kappa} (1+\alpha) q_{\alpha}, 
    & p = 1,  \\
    & \\
     (\kappa C_{\mu}^{\kappa})^{1+(1+\alpha)\left(q_{\alpha} + \frac{ q'(p-1)}{1+\alpha p} \right)} \frac{(q-1)(1+\alpha p)}{(p-1)(1+\alpha q)}, & p > 1.
\end{array} \right.
\end{equation*}

Observe that $V\geq  W\chi_{\{|g_{\varrho}|>y\}}$, and we can apply \eqref{eqAPoffuph} to control $\lambda_{g_{\varrho}}^W(y)$, as we did in \eqref{eqoffup2}. Now, for $p>1$, 
\begin{align*}
 v(\{|f|>t\}) & \leq w(\{|f|>t\})^{p/q} \left \| \frac{M_{\mu}(w^{\frac{1}{q(1+\alpha) }} W^{\frac{1}{q'(1+\alpha)}}\chi_{\{|g_{\varrho}|>y\}})}{w^{\frac{1}{1+\alpha}}} \right \|_{L^{(1+\alpha)q'}(w)}^{\frac{(q-p)(1+\alpha)}{q-1}},
\end{align*}
and for $p = 1$,
\begin{align*}
 v(\{|f|>t\}) & \leq w(\{|f|>t\})^{1/q} \left \| \frac{M_{\mu}(w^{\frac{1}{q \beta}}W^{\frac{1}{q' \beta}} \chi_{\{|g_{\varrho}|>y\}})}{w^{1/\beta}} \right \|_{L^{\beta q'}(w)}^\beta.
\end{align*}

By the Buckley-type bounds for $M_{\mu}$ in \cite[Theorem 1.3]{hpr},
\begin{align*}
\begin{split}
    \left \| M_{\mu}(w^{\frac{1}{q(1+\alpha) }} W^{\frac{1}{q'(1+\alpha)}}\chi_{\{|g_{\varrho}|>y\}}) \right \|_{L^{(1+\alpha)q'}(w^{\frac{1}{1-q}}\mu^{-1} d\mu)} & \leq \kappa C_{\mu}^{\kappa} (1+\alpha)q_{\alpha}[W \mu^{-1}]_{A_{(1+\alpha)q_{\alpha}}(\mu)} \\ & \times W(\{|g_{\varrho}|>y\})^{\frac{1}{(1+\alpha)q'}},
    \end{split}
\end{align*}
and 
\begin{align*}
    \left \| M_{\mu}(w^{\frac{1}{q \beta}}W^{\frac{1}{q' \beta}} \chi_{\{|g_{\varrho}|>y\}}) \right \|_{L^{\beta q'}(w^{\frac{1}{1-q}}\mu^{-1} d\mu)} & \leq \kappa C_{\mu}^{\kappa} (\beta q')' [w^{\frac{1}{1-q}}\mu^{-1}]_{A_{(1+\alpha)q'-\varepsilon}(\mu)}^{\frac{1}{\beta q' -1}}
\\ & \times W(\{|g_{\varrho}|>y\})^{\frac{1}{\beta q'}}.
\end{align*}

Note that $\frac{1}{q'} < \beta < 1+\alpha$, and
$$\frac{1}{\beta q' -1} = \frac{q-1}{1+\alpha q} \cdot \frac{ \kappa C_{\mu}^{\kappa} [W \mu^{-1}]_{A_{(1+\alpha)q_{\alpha}}(\mu)}}{\kappa C_{\mu}^{\kappa}[W \mu^{-1}]_{A_{(1+\alpha)q_{\alpha}}(\mu)}  -1} \leq \frac{(q-1)\kappa '}{1+\alpha q}.$$
Moreover,
$$\frac{\beta q' }{\beta q' -1} = \frac{(1+\alpha)q_{\alpha} \kappa C_{\mu}^{\kappa} [W \mu^{-1}]_{A_{(1+\alpha)q_{\alpha}}(\mu)} -1}{\kappa C_{\mu}^{\kappa}[W \mu^{-1}]_{A_{(1+\alpha)q_{\alpha}}(\mu)}  -1} \leq \frac{(1+\alpha)q_{\alpha} \kappa  -1}{\kappa  -1},$$
so $(1+\alpha)q_{\alpha} \leq (\beta q')'\leq (1+\alpha)q_{\alpha} \kappa '$. 

Combining the previous estimates, we get that 
\begin{align*}
\begin{split}
    v (\{|f|>t\}) &\leq  
 \phi([W \mu^{-1}]_{A_{(1+\alpha)q_{\alpha}}(\mu)}) W(\{|g_{\varrho}|>y\})^{1-\frac{p}{q}} w(\{|f|>t\})^{p/q},
    \end{split}
\end{align*}
and 
\begin{align*}
\begin{split}
   \left \| f\right \|_{L^{p}(v)} & \leq \left(\frac{p}{q} \phi([W \mu^{-1}]_{A_{(1+\alpha)q_{\alpha}}(\mu)}) \right)^{1/p} 
W(\{|g_{\varrho}|>y\})^{\frac{1}{p}-\frac{1}{q}}   \left \| f \right \|_{L^{q,p}(w)},
   \end{split}
\end{align*}
with
\begin{equation*}
   \phi(\xi) \perdef \left \{ \begin{array}{lr} \left(\kappa \kappa ' C_{\mu}^{\kappa}  (1+\alpha)q_{\alpha} \right)^{1+\alpha}\left(2  C_{\mu}^{2(1+\alpha)q_{\alpha}} \xi\right)^{(1+\alpha) \kappa'}, 
    & p = 1,  \\ & \\
   \left(\kappa C_{\mu}^{\kappa} (1+\alpha)q_{\alpha}\xi\right)^{\frac{(q-p)(1+\alpha)}{q-1}} , & p > 1,
\end{array} \right. \quad \xi \geq 1.
\end{equation*}

Finally, if we follow the proof of Theorem~\ref{offup} performing the previous changes, we conclude that \eqref{eqAPoffupc} holds, with
\begin{align*}
   \Psi( \xi) & =  \left(\frac{p}{q} \phi(\xi) \right)^{1/p} \psi ( \mathfrak C_0 \xi ^{1+\frac{(p-1)(1+\alpha q)}{(q-1)(1+\alpha p)}}), \quad \xi \geq 1.
\end{align*}
\end{proof}

\begin{remark}\label{rmkAprdfup}
As in Remark~\ref{rmkrdfup}, if $1\leq p <q < \infty$, then we can take 
\begin{equation*}
    v \perdef  \mathscr R_{\mu}' (w^{\frac{1}{q(1+\alpha) }} W^{\frac{1}{q'(1+\alpha)}}\chi_{\{|g_{\varrho}|>y\}})^{\frac{(q-p)(1+\alpha)}{q-1}}w^{\frac{p-1}{q-1}},
\end{equation*}
where for a measurable function $h\in L^{(1+\alpha)q'}(w^{\frac{1}{1-q}}\mu^{-1}d \mu)$,
\begin{equation*}
    \mathscr R_{\mu} ' h \perdef \sum_{k=0}^{\infty} \frac{M_{\mu} ^k (|h|)}{2^k\Vert M_{\mu} \Vert_{L^{(1+\alpha)q'}(w^{\frac{1}{1-q}}\mu^{-1}d \mu)}^k},
\end{equation*}
and argue as before to conclude that \eqref{eqAPoffupc} holds, with
\begin{align*}
   \Psi( \xi) & =  \left(\frac{p}{q}  \right)^{1/p} 2^{(1+\alpha) q' \left(\frac{1}{p}-\frac{1}{q}\right)} \psi (( (1+\alpha)q_{\alpha}\kappa C_{\mu}^{\kappa})^{\frac{(q-p)(1+\alpha)}{(q-1)(1+\alpha p)}}\xi), \quad \xi \geq 1.
\end{align*}
\end{remark}

At this point, we suspect that, to deduce restricted weak-type analogs of Theorems \ref{APoffdown} and \ref{APoffup}, we should examine the following conjecture and a convenient dual version yet to be determined.

\begin{conjecture}\label{APsawyer}
Fix $0 \leq \alpha<\infty$, and let $\nu$ be a positive, measurable function such that for $\mu=\nu^{\frac{\alpha}{1+\alpha}}$, $\mu(x)dx$ is 
a locally finite, doubling Borel measure on $\mathbb R^n$, with doubling constant $1\leq C_{\mu}< \infty$. Fix an exponent $q \geq 1$, and write $\frac{1}{q_\alpha} = \frac{1}{q} + \alpha$. Let $w$ be a positive, measurable function such that $W \mu^{-1} \in A_{(1+\alpha)q_{\alpha}}^{\mathcal R}(\mu)$, where $W = w^{q_{\alpha} / q} \nu ^{\alpha q_\alpha}$. Then, there exists a function $\phi :[1,\infty)^2 \longrightarrow [0,\infty)$, increasing in each variable, such that for every measurable function $f$, 
\begin{equation*}
\left \Vert M_{\mu} (|f|^{\frac{1+\alpha q}{1+\alpha}}w^{\frac{\alpha}{1+\alpha}}\mu^{-1}) \right \Vert_{L^{(1+\alpha)q_{\alpha},\infty}(W )}^{(1+\alpha)q_{\alpha}}  \leq  \phi(C_{\mu},[W \mu^{-1}]_{ A_{(1+\alpha)q_{\alpha}}^{\mathcal R}(\mu)}) 
\left \Vert f \right \Vert_{L^{q,1}(w)}^{q}.
\end{equation*}
Equivalently,
\begin{equation*}
\left \Vert \frac{M_{\mu} f}{v} \right \Vert_{L^{(1+\alpha)q_{\alpha},\infty}_v(w v^{(1+\alpha)q_{\alpha}})}  \leq  \phi(C_{\mu},[W \mu^{-1}]_{ A_{(1+\alpha)q_{\alpha}}^{\mathcal R}(\mu)}) ^{\frac{1}{(1+\alpha)q_{\alpha}}}
\left \Vert f \right \Vert_{L^{(1+\alpha)q_{\alpha},\frac{1+\alpha}{1+\alpha q}}_v(w)},
\end{equation*}
where $v=\left(\frac{\nu}{w} \right)^{\frac{\alpha}{1+\alpha}}$, and for $1\leq r < \infty$ and $0<s \leq \infty$, $L^{r,s}_v(\varpi)$ is the weighted Lorentz space given by the quasi-norm $\Vert f \Vert_{L^{r,s}_v(\varpi)} \perdef \Vert f v \Vert_{L^{r,s}(\varpi)}$.
\end{conjecture}

\begin{remark}
    Further hypotheses on $\nu$ and $w$ may be required (see Subsection~\ref{ss83}). In particular, $\mu, w^{\varepsilon} \in A_{\infty}$ for some $0<\varepsilon \leq 1$.
\end{remark}

    Observe that 
    $L^{p,p}_{w^{1/p}}(\mathbb R^n) = L^{p}(w)$, but this relation generally fails for arbitrary exponents $r$ and $s$. We wonder what would happen if we were to replace the spaces $L^{p,1}(w)$ and $L^{p,\infty}(w)$ with $L^{p,1}_{w^{1/p}}(\mathbb R^n)$ and $L^{p,\infty}_{w^{1/p}}(\mathbb R^n)$ in the definition of restricted weak-type, and how this would affect the characterizations of $A_p^{\mathcal R}$ and the corres\-ponding extrapolation schemes. Similar weak-type questions were studied in \cite{muwhe}.

\subsection{\texorpdfstring{Extensions of $A_{\vec{P}}$ and extrapolation}{Extensions of AP and extrapolation}}\label{ss83}\hfill\vspace{2.5mm}

Given exponents $1\leq p_1,\dots,p_m<\infty$, with $\vec P =(p_1,\dots,p_m)$, write
$$\alpha = \frac{1}{p_2}+\dots+\frac{1}{p_m} \quad \text{and} \quad \frac{1}{p} = \frac{1}{p_1}+\alpha.$$
Also, for positive, measurable functions $v_1,\dots,v_m$, with $\vec v = (v_1,\dots,v_m)$, write
$$\nu = \prod_{i=2}^m v_i ^{\frac{1}{\alpha p_i}}, \quad \mu = \nu ^{\frac{\alpha}{1+\alpha}}, \quad \text{and} \quad V = \nu_{\vec v} = v_1^{p/p_1}\nu^{\alpha p}. 
$$

In virtue of \cite[Lemma 3.2]{multiap1} and Definition~\ref{defapsigma}, if $\vec v \in A_{\vec P}$, then for $i=2,\dots,m$, $v_i\in A_{p_i}^{\sigma_i}$, with $\sigma_i = \frac{1}{1+p_i(m-1)}$, $\mu \in A_{\frac{m}{1+\alpha}}$, and $V \mu^{-1} \in A_{(1+\alpha)p}(\mu)$, with
$$[v_i]_{A_{p_i}^{\sigma_i}} \leq [\vec v]_{A_{\vec P}}^{\sigma_i p_i}, \quad [\mu]_{A_{\frac{m}{1+\alpha}}}\leq [\vec v]_{A_{\vec P}}^{\frac{1}{1+\alpha}}, \quad \text{and} \quad [V \mu^{-1}] _{A_{(1+\alpha)p}(\mu)} \leq [\vec v]_{A_{\vec P}}^{p}.
$$
Conversely, if $v_2,\dots,v_m$, and $\mu$ are as above, and $v_1$ is such that $V \mu^{-1} \in A_{(1+\alpha)p}(\mu)$, then $\vec v \in A_{\vec P}$, with
$$
[\vec v]_{A_{\vec P}} \leq [V \mu^{-1}]_{A_{(1+\alpha)p}(\mu)}^{1/p} [\mu]_{A_{\frac{m}{1+\alpha}}}^{1+\alpha} \prod_{i=2}^m [v_i]_{A_{p_i}^{\sigma_i}} ^{\frac{1}{\sigma_i p_i}}.
$$

Note that $\mu$ is a positive, 
locally finite, doubling Borel measure on $\mathbb R^n$, with doubling constant $C_{\mu}\leq 2^{\frac{n m}{1+\alpha}} [\mu]_{A_{\frac{m}{1+\alpha}}}$ (see \cite[Proposition 7.1.5]{grafclas}).

Hence, we can combine Theorems \ref{APoffdown} and \ref{APoffup} as we did in the proof of Theorem~\ref{multiextrapoldowntotal} to establish the weak-type $A_{\vec P}$ extrapolation theory in an alternate manner from \cite{multiap2,multiap1,zoe}, avoiding the use of Rubio de Francia's iteration algorithm as the primary tool in constructing measures.

\begin{theorem}\label{nextmultiextrapoltotal}
Given measurable functions $f_1, \dots, f_m$, and $g$, suppose that for some exponents $1\leq p_1,\dots,p_m < \infty$, $\frac{1}{p}=\frac{1}{p_1}+\dots+\frac{1}{p_m}$, and every $\vec v \in A_{\vec P}$,
\begin{equation*}
  \left \| g \right \|_{L^{p,\infty}(\nu_{\vec v})} \leq \varphi ([\vec v]_{A_{\vec P}}) \prod_{i=1}^m \left \| f_i \right \|_{L^{p_i}(v_i)},
\end{equation*}
where $\varphi:[1,\infty)\longrightarrow [0,\infty)$ is an increasing function.
Then, for all exponents $1< q_{1}, \dots,q_m < \infty$, $\frac{1}{q}=\frac{1}{q_{1}}+\dots+\frac{1}{q_m}$, and every $\vec w \in A_{\vec Q}$, 
\begin{equation*}
        \left \| g \right \|_{L^{q,\infty}(\nu_{\vec w})} \leq \Phi ([\vec w]_{A_{\vec Q}})\prod_{i=1}^{m} \left \| f_i \right \|_{L^{q_i,\min  \{p_i,q_i \}}(w_i)},
\end{equation*}
where $\Phi:[1,\infty)\longrightarrow [0,\infty)$ is an increasing function. If for some $1\leq i \leq m$, $p_i=1$, then we can also take $q_i=1$.
\end{theorem}

The first step towards a restricted weak-type analog of this theorem is to define a proper class of tuples of measures to extrapolate. We introduce some ideas that may not be final but point in the right direction. 

Let us review the structure of $A_{\vec P}$. It follows from \cite[Theorem 3.6]{LOPTT} that $\vec v \in A_{\vec P}$ if, and only if for $i=1,\dots,m$, $v_i\in A_{p_i}^{\sigma_i}$, with $\sigma_i = \frac{1}{1+p_i(m-1)}$,  and $\nu_{\vec v} 
\in A_{mp}$. Moreover,
$$
[\vec v]_{A_{\vec P}} \leq [\nu_{\vec v}]_{A_{mp}}^{1/p}\prod_{i=1}^m [v_i]_{A_{p_i}^{\sigma_i}}^{\frac{1}{\sigma_i p_i}}, \quad [\nu_{\vec v}]_{A_{mp}} \leq [\vec v]_{A_{\vec P}}^p, \quad \text{ and } \quad [v_i]_{A_{p_i}^{\sigma_i}} \leq [\vec v]_{A_{\vec P}}^{\sigma_i p_i}.
$$

This argument motivates the next definition.

\begin{definition}\label{defvecApgorro}
    Fix $1 \leq N_1,\dots,N_m \in \mathbb N \mathcup \{ \infty \}$, and $0< \ell < m$. Given exponents $1\leq p_1,\dots,p_m<\infty$, $\frac{1}{p}= \frac{1}{p_1}+\dots+\frac{1}{p_m}$, and $\sigma_i = \frac{1}{1+p_i(m-1)}$, $i=1,\dots,m$, and positive, measurable functions $w_1,\dots,w_m$, with $\vec w = (w_1,\dots,w_m)$, we say that:
    
    \begin{enumerate}
\item[($a$)] $\vec w \in \widehat A_{\vec P,\vec N}$, with $\vec P= (p_{1},\dots,p_m)$ and $\vec N=(N_1,\dots,N_m)$, if $w_i\in \widehat A_{p_i,N_i}^{\sigma_i}$,  $i=1,\dots,m$, and $\nu_{\vec w}\in \widehat A_{mp,N_1+\dots+N_m}$. If $N_1=\dots=N_m=1$, we simply write $\widehat A_{\vec P}$.

\item[($b$)] $\vec w \in \widehat A_{\vec P,\vec R, \vec N}$, with $\vec R=(p_1,\dots,p_{\ell})$, $\vec N=(N_1,\dots,N_{\ell})$, $\vec P= (p_{\ell+1},\dots,p_m)$, if $w_i\in \widehat A_{p_i,N_i}^{\sigma_i}$,  $i=1,\dots,\ell$, $w_i\in A_{p_i}^{\sigma_i}$, $i=\ell+1,\dots,m$, and $\nu_{\vec w}\in A_{mp}$ (or $\widehat A_{mp,N_1+\dots+N_{\ell}+m-\ell}$).
\end{enumerate}
If $\ell=0$, we take $\widehat A_{\vec P,\vec R, \vec N} \perdef A_{\vec P}$, and if $\ell=m$, we take $\widehat A_{\vec P,\vec R, \vec N} \perdef \widehat A_{\vec R, \vec N}$.
\end{definition}

\begin{remark}
    It is clear that these classes contain $A_{\vec P}$. Moreover, $\prod_{i=1}^m \widehat A_{p_i,N_i} \subseteq \widehat A_{\vec P,\vec N}$ and $\left(\prod_{i=1}^{\ell} \widehat A_{p_i,N_i}\right)\times \left(\prod_{i=\ell+1}^{m} A_{p_i} \right) \subseteq \widehat A_{\vec P,\vec R, \vec N}$, and we want to believe that $\widehat A_{\vec P,\vec N} \subseteq A_{\vec P}^{\mathcal R}$ and $\widehat A_{\vec P,\vec R,\vec N} \subseteq A_{\vec P,\vec R}^{\mathfrak M}$, but proof is required. It would be helpful to produce characterizations of $A_{\vec P}^{\mathcal R}$ and $A_{\vec P,\vec R}^{\mathfrak M}$ resembling \cite[Theorem 3.6]{LOPTT}.
\end{remark}

Note that if we take parameters $-n \left(1+p_i(m-1)\right)<\beta_i \leq n(p_i-1)$, $i=1,\dots,m$, with $-n< \sum_{i=1}^m \frac{p \beta_i }{p_i}\leq n(mp-1)$, then $\vec w = (|x|^{\beta_1},\dots,|x|^{\beta_m})\in \widehat A_{\vec P}$, and if for $i=\ell+1,\dots,m$, we also impose that $\beta_i < n(p_i-1)$, then $\vec w \in \widehat A_{\vec P,\vec R,\vec 1}$ (see \cite[Example 7.1.7]{grafclas}). In particular, if for some $1\leq j \leq m$, $p_j\neq 1$, then for $\beta_j = n(p_j-1)$, $\vec w \not \in A_{\vec P}$ if $1\leq j \leq \ell$, and $\vec w \not \in \widehat A_{\vec P,\vec R,\vec 1}$ if $\ell+1\leq j \leq m$.

It is worth mentioning that, mimicking the case $m=1$, we can identify the following well-behaved subclass of $\widehat A_{\vec P}$. 

\begin{definition}
    Given exponents $1\leq p_1,\dots,p_m<\infty$, $\frac{1}{p}= \frac{1}{p_1}+\dots+\frac{1}{p_m}$, and positive, measurable functions $w_1,\dots,w_m$,  we say that $\vec w = (w_1,\dots,w_m)$ belongs to $\widehat{\text{{\Large \textit{a}}}}_{\vec P}$ if there exist functions $h_1,\dots,h_m \in L^1_{loc}(\mathbb R^n)$, and $\vec u = (u_1,\dots,u_m) \in A_{\vec 1}$ such that $$\vec w = \left((Mh_1)^{1-p_1}u_1^{\frac{p_1}{mp}},\dots,(Mh_m)^{1-p_m}u_m^{\frac{p_m}{mp}}\right).$$
    We associate to this class the constant given by
\begin{equation*}
    \Vert \vec w \Vert_{\widehat{\text{{\large \textit{a}}}}_{\vec P}} \perdef  \inf \, [\vec u]_{A_{\vec 1}},
\end{equation*}
where the infimum is taken over all suitable representations of $\vec w$. 
\end{definition}

\begin{remark}
    In practice, for $\vec w = (w_1,\dots,w_m) \in \widehat A_{\vec P}$, it would be convenient to construct functions $h_1,\dots,h_m \in L^1_{loc}(\mathbb R^n)$, and $\vec u = (u_1,\dots,u_m)$ such that $\vec w = \left((Mh_1)^{1-p_1}u_1,\dots,(Mh_m)^{1-p_m}u_m\right)$, with $u_i^{\frac{1}{1+p_i(m-1)}}\in A_1$, $i=1,\dots,m$, and $u_1^{p/p_1}\dots u_m^{p/p_m}\in A_1$. Is there a similar factorization result for $\vec w \in A_{\vec P}$, replacing $Mh_1,\dots,Mh_m$ by $v_1,\dots,v_m\in A_1$?
\end{remark}

After all we have seen so far, it seems reasonable to consider extrapolation theories for $\widehat A_{\vec P, \vec N}$ and $\widehat A_{\vec P,\vec R, \vec N}$ (or even $\widehat{\text{{\Large \textit{a}}}}_{\vec P}$), but for now, we cannot follow the road map suggested by Theorems \ref{APoffdown} and \ref{APoffup} because, given $(w_1,w_2,\dots,w_m)\in \widehat A_{\vec Q}$, with $\vec Q=(q_1,p_2,\dots,p_m)$, we don't know if for $v_1$ as in \eqref{eqAPoffdown6} or \eqref{eqAPoffup2}, $(v_1,w_2,\dots,w_m)$ is in $\widehat A_{\vec P,\vec N}$, with $\vec P=(p_1,p_2,\dots,p_m)$ and $\vec N=(2,1,\dots,1)$, or $\widehat A_{\vec P}$ if $q_1=1$. 
Luckily, for $\vec Q= \vec 1$, we have that $\mu= w_2^{1/m}\dots w_m^{1/m} \in A_1$, and something can be arranged.

We are now ready to present the first restricted weak-type extrapolation scheme assuming multi-variable conditions on the tuples of measures involved, which is an endpoint result that falls outside the classical multi-variable Rubio de Francia's extrapolation theory in \cite{multiap2,multiap1,zoe}. Remarkably, our favorite Sawyer-type inequality saves the day through Lemma~\ref{wweights2}.

\begin{theorem}\label{A1new}
Fix $m\geq 2$. Given measurable functions $f_1, \dots, f_m$, and $g$, suppose that for some exponents $1< p_1< \infty$, $p_2=\dots=p_m =1$, and 
$p=\frac{p_1}{1+p_1(m-1)}$, and every $m$-tuple of positive, measurable functions $\vec v = (v_1,\dots,v_m)$ such that $v_1^{\frac{1}{1+p_1(m-1)}}\in A_{\frac{m p_1}{1+p_1(m-1)}}^{\mathcal R}$, $v_2^{1/m},\dots,v_m^{1/m}\in A_1$, and $\nu_{\vec v}\in A_{mp}^{\mathcal R}$,
\begin{equation}\label{eqA1newh}
  \left \| g \right \|_{L^{p,\infty}(\nu_{\vec v})} \leq \varphi ( [v_1^{\frac{1}{1+p_1(m-1)}}]_{A_{\frac{m p_1}{1+p_1(m-1)}}^{\mathcal R}},[v_2^{1/m}]_{A_1},\dots,[v_m^{1/m}]_{A_1},[\nu_{\vec v}]_{A_{mp}^{\mathcal R}}) \prod_{i=1}^m \left \| f_i \right \|_{L^{p_i,1}(v_i)},
\end{equation}
where $\varphi:[1,\infty)^{m+1}\longrightarrow [0,\infty)$ is a function increasing in each variable. Then, for every $\vec w =(w_1,\dots,w_m)\in A_{\vec 1}$, 
\begin{equation}\label{eqA1newc}
        \left \| g \right \|_{L^{\frac{1}{m},\infty}(\nu_{\vec w})} \leq \Phi ([\vec w]_{A_{\vec 1}})\prod_{i=1}^{m} \left \| f_i \right \|_{L^{1,\frac{1}{p_i}}(w_i)},
\end{equation}
where $\Phi:[1,\infty)\longrightarrow [0,\infty)$ is an increasing function.
\end{theorem}

\begin{proof}
We will adapt the proof of Theorem~\ref{APoffdown}. Pick $\vec w= (w_1,\dots,w_m) \in A_{\vec 1}$. We know from \cite[Theorem 3.6]{LOPTT} and \cite[Lemma 3.2]{multiap1} that for $i=1,\dots,m$, and $\mu \perdef w_2^{1/m}\dots w_m^{1/m}$,
\begin{equation}\label{eqA1new01}
\max \left \{ [w_i^{1/m}]_{A_{1}} , [\mu]_{A_{1}} , [w_1^{1/m}] _{A_{1}(\mu)} ,[\nu_{\vec w}]_{A_{1}} \right \} \leq [\vec w]_{A_{\vec 1}}^{1/m}.
\end{equation}

Fix $y>0$ and $\gamma >0$. We have that
\begin{align}\label{eqA1new1}
\begin{split}
\lambda_g ^{\nu_{\vec w}} (y) & =\int_{\{|g|>y\}} \nu_{\vec w} \leq \lambda ^{\nu_{\vec w}} _{\mathscr Z} (\gamma y) + \int_{\{|g|>y\}} \left( \frac{\gamma y}{\mathscr Z} \right)^{p - \frac{1}{m}}\nu_{\vec w} \defper I + II,
\end{split}
\end{align}
where $\mathscr Z \perdef M_{\mu}(f_1 w_1^{1-\frac{1}{m}}\mu^{-1})^m$.

To estimate the term $I$ in \eqref{eqA1new1}, in virtue of \eqref{eqA1new01}, \cite[(3.11)]{hpr}, and \cite[Proposition 7.1.5]{grafclas}, we deduce that
\begin{align}
\begin{split}\label{eqA1new2}
    I & \leq \frac{1}{(\gamma y)^{1/m}} \Vert \mathscr Z \Vert_{L^{\frac{1}{m},\infty}(\nu_{\vec w})}^{1/m} = \frac{1}{(\gamma y)^{1/m}} \left \Vert M_{\mu}(f_1 w_1^{1-\frac{1}{m}}\mu^{-1}) \right \Vert_{L^{1,\infty}(w_1^{1/m} d\mu )} \\ & \leq \frac{10^{n+\log_2 [\mu]_{A_1}}}{(\gamma y)^{1/m}}[w_1^{1/m}] _{A_{1}(\mu)} \Vert f_1 \Vert_{L^{1}(w_1)} \leq \frac{p_1^{1-p_1}}{(\gamma y)^{1/m}} 10^{n+\frac{1}{m}\log_2 [\vec w]_{A_{\vec 1}}} [\vec w]_{A_{\vec 1}}^{1/m} \Vert f_1 \Vert_{L^{1,\frac{1}{p_1}}(w_1)}.
\end{split}
\end{align}

To estimate the term $II$ in \eqref{eqA1new1}, take 
\begin{equation}\label{eqA1new4}
    v_1 \perdef \left(M_{\mu}(f_1 w_1^{1-\frac{1}{m}}\mu^{-1}) \mu \right)^{1-p_1} w_1^{\frac{1+p_1(m-1)}{m}},
\end{equation}
and $v_i\perdef w_i$, $i=2,\dots,m$. Since $\frac{1-p_1}{1+p_1(m-1)}=1-\frac{m p_1}{1+p_1(m-1)}$, it follows from \eqref{eqA1new01} and Lemma~\ref{wweights2} that $v_1^{\frac{1}{1+p_1(m-1)}}\in A_{\frac{m p_1}{1+p_1(m-1)}}^{\mathcal R}$,  with $$[v_1^{\frac{1}{1+p_1(m-1)}}]_{A_{\frac{m p_1}{1+p_1(m-1)}}^{\mathcal R}} \leq \phi_n([\vec w]_{A_{\vec 1}}^{1/m})^{\frac{1}{m p_1'}} [\vec w]_{A_{\vec 1}}^{\frac{1}{m^2 p}},
$$
where $\phi_n: [1,\infty) \longrightarrow [0,\infty)$ is an increasing function depending only on the dimension $n$. Moreover, $$\nu_{\vec v} = v_1^{p/p_1}v_2^{p/p_2}\dots v_m^{p/p_m} = v_1^{\frac{1}{1+p_1(m-1)}} \mu^{m p} = M_{\mu}(f_1 w_1^{1-\frac{1}{m}}\mu^{-1})^{1-m p}\mu w_1^{1/m},$$ 
and Lemma~\ref{wweights1} tells us that $\nu_{\vec v} \in A_{mp}^{\mathcal R}$, with
$$
[\nu_{\vec v}]_{A_{m p}^{\mathcal R}} \leq (\mathfrak c_n [\vec w]_{A_{\vec 1}}^{1/m})^{\frac{\kappa}{(mp)'}}[\vec w]_{A_{\vec 1}}^{\frac{1}{m}\left(1+\frac{1}{mp} \right)}=\mathfrak c_n ^{\frac{\kappa}{(mp)'}} [\vec w]_{A_{\vec 1}}^{\frac{1}{m}\left(1+\frac{1}{mp} + \frac{\kappa}{(mp)'} \right)},
$$
and $\kappa >1$ universal.

Hence, by \eqref{eqA1newh} and the monotonicity of $\varphi$, we get that
\begin{align}\label{eqA1new3}
\begin{split}
II & = \frac{(\gamma y)^{p}}{(\gamma y)^{1/m}} \int_{\{|g|>y\}} \nu_{\vec v} \\ & \leq \frac{\gamma ^{p}}{(\gamma y)^{1/m}} \varphi ( \phi_n([\vec w]_{A_{\vec 1}}^{1/m})^{\frac{1}{m p_1'}} [\vec w]_{A_{\vec 1}}^{\frac{1}{m^2 p}},[\vec w]_{A_{\vec 1}}^{1/m},\dots,\mathfrak c_n ^{\frac{\kappa}{(mp)'}} [\vec w]_{A_{\vec 1}}^{\frac{1+\kappa}{m}}) ^p \prod_{i=1}^m \left \| f_i \right \|_{L^{p_i,1}(v_i)}^p,
\end{split}
\end{align}
and $\left \| f_1 \right \|_{L^{p_1,1}(v_1)} \leq p_1 \left \| f_1 \right \|_{L^{1,\frac{1}{p_1}}(w_1)}^{1/p_1}$.

Combining the estimates \eqref{eqA1new1}, \eqref{eqA1new2}, and \eqref{eqA1new3}, we conclude that
\begin{align*}
    \begin{split}
       \lambda_g ^{\nu_{\vec w}} (y) & \leq \frac{p_1^{1-p_1}}{(\gamma y)^{1/m}} 10^{n+\frac{1}{m}\log_2 [\vec w]_{A_{\vec 1}}} [\vec w]_{A_{\vec 1}}^{1/m} \Vert f_1 \Vert_{L^{1,\frac{1}{p_1}}(w_1)} \\ & + \frac{p_1^p \gamma ^{p}}{(\gamma y)^{1/m}} \varphi ( \phi_n([\vec w]_{A_{\vec 1}}^{1/m})^{\frac{1}{m p_1'}} [\vec w]_{A_{\vec 1}}^{\frac{1}{m^2 p}},[\vec w]_{A_{\vec 1}}^{1/m},\dots,\mathfrak c_n ^{\frac{\kappa}{(mp)'}} [\vec w]_{A_{\vec 1}}^{\frac{1+\kappa}{m}}) ^p \prod_{i=1}^m \left \| f_i \right \|_{L^{1,\frac{1}{p_i}}(w_i)}^{p/p_i},
    \end{split}
\end{align*}
and taking the infimum over all $\gamma >0$ (see \cite[Lemma 3.1.1]{thesis}), \eqref{eqA1newc} holds, with
\begin{equation}\label{eqA1new5}
\Phi (\xi) = \mathfrak  C_{p_1}^{m} \left(10^{n+\frac{1}{m}\log_2 \xi} \xi^{1/m} \right)^{1-\frac{1}{p_1}} \varphi ( \phi_n(\xi^{1/m})^{\frac{1}{m p_1'}} \xi^{\frac{1}{m^2 p}},\xi^{1/m},\dots,\mathfrak c_n ^{\frac{\kappa}{(mp)'}} \xi^{\frac{1+\kappa}{m}}), \quad \xi\geq 1,
\end{equation}
where 
$\mathfrak C_{p_1}^{m}\perdef p_1^{3-\frac{1}{p_1}-p_1} (mp)^{m}( mp-1)^{\frac{1}{p_1}-1}$. 
\end{proof}

From Theorem~\ref{A1new}, we can obtain the following extrapolation result for multi-variable operators arguing as in the proof of Corollary~\ref{cextrapolone}. 

\begin{corollary}\label{corolA1new}
Fix $m\geq 2$, and let $T$ be an $m$-variable operator defined for characteristic functions. Suppose that for some exponents $1< p_1< \infty$, $p_2=\dots=p_m =1$, and $p=\frac{p_1}{1+p_1(m-1)}$, every $m$-tuple of positive, measurable functions $\vec v = (v_1,\dots,v_m)$ such that $v_1^{\frac{1}{1+p_1(m-1)}}\in A_{\frac{m p_1}{1+p_1(m-1)}}^{\mathcal R}$, $v_2^{1/m},\dots,v_m^{1/m}\in A_1$, and $\nu_{\vec v}\in A_{mp}^{\mathcal R}$, and all measurable sets $E_1,\dots,E_m \subseteq \mathbb R^n$,
\begin{equation*}
     \left \| T(\chi_{E_1},\dots,\chi_{E_m}) \right \|_{L^{p,\infty}(\nu_{\vec v})}\leq \varphi (v_1,\dots,v_m,\nu_{\vec v}) \prod_{i=1}^m v_i (E_i)^{1/p_i},
\end{equation*}
with $\varphi$ as in \eqref{eqA1newh}. Then, for every $\vec w =(w_1,\dots,w_m) \in A_{\vec 1}$, and all measurable sets $F_1,\dots,F_m \subseteq \mathbb R^n$,
\begin{equation*}
     \left \| T(\chi_{F_1},\dots,\chi_{F_m}) \right \|_{L^{\frac{1}{m},\infty}(\nu_{\vec w})}\leq p_1^{p_1-1} \Phi ( [\vec w]_{A_{\vec 1}}) \prod_{i=1}^m w_i (F_i),
\end{equation*}
with $\Phi$ as in \eqref{eqA1new5}. 
\end{corollary}

    \begin{remark}
        The extension of this last bound to the full weak-type $(1,\dots,1,\frac{1}{m})$ would require a version of Theorem~\ref{epsdelta} for tuples of measures in $A_{\vec 1}$. At the time of writing, such a result is not available.
    \end{remark}

Note that for $\vec w \in A_{\vec 1}$, and $v_1,\dots,v_m$ as in \eqref{eqA1new4}, we have that $\vec v \in A_{\vec P}^{\mathcal R}$, with $\vec P = (p_1,1,\dots,1)$. Indeed, by the argument in the proof of Lemma~\ref{wweights2}, and using \eqref{eqA1new01}, we get that for every cube $Q \subseteq \mathbb R^n$,
\begin{align*}
\frac{\Vert \chi_Q v_1^{-1}\Vert_{L^{p_1',\infty}(v_1)}}{|Q|} & \leq \phi_n([\vec w]_{A_{\vec 1}} ^{1/m})^{1-\frac{1}{p_1}}  \left([\vec w]_{A_{\vec 1}} ^{1/m}\frac{\mu(Q)}{\nu_{\vec w}(Q)}\right)^{\frac{1}{p_1}+m-1}  
\\ & \times \left ( \big(\essinf_{y\in Q} \mu(y) \big) \big(\essinf_{y\in Q} M_{\mu} h(y)\big) \right)^{1-\frac{1}{p_1}} |Q|^{-\frac{1}{p_1}},
\end{align*}
where $h \perdef f_1 w_1^{1-\frac{1}{m}}\mu^{-1}$. Also, from \cite[Lemma 3.2]{multiap1} and \cite[Theorem 3.6]{LOPTT}, we deduce that the $m$-tuple $(1,w_2,\dots,w_m)$ is in $A_{\vec 1}$, with
$$
[(1,w_2,\dots,w_m)]_{A_{\vec 1}} \perdef \sup_Q \left( \frac{\mu(Q)}{|Q|}\right)^m \prod_{i=2}^m \big( \essinf_{y \in Q} w_i(y)\big)^{-1} \leq [\vec w]_{A_{\vec 1}}^m.
$$

Thus,
\begin{align*}
[\vec v]_{A_{\vec P}^{\mathcal R}} & \perdef \sup_{Q} \nu_{\vec v}(Q)^{1/p} \frac{\Vert \chi_Q v_1^{-1}\Vert_{L^{p_1',\infty}(v_1)}}{|Q|}  \prod_{i=2} ^{m}  \frac{\esssup_{y\in Q}w_i(y)^{-1}}{|Q|} \\ & \leq \phi_n([\vec w]_{A_{\vec 1}} ^{1/m})^{1-\frac{1}{p_1}}  \sup_Q   \big(\essinf_{y\in Q} \mu(y) \big)^{1-\frac{1}{p_1}} |Q|^{1-m-\frac{1}{p_1}} \prod_{i=2}^m \big( \essinf_{y \in Q} w_i(y)\big)^{-1} \\ & \times \left([\vec w]_{A_{\vec 1}} ^{1/m}\frac{\mu(Q)}{\nu_{\vec w}(Q)}\int_Q (M_{\mu}h)^{\frac{1-p_1}{1+p_1(m-1)}} \nu_{\vec w} \right)^{\frac{1}{p_1}+m-1} \big(\essinf_{y\in Q} M_{\mu} h(y)\big) ^{1-\frac{1}{p_1}} \\ & \leq [\vec w]_{A_{\vec 1}}^{m+\frac{1}{mp}} \phi_n([\vec w]_{A_{\vec 1}} ^{1/m})^{1-\frac{1}{p_1}}  \sup_Q \big(|Q| \essinf_{y\in Q} \mu(y)\big)^{1-\frac{1}{p_1}} \mu(Q)^{\frac{1}{p_1}-1} \\ & \leq [\vec w]_{A_{\vec 1}}^{m+\frac{1}{mp}} \phi_n([\vec w]_{A_{\vec 1}} ^{1/m})^{1/p_1'}. 
\end{align*}   

Therefore, the extrapolation argument  described in the proof of Theorem~\ref{A1new} and Corollary~\ref{corolA1new} can be applied to classical multi-variable operators such as the multi-sub-linear maximal operator $\mathcal M$, sparse operators like $\mathcal A_{\mathcal S}$, and multi-linear Calder\'on-Zygmund operators, for which we can prove weighted restricted weak-type bounds (see \cite[Theorem 10]{prp}).

   Theorem~\ref{A1new} and Corollary~\ref{corolA1new} are proof of concept; they show that a genuinely multi-variable restricted weak-type extrapolation theory is conceivable, and that the appropriate conditions on the measures involved are close to $v_i^{\frac{1}{1+p_i(m-1)}}\in A_{\frac{mp_i}{1+p_i(m-1)}}^{\mathcal R}$, $i=1,\dots,m$, and $\nu_{\vec v} \in A_{mp}^{\mathcal R}$, which should characterize $A_{\vec P}^{\mathcal R}$. 

   To keep advancing, one may have to develop the theory of $A_p^{\mathcal R}(\mu)$ weights further, generalize Subsection~\ref{weiweights} for $\mu \in A_{\infty}$, prove restricted weak-type versions of \cite[Lemma 3.2]{multiap1} and \cite[Theorem 3.6]{LOPTT}, and make some progress with Conjecture~\ref{APsawyer} and suitable dual forms of it.

In \cite{prp}, we suggested that multi-variable Sawyer-type inequalities for $\mathcal M$ might play a role in $A_{\vec P}^{\mathcal R}$ extrapolation, but for now, this doesn't seem to be the case. Nevertheless, the relation between \cite[Conjecture 1]{prp} and Conjecture~\ref{APsawyer} should be explored.

\section*{Author's note}

The reader may encounter the phantom reference \textit{Endpoint weighted estimates for bi-linear Fourier multipliers via restricted bi-linear extrapolation} in the literature, an unfinished draft containing E. R. P.'s ideas, results, and figure designs from Chapter 3 of his Ph.D. dissertation. Due to the censorship, discrimination, and contempt that E.~R.~P. suffered from the principal investigators of his former research group while preparing such a manuscript, he decided to discontinue the project and not authorize its publication. 
That material should not have been circulated without proper citation; that is, 
\cite{thesis}. The present document supersedes the one above.

Coping with rights violations, negligence, unethical behavior, toxic workplace culture, 
and the challenges posed by COVID-19 resulted in mental and emotional exhaustion. There have been plenty of reasons to let this research rot, yet here it is, fresh as a fish, against all odds. The intervention of two individuals made it possible. The author is indebted to Carlos P\'erez Moreno for his trust, support, and mentorship, especially during the aftermath of the pandemic. The author also wishes to express gratitude to Marta de Le\'on Contreras for providing an opportunity that motivated him enough to jump-start this study again, as well as to the personnel he met at UPV/EHU and BCAM 
for their help and assistance. 

Regarding funding, the Margarita Salas grant was advertised as highly competitive and prestigious but did not comply, forcing part of its beneficiaries to resign or file a lawsuit against the employing universities. By contrast, PGC2018-094522-B-I00 worked like a charm.



\end{document}

%% file: dib2.pdf_tex
\begingroup%
  \makeatletter%
  \providecommand\color[2][]{%
    \errmessage{(Inkscape) Color is used for the text in Inkscape, but the package 'color.sty' is not loaded}%
    \renewcommand\color[2][]{}%
  }%
  \providecommand\transparent[1]{%
    \errmessage{(Inkscape) Transparency is used (non-zero) for the text in Inkscape, but the package 'transparent.sty' is not loaded}%
    \renewcommand\transparent[1]{}%
  }%
  \providecommand\rotatebox[2]{#2}%
  \newcommand*\fsize{\dimexpr\f@size pt\relax}%
  \newcommand*\lineheight[1]{\fontsize{\fsize}{#1\fsize}\selectfont}%
  \ifx\svgwidth\undefined%
    \setlength{\unitlength}{529.64932179bp}%
    \ifx\svgscale\undefined%
      \relax%
    \else%
      \setlength{\unitlength}{\unitlength * \real{\svgscale}}%
    \fi%
  \else%
    \setlength{\unitlength}{\svgwidth}%
  \fi%
  \global\let\svgwidth\undefined%
  \global\let\svgscale\undefined%
  \makeatother%
  \begin{picture}(1,0.7087811)%
    \lineheight{1}%
    \setlength\tabcolsep{0pt}%
    \put(0,0){\includegraphics[width=\unitlength,page=1]{dib2.pdf}}%
    \put(0.58117535,0.41564313){\color[rgb]{0,0,0}\makebox(0,0)[lt]{\lineheight{1.25}\smash{\begin{tabular}[t]{l}$(\xi_1,\xi_2)$\end{tabular}}}}%
    \put(0,0){\includegraphics[width=\unitlength,page=2]{dib2.pdf}}%
    \put(0.15700872,0.1742854){\color[rgb]{0,0,0}\makebox(0,0)[lt]{\lineheight{1.25}\smash{\begin{tabular}[t]{l}$R_{\xi_1,\xi_2}$\end{tabular}}}}%
    \put(0.79014547,0.56924458){\color[rgb]{0,0,0}\makebox(0,0)[lt]{\lineheight{1.25}\smash{\begin{tabular}[t]{l}$E$\end{tabular}}}}%
    \put(0.1605829,0.56900036){\color[rgb]{0,0,0}\makebox(0,0)[lt]{\lineheight{1.25}\smash{\begin{tabular}[t]{l}$\mathbb R^2$\end{tabular}}}}%
    \put(0.72359204,0.13593905){\color[rgb]{0,0,0}\makebox(0,0)[lt]{\lineheight{1.25}\smash{\begin{tabular}[t]{l}$\mathfrak m_{\nu,E}(\xi_1,\xi_2)=\nu(E \cap R_{\xi_1,\xi_2})$\end{tabular}}}}%
    \put(0,0){\includegraphics[width=\unitlength,page=3]{dib2.pdf}}%
  \end{picture}%
\endgroup%